\documentclass{article}
\usepackage[T1]{fontenc}
\usepackage{lmodern}
\usepackage{mathtools, amssymb, amsthm}
\usepackage{enumerate}
\usepackage[utf8]{inputenc}
\usepackage{listings}
\usepackage{graphicx}
\usepackage{thmtools}
\usepackage{biblatex}
\usepackage{tikz-cd}
\usepackage{tikz}
\usepackage{hyperref}
\usepackage{biblatex}
\usepackage{titlesec}
\usepackage[left=2.5cm,right=2.5cm,top=3.75cm,bottom=1.5cm]{geometry}

\usepackage{scalerel}
\usepackage{stmaryrd}
\usepackage{bbm}
 \usepackage{standalone}

\bibstyle{plain}
\theoremstyle{definition}
\newtheorem{definition}{Definition}[section]
\newtheorem{situation}[definition]{Situation}
\newtheorem{question}[definition]{Question}
\newtheorem{construction}[definition]{Construction}
\newtheorem{example}[definition]{Example}
\newtheorem{remark}[definition]{Remark}

\theoremstyle{plain}
\newtheorem{theorem}[definition]{Theorem}
\newtheorem{prop}[definition]{Proposition}
\newtheorem{lemma}[definition]{Lemma}

\newtheorem{corollary}[definition]{Corollary}

\theoremstyle{remark}

\DeclareMathOperator{\id}{id}
\DeclareMathOperator{\Spec}{Spec}

\DeclareMathOperator{\Ext}{Ext}
\DeclareMathOperator{\HH}{HH}
\DeclareMathOperator{\coh}{H}
\DeclareMathOperator{\fib}{fib}
\DeclareMathOperator{\Fun}{Fun}
\DeclareMathOperator{\Map}{Map}

\DeclareMathOperator{\Ch}{Ch}
\DeclareMathOperator{\Sym}{Sym}

\DeclareMathOperator{\ob}{ob}
\DeclareMathOperator{\colim}{colim}
\DeclareMathOperator{\cofib}{cofib}
\DeclareMathOperator{\sk}{sk}

\newcommand{\Ani}{\mathrm{Ani}}
\newcommand{\AniPair}{\mathrm{AniPair}}
\newcommand{\forget}{\mathrm{forget}}
\newcommand{\AniPDPair}{\mathrm{AniPDPair}}

\newcommand{\Lenv}{\mathrm{Lenv}}
\newcommand{\Fil}{\mathrm{F}}
\newcommand{\gr}{\mathrm{gr}}
\newcommand{\cc}{\mathrm{c}}
\newcommand{\op}{\mathrm{op}}
\newcommand{\NN}{\mathbb{N}}
\newcommand{\ZZ}{\mathbb{Z}}
\newcommand{\QQ}{\mathbb{Q}}

\newcommand{\GG}{\mathbb{G}}
\newcommand{\CC}{\mathbb{C}}
\newcommand{\EE}{\mathbb{E}}
\newcommand{\PP}{\mathbb{P}}

\newcommand{\LL}{\mathbb{L}}
\newcommand{\derL}{\mathrm{L}}

\newcommand{\HKR}{\mathrm{HKR}}
\newcommand{\catO}{\mathcal{O}}
\newcommand{\derD}{\mathcal{D}}
\newcommand{\derR}{\mathrm{R}}
\newcommand{\Poly}{\mathrm{Poly}}
\newcommand{\Hom}{\mathrm{Hom}}

\newcommand{\sss}{\mathrm{s}}
\newcommand{\an}{\mathrm{an}}

\newcommand{\Alg}{\mathrm{CAlg}^\heart}

\newcommand{\CAlg}{\mathrm{CAlg}}
\newcommand{\AlgMod}{\mathrm{CAlgMod}}
\newcommand{\Cech}{\mathrm{Cech}}
\newcommand{\Set}{\mathrm{Set}}
\newcommand{\Ab}{\mathrm{Ab}}

\newcommand{\Crys}{\mathrm{Crys}}
\newcommand{\adic}{\mathrm{adic}}
\newcommand{\pdadic}{\mathrm{PD-adic}}

\newcommand{\At}{\mathrm{At}}
\newcommand{\Comp}{\mathrm{Comp}}
\newcommand{\Spaces}{\mathrm{Spaces}}

\newcommand{\CrysCon}{\mathrm{Cryscon}}
\newcommand{\Tot}{\mathrm{Tot}}
\newcommand{\Day}{\mathrm{Day}}
\newcommand{\ins}{\mathrm{ins}}
\newcommand{\fil}{\mathrm{fil}}
\newcommand{\coins}{\mathrm{coins}}

\newcommand{\surj}{\mathrm{surj}}
\newcommand{\gen}{\mathrm{gen}}

\newcommand{\csmall}{\mathrm{c}}
\newcommand{\ch}{\mathrm{ch}}

\newcommand{\Perf}{\mathrm{Perf}}

\newcommand{\har}{\mathrm{Har}}
\newcommand{\spet}{\text{Sp\'et}}
\newcommand{\spdm}{\mathrm{SpDM}}
\newcommand{\spsch}{\mathrm{SpSch}}
\newcommand{\QCoh}{\mathrm{QCoh}}
\newcommand{\Cat}{\mathrm{Cat}}
\newcommand{\B}{\mathrm{B}}
\newcommand{\Shv}{\mathrm{Shv}}
\newcommand{\Pic}{\mathrm{Pic}}

\newcommand{\tr}{\mathrm{tr}}
\newcommand{\Tr}{\mathrm{Tr}}

\newcommand{\Sch}{\mathrm{Sch}}

\newcommand{\Aut}{\mathrm{Aut}}
\newcommand{\Vect}{\mathrm{Vect}}
\newcommand{\St}{\mathrm{St}}
\newcommand{\fppf}{\text{fppf}}

\newcommand{\matP}{\mathrm{P}}

\newcommand{\GL}{\mathrm{GL}}
\newcommand{\ev}{\mathrm{ev}}
\newcommand{\coev}{\mathrm{coev}}

\newcommand{\PDPair}{\mathrm{PDPair}}
\newcommand{\Pair}{\mathrm{Pair}}

\newcommand{\sqzero}{\mathrm{sqz}}
\newcommand{\comp}{\mathrm{comp}}
\newcommand{\const}{\mathrm{const}}
\newcommand{\stacksref}[1]{\cite[\href{https://stacks.math.columbia.edu/tag/#1}{#1}]{stacks-project}}
\newcommand{\heart}{\ensuremath\heartsuit}
\newcommand{\otimesr}{\, \underset{R'}{\widehat{\otimes}} \,}

\numberwithin{equation}{section}

\DeclareSourcemap{
  \maps[datatype=bibtex]{
    \map{
      \step[fieldset=shorthand, null]
    }
  }
}

\newcommand{\dR}{\mathrm{dR}}
\renewcommand{\(}{\left(}
\renewcommand{\)}{\right)}

\newcommand{\stackspace}{2.7}
\newcommand{\stack}[2][1cm]{\;\tikz[baseline, yshift=.65ex]%
    {\foreach \k [evaluate=\k as \r using (.5*#2+.5-\k)*\stackspace] in {1,...,#2}{%
    \ifodd\k{\draw[->](0,\r pt)--(#1,\r pt);}%
    \else{\draw[<-](0,\r pt)--(#1,\r pt);}\fi
    }}\;}

\begin{document}
\begin{center}
\LARGE{On the deformation theory of Fourier--Mukai transforms between Calabi--Yau varieties}\\
\normalsize{Wouter Rienks, University of Amsterdam}
\end{center}

\begin{abstract}
We study the deformation theory of fully faithful Fourier--Mukai transforms in both characteristic zero and mixed characteristic. Our main result shows that obstructions to deforming such transforms can be completely controlled by Hodge theory when the source variety has trivial canonical bundle. This generalizes work of Addington-Thomas and Lieblich-Olsson. The main technical contribution is a formula for the obstruction class measuring the failure of a Chern character to remain within the Hodge filtration as a cup product with a (derived) Kodaira--Spencer class.
\end{abstract}

\tableofcontents


\section{Introduction}
For a smooth and projective variety $X$, we denote with $\Perf(X)$ the category of perfect complexes on $X$. The goal of this article is to study the following question in both characteristic zero and mixed characteristic.
\begin{question}\label{iaa}
Let $X$ and $Y$ be smooth and projective varieties. Assume $X$ has trivial canonical bundle. Suppose we are given a fully faithful embedding $\Perf(X) \hookrightarrow \Perf(Y)$. If $\mathcal{X}$ and $\mathcal{Y}$ are deformations of $X$ and $Y$ respectively, when does the embedding extend to a fully faithful embedding $\Perf(\mathcal{X}) \hookrightarrow \Perf(\mathcal{Y})$?
\end{question}
This question has been studied before. For example, in 2014 Addington and Thomas studied the case where $X$ is a K3 surface and $Y$ is a cubic fourfold in characteristic zero, and in 2011 Lieblich and Olsson \cite{liebols2} studied the case when $X$ and $Y$ are K3 surfaces (and the embedding is an equivalence) in mixed characteristic. 

By a theorem of Orlov \cite[Theorem 5.14]{huybrechts} any fully faithful exact functor $\Perf(X) \to \Perf(Y)$ corresponds to a unique perfect complex $\mathcal{E} \in \Perf(X \times Y)$ (the \emph{Fourier--Mukai kernel}). Thus one may reduce Question \ref{iaa} to the following question (where $Z = X \times Y$). 
\begin{question}\label{iab}
Let $Z$ be a smooth and projective variety, and let $\mathcal{E} \in \Perf(Z)$. If $\mathcal{Z}$ is a deformation of $Z$, when does there exist an object $\tilde{\mathcal{E}}$ in $\Perf(\mathcal{Z})$ such that the derived restriction of $\tilde{\mathcal{E}}$ to $Z$ is quasi-isomorphic to $\mathcal{E}$? 
\end{question}
A necessary condition for the object $\tilde{\mathcal{E}}$ to exist is that the Chern character of $\mathcal{E}$ \emph{remains within the Hodge filtration} along $\mathcal{Z}$. This condition on the Chern character is the same as the condition used in formulating the \emph{variational Hodge conjecture} \cite[Footnote 13]{grothendieck}. We refer to Section \ref{sec_main_res} and Section \ref{sec_def_theory_hodge} for details on this condition.

The result of Addington and Thomas \cite[Theorem 7.1]{addingtonthomas14} then can be formulated as follows.
\begin{theorem}
Let $X$ be a K3 surface and let $Y$ be a cubic fourfold. Let $\mathcal{X}$ and $\mathcal{Y}$ be deformations of $X$ and $Y$ over the ring $\CC[t] / (t^n)$. Let $\mathcal{E} \in \Perf(X \times Y)$ be the kernel of a fully faithful transform $\Perf(X) \hookrightarrow \Perf(Y)$. If the Chern character of $\mathcal{E}$ remains within the Hodge filtration along $\mathcal{X} \times \mathcal{Y}$, then there exists $\tilde{\mathcal{E}}$ in $\Perf(\mathcal{X} \times \mathcal{Y})$ such that the derived restriction of $\tilde{\mathcal{E}}$ to $X \times Y$ is quasi-isomorphic to $\mathcal{E}$.
\end{theorem}
In particular, the above theorem can also be interpreted as a very specific case of the variational Hodge conjecture.

In this work, we generalize the above theorem to the case where $X$ is a smooth and projective variety with trivial canonical bundle, $Y$ is any smooth and projective variety, and the base is any local Artinian $\CC$-algebra. This generalization to an arbitrary base forces us to use more complicated machinery than Addington and Thomas, since the semiregularity result of Buchweitz and Flenner \cite{buchflen} no longer applies (see also Section \ref{sec_compare_bloch} for a more detailed comparison of our results to the classical result of Bloch \cite{bloch}). Ultimately, the generalization to arbitrary base allows us to generalize to mixed characteristic with relative ease. This will allow us to prove a more general version of a theorem of Lieblich and Olsson \cite{liebols2}. 
\subsection{Statement of main results}
\label{sec_main_res}
Let $A$ be an Artinian local $\CC$-algebra, and $\mathcal{X} \to \Spec(A)$ be a smooth and projective morphism with special fiber 
\[
X := \mathcal{X} \times_{\Spec(A)} \Spec(\CC)
\]
Then by \cite[Proposition 3.8]{bloch} there exists a canonical isomorphism
\[
\varphi_X \colon \coh^*_{\mathrm{dR}}(X / \CC) \otimes_\CC A  \xrightarrow{\sim} \coh^*_{\mathrm{dR}}( \mathcal{X} / A) 
\]
relating the de Rham cohomology of $\mathcal{X}$ with that of the special fiber. In the case that $A = \CC[t] / (t^n)$ the isomorphism $\varphi_X$ can be completely described in terms of the Gauss--Manin connection: It is the unique $A$-linear map sending $v \otimes 1$ to the unique $w$ such that $w \equiv v \pmod{\mathfrak{m}_A}$, and $\nabla(w) = 0$. For general $A$, the map $\varphi_X$ can be constructed using GAGA, we refer the reader to Section \ref{sec_def_theory_hodge} for the details.

For any morphism $X \to S$ of schemes over $\CC$ the de Rham cohomology comes with a canonical filtration 
\begin{align*}
\Fil^m \coh^*_{\mathrm{dR}}( X / S)  \subseteq \coh^*_{\mathrm{dR}} ( X / S) 
\end{align*}
for $m \in \ZZ_{\geq 0}$, called the \emph{Hodge filtration} \stacksref{0FM7}. Finally, for $i \in \ZZ_{\geq 0}$ and  $\mathcal{E} \in \Perf(X)$ one may define a Chern character 
\[
\ch_i(\mathcal{E}) \in \Fil^i \coh^{2i}_{\dR}(X / S)
\]
see \stacksref{0FWB}. 

We are now ready to state our main result in characteristic zero.
\begin{theorem}\label{thm_zero}Let $A$ be an Artinian local $\CC$-algebra. Let $X$ and $Y$ be smooth and projective varieties over $\CC$, such that $X$ has trivial canonical bundle. Let $\mathcal{X}$ and $\mathcal{Y}$ be deformations of $X$ and $Y$ over $A$.  Finally, let $\mathcal{E} \in \Perf(X \times Y)$ be the kernel of a fully faithful transform $\Phi_{\mathcal{E}} \colon \Perf(X) \hookrightarrow \Perf(Y)$. Then the following are equivalent: 
\begin{enumerate}[(i)]
\item The kernel $\mathcal{E}$ deforms to an object $\tilde{\mathcal{E}} \in \Perf(\mathcal{X} \times \mathcal{Y})$. 
\item One has
\[
\varphi_{X \times Y}(\ch_i(\mathcal{E}) \otimes 1) \in \Fil^i \coh^{2i}_{\dR}(\mathcal{X} \times \mathcal{Y} / A)
\]
for all $i$. 
\end{enumerate}
If (i) and (ii) hold, then $\Phi_{\tilde{\mathcal{E}}}$ is always fully faithful, and $\Phi_{\tilde{\mathcal{E}}}$ is an equivalence if and only if $\Phi_{\mathcal{E}} $ is an equivalence.
\end{theorem}
This has been studied before. In 2007 Toda \cite{toda} studied the case where $A = \CC[x] / (x^2)$. In 2009, Huybrechts, Macri and Stellari \cite{huybrechtsmacristellari} proved the above to be true when $X$ and $Y$ are K3--surfaces, $A = \CC[x] / (x^n)$ and $\Phi$ is an equivalence. In 2013, Addington and Thomas \cite{addingtonthomas14} proved the above to be true in the case that $A = \CC[x] / (x^n)$, $X$ is a K3-surface and $Y$ is a cubic fourfold.

Our second main result is a generalization of the above to mixed characteristic.
\begin{theorem}\label{thm_mixed}
Let $p$ be a prime number, let $A$ be an Artinian local $\ZZ_{(p)}$-algebra with a divided power structure $\gamma$ on $\mathfrak{m}_A$ such that $\gamma_p$ acts nilpotently on $\mathfrak{m}_A$ (see Definition \ref{gae}).  Set $k = A / \mathfrak{m}_A$. Let $X$ and $Y$ be smooth and projective varieties over $k$, such that $X$ has trivial canonical bundle. Let $\mathcal{X}$ and $\mathcal{Y}$ be lifts of $X$ and $Y$ over $A$. Let $\mathcal{E} \in \Perf(X \times Y)$ be the kernel of a fully faithful transform $\Phi_{\mathcal{E}} \colon \Perf(X) \hookrightarrow \Perf(Y)$. Assume that $p > \dim(X) + \dim(Y)$. Then the following are equivalent:
\begin{enumerate}[(i)]
\item The kernel $\mathcal{E}$ admits a lift $\tilde{\mathcal{E}} \in \Perf(\mathcal{X} \times \mathcal{Y})$.
\item The crystalline Chern character 
\[
\ch_i(\mathcal{E}) \in \Fil^i \coh^{2i}_{\dR}(\mathcal{X} \times \mathcal{Y} / A)
\]
of $\mathcal{E}$ lands in the Hodge filtration for all $i$. 
\end{enumerate}
If (i) and (ii) hold, then $\Phi_{\tilde{\mathcal{E}}}$ is always fully faithful, and $\Phi_{\tilde{\mathcal{E}}}$ is an equivalence if and only if $\Phi_{\mathcal{E}}$ is an equivalence.
\end{theorem}

\begin{corollary}\label{corol_mixed}
Let $k$ be a field of characteristic $p > 2$ and let $W = W(k)$ be the ring of Witt vectors over $k$. Let $X$ and $Y$ be smooth and projective varieties over $k$, such that $X$ has trivial canonical bundle. Let $\mathcal{X}$ and $\mathcal{Y}$ be lifts of $X$ and $Y$ over $W$.  Let $\mathcal{E} \in \Perf(X \times Y)$ be the kernel of a fully faithful transform $\Phi_{\mathcal{E}} \colon \Perf(X) \hookrightarrow \Perf(Y)$. Assume that $p > \dim(X) + \dim(Y)$. Then the following are equivalent:
\begin{enumerate}[(i)]
\item The kernel $\mathcal{E}$ admits a lift $\tilde{\mathcal{E}} \in \Perf(\mathcal{X} \times \mathcal{Y})$. 
\item The crystalline Chern character 
\[
\ch_i(\mathcal{E}) \in \Fil^i \coh^{2i}_{\dR}(\mathcal{X} \times \mathcal{Y} / W)
\]
of $\mathcal{E}$ lands in the Hodge filtration for all $i$. 
\end{enumerate}
If (i) and (ii) hold, then $\Phi_{\tilde{\mathcal{E}}}$ is always fully faithful, and $\Phi_{\tilde{\mathcal{E}}}$ is an equivalence if and only if $\Phi_{\mathcal{E}}$ is an equivalence.
\end{corollary}
In 2011,  Lieblich and Olsson \cite{liebols2} showed the above to be true in the special case that $X$ and $Y$ are K3--surfaces and $k$ is algebraically closed of characteristic $p > 2$, using a technique specific to  K3--surfaces and their moduli spaces of perfect complexes. 

\subsection{Overview of the text}
The main technical difficulty in the proofs of Theorem \ref{thm_zero} and Theorem \ref{thm_mixed} is analyzing whether or not a Chern class remains of Hodge type when deforming a variety. The starting point on this subject is a classical article by Bloch \cite{bloch}, who showed that one can define a \emph{Hodge-theoretic obstruction class} measuring the failure of the Chern class to remain within the Hodge filtration along a square zero extension. Moreover, Bloch gave an expression for the Hodge-theoretic obstruction class as a \emph{cup product with a Kodaira--Spencer class} (with conditions on the base $A$).

It was already shown that a similar expression exists for the \emph{obstruction class to deforming a vector bundle} by Illusie \cite{ill71two}, using his \emph{cotangent complex}. This was generalized to the case of a \emph{perfect complex} by Huybrechts and Thomas \cite{huto}. Moreover, Buchweitz and Flenner \cite{buchflen} constructed a \emph{semiregularity map} relating the two obstruction classes. 

In Section \ref{sec_def_theory_hodge}, we give an expression for the Hodge-theoretic obstruction class as a cup product with a Kodaira--Spencer class for a general base $A$. The main difficulty here is that Bloch's construction is of a topological nature, and therefore hard to apply algebraic techniques to. This is where derived algebraic geometry makes its appearance: the work of Pridham \cite{pridham} suggests that one could replace the topological isomorphism in the work of Bloch by nil-invariance of derived de Rham cohomology. It is here that the main technical results are stated.

To achieve this we will construct a theory of Chern classes in derived de Rham cohomology in Section \ref{sec_drc_stacks} (essentially following Bhatt--Lurie \cite{bhatt-lurie}). Moreover, we will show this Chern class corresponds to the trace of the Atiyah class, unifying it with Illusie's construction of the Chern class in \cite{ill71two} (see Proposition \ref{prop_compare_chern}). 

It then turns out that the above generalizes quite easily to mixed characteristic, if one replaces derived de Rham cohomology with the theory of derived crystalline cohomology (see work of Bhatt \cite{bhatt-padic} and Mao \cite{mao}). Throughout the text, the various related results for the crystalline case will usually be stated directly after the characteristic zero result, allowing for an easy comparison. 

Using the work of Căldăraru \cite{caldararu1, caldararu2}, we will finally show the semiregularity map is injective under the conditions of Theorem \ref{thm_mixed}. In Section \ref{sec_hochschild_semiregular}, we provide a Hochschild--theoretic formulation of the semiregularity map by means of the Hochschild--Kostant--Rosenberg isomorphism. Finally in Section \ref{sec_deforming_transforms} we show the semiregularity map is injective in the cases we need, to prove Theorems \ref{thm_zero} and \ref{thm_mixed}.

Our proof of Theorem \ref{thm_zero} relies heavily on derived algebraic geometry, which is needed since we assume $A$ to be very general. If one is only interested in the case $A = \CC[t] / (t^n)$, one can give a classical proof using only Theorem \ref{thm_semireg_inj} and $T^1$-lifting methods. The main upshot of our methods is that they can also be generalized to mixed characteristic.

\subsection*{Acknowledgements}
I wish to thank my advisor Lenny Taelman for suggesting the problem and his support along the way. Furthermore, I would like to express my gratitude to Richard Thomas for numerous comments and suggestions, to Dhyan Aranha for helping me figure out the argument for Lemma \ref{pushout}, and to Gijs Heuts, Zhouhang Mao, Tasos Moulinos and Alexander Petrov and  for useful comments and discussions. This research has been carried out at as part of the author's PhD thesis and was funded by the European Research Council (ERC), grant 864145.

\section{Preliminaries}
\label{sec_notation}
We will use the language of $\infty$-categories as developed in \cite{htt}. By an $n$-category we mean an $\infty$-category in which all mapping spaces are $(n - 1)$-truncated. For example, a $1$-category is an $\infty$-category in which all mapping spaces are discrete. We write $\mathcal{S}$ for the $\infty$-category of spaces, and denote with $\mathcal{S}_{\leq n}$ the full subcategory of $n$-truncated spaces. We write $\Cat_\infty$ for the $\infty$-category of $\infty$-categories. The inclusion 
\[
\mathcal{S} \to \Cat_\infty
\]
has a right adjoint which we will denote by $(-)^\simeq$. 

For an $\infty$-category $\mathcal{C}$, we shall denote with $\sss \mathcal{C}$, resp. $\cc \mathcal{C}$ the $\infty$-category of simplicial, resp. cosimplicial diagrams in $\mathcal{C}$. 

For any stable $\infty$-category $\mathcal{C}$ and $n \in \ZZ$ we shall denote with $[n] \colon \mathcal{C} \to \mathcal{C}$ the $n$-fold composition of the suspension functor \cite[Notation 1.1.2.7]{ha}.

For $k$ a ring, we denote with $\Ch(k)$ the $1$-category of chain complexes over $k$, with $\Ch(k)_{\mathrm{dg}}$ the dg-category of chain complexes over $k$ and with $\derD(k)$ the stable $\infty$-category $\mathrm{N}_{\mathrm{dg}}(\Ch(k)_{\mathrm{dg}})$. Note that we have a canonical functor $\Ch(k) \to \derD(k)$.

If $\mathcal{C}$ is a symmetric monoidal $\infty$-category, we shall denote with $\CAlg(\mathcal{C})$ the $\infty$-category of $\EE_{\infty}$-algebras in $\mathcal{C}$, see \cite{ha}. We will write $\CAlg_k := \CAlg(\derD(k))$, and denote with $\Alg_k$ the $1$-category of discrete commutative $k$-algebras.
\subsection{Filtrations}
The structure of a partially ordered set on $\NN$ gives $\NN$ the structure of a $1$-category such that there is a unique morphism $i \to j$ if $i \leq j$. We denote $\NN^{\mathrm{disc}}$ for the $1$-category with objects the natural numbers, and all morphisms the identity.
\begin{definition}
Let $\mathcal{C}$ be an arbitrary $\infty$-category. Then we define
\[
\mathcal{C}_\fil := \Fun(\NN^\op, \mathcal{C})
\]
the $\infty$-category of filtered objects in $\mathcal{C}$. For $X \in \mathcal{C}_\fil$, we write $\Fil^i X := X(i)$. Similarly, we let 
\[
\mathcal{C}_\gr := \Fun(\NN^{\mathrm{disc}}, \mathcal{C})
\]
If $\mathcal{C}$ is stable, define a functor $\gr \colon \mathcal{C}_\fil \to \mathcal{C}_\gr$ by 
\[
\gr (X)(i) = \cofib(\Fil^{i + 1}(X) \to \Fil^i(X))
\]
on $X \in \mathcal{C}_\fil$. We will refer to $\gr(X)$ as the \emph{associated graded} of the filtered object $X$, and use the shorthand notation $\gr^i(X) := \gr(X)(i)$. 
\end{definition}

One may give $\NN^{\mathrm{disc}}$ and $\NN^\op$ the structure of a symmetric monoidal category by setting $[p] \otimes [q] := [p + q]$. If $\mathcal{C}$ is symmetric monoidal and its tensor product preserves colimits in each variable separately, the procedure of Day convolution \cite[$\mathsection$2.2.6]{ha} then gives $\mathcal{C}_\gr$ and $\mathcal{C}_\fil$ the structure of a symmetric monoidal category. Explicitly, one has
\begin{align*}
\Fil^n \(X \otimes^\Day Y\) &:= \underset{p + q \geq n}{\colim}\  \Fil^p X \otimes \Fil^q Y \\
(A \otimes^\Day B)(n) &:= \underset{p + q = n}{\bigsqcup}\  A(p) \otimes B(q)
\end{align*}
for $X, Y \in \mathcal{C}_\fil$ and $A, B \in \mathcal{C}_\gr$. Moreover, $\gr$ has a canonical structure of a symmetric monoidal functor, that is there exist canonical isomorphisms
\[
\gr^n \(X \otimes^\Day Y\) \cong \underset{p + q = n}{\bigsqcup}\  \gr^p X \otimes \gr^q Y
\]
for $X, Y \in \mathcal{C}_\fil$. We will denote with 
\begin{equation}\label{proj_i}
\pi^\gr_i \colon \gr^p(X \otimes^\Day Y) \to \gr^{p - i}(X) \otimes \gr^i(Y)
\end{equation}
the projection on the $i$th component.

\begin{definition}[Filtered $\EE_\infty$-algebras]\label{def_filt_alg}
Let $\mathcal{C}$ be a symmetric monoidal $\infty$-category whose tensor product preserves colimits in each variable separately. We define the $\infty$-category of filtered $\EE_\infty$-algebras in $\mathcal{C}$ as 
\[
\CAlg_\fil(\mathcal{C}) := \CAlg((\mathcal{C}_\fil, \otimes^\Day)).
\]
\end{definition}
\begin{lemma}\label{lem_alg_coprod}
The symmetric monoidal structure on $\CAlg_\fil(\mathcal{C})$ is cocartesian, that is the coproduct of algebras is given by the Day convolution product of their underlying objects. In particular $\CAlg_\fil(\mathcal{C})$ admits finite coproducts.
\end{lemma}
\begin{proof}
See \cite[Proposition 3.2.4.7]{ha}.
\end{proof}
We warn the reader that this is a distinctly different category then $\CAlg(\mathcal{C})_\fil$. For $k$ a discrete commutative ring, we will write $\CAlg_\fil(k) := \CAlg_\fil(\derD(k))$. 

For any stable symmetric monoidal $\infty$-category $\mathcal{C}$ whose tensor product preserves colimits in each variable separately and $p \in \NN$, we have a lax symmetric monoidal functor
\[
\gr^{[0, p)} \colon \mathcal{C}_\fil \to \mathcal{C}_\fil
\]
defined by
\begin{equation}
\Fil^i \gr^{[0, p)}(X) := \begin{cases}
\cofib(\Fil^p X \to \Fil^i X) & i \leq p \\
0 & i > p
\end{cases} \label{aae}
\end{equation}
For any $p < q$ we have a natural transformation $\gr^{[0, q)} \to \gr^{[0, p)}$.  Thus if $\mathcal{C}$ admits all limits, we may define the \emph{completion functor}
\begin{align*}
\widehat{(-)} \colon \mathcal{C}_\fil &\to \mathcal{C}_\fil \\ 
X &\mapsto \widehat{X} := \lim_{p \to \infty} \gr^{[0, p)}(X)
\end{align*}
which is also lax symmetric monoidal. Explicitly, one has
\[
\Fil^p \widehat{X} = \lim_{q \to \infty} \cofib(\Fil^q X \to \Fil^p X)
\]
for any $X \in \mathcal{C}_\fil$. We thus get induced functors
\begin{align*}
\gr^{[0, p]} \colon \CAlg_\fil(\mathcal{C}) \to \CAlg_\fil(\mathcal{C}) \\
\widehat{(-)} \colon \CAlg_\fil(\mathcal{C}) \to \CAlg_\fil(\mathcal{C})
\end{align*}
By \cite[Corollary 3.2.2.4]{ha}, we have an equality
\[
\widehat{(-)} = \lim_{p \to \infty} \gr^{[0, p)}
\]
of functors $\CAlg_\fil(\mathcal{C}) \to \CAlg_\fil(\mathcal{C})$. If one sets $\mathcal{C}_{\fil, \comp}$ to be the full subcategory consisting of those $X \in \mathcal{C}_\fil$ such that the natural map $X \to \widehat{X}$ is an equivalence, one may show the functor $\widehat{(-)} \colon \mathcal{C}_\fil \to \mathcal{C}_{\fil, \comp}$ is left adjoint to the inclusion $\mathcal{C}_{\fil, \comp} \to \mathcal{C}_\fil$. For $X, Y \in \mathcal{C}_{\fil, \comp}$ we will denote 
\[
X \widehat{\otimes} Y := \widehat{X \otimes Y}
\]
and similarly for complete $X, Y \in \CAlg_\fil(\mathcal{C})$. 

Finally we will often use without mention that the functor $\gr \colon \mathcal{C}_{\fil, \comp} \to \mathcal{C}_\gr$ is conservative.
\subsection{Sheaves and stacks}
\label{extend_functors}
For any $\infty$-category $\mathcal{C}$ equipped with a Grothendieck topology $\tau$, and  any $\infty$-category $\mathcal{D}$ in which all limits exist, we shall denote by $\Shv_{\tau}(\mathcal{C}, \mathcal{D})$ the $\infty$-category of $\mathcal{D}$-valued sheaves on $\mathcal{C}$, see \cite[Definition  7.3.3.1]{htt}. More generally, for any $\infty$-topos $\mathcal{X}$ we shall denote by $\Shv_{\mathcal{D}}(\mathcal{X})$ the category of $\mathcal{D}$-valued sheaves on $\mathcal{X}$, i.e. the category of functors $\mathcal{X}^\op \to \mathcal{D}$ that preserve small limits. 
\begin{definition}
We define
\[
\St_k := \Shv_\fppf((\Alg_k)^\op, \mathcal{S})
\]
the $\infty$-category of \emph{higher stacks} over $k$. 
\end{definition}
By \cite[Proposition 6.2.2.7]{htt}, $\St_k$ has the structure of an $\infty$-topos. We warn the reader that these are \emph{underived} stacks, since $\Alg_k$ is the $1$-category of discrete commutative $k$-algebras. Note that inclusion $\Set \to \mathcal{S}$ induces a functor $\Sch_{/ k} \to \St_k$, so in particular we get a fully faithful Yoneda embedding 
\[
\Spec \colon (\Alg_k)^\op \to \St_k
\]
 (essentially because the fppf topology on affine schemes is subcanonical, see \stacksref{03O4}). 
 
For any $\infty$-category $\mathcal{D}$ in which all limits exist, a functor $\mathcal{F} \colon \Alg_k \to \mathcal{D}$ induces (by right Kan extension) a unique functor $\St_k^\op \to \mathcal{D}$ which we shall also denote by $\mathcal{F}$. 
Explicitly, for any $k$-stack $X$ one has
\begin{equation}\label{def_extend}
F(X) := \lim_{\Spec(R) \to X} \mathcal{F}(R)
\end{equation}
and in particular $\mathcal{F}(\Spec(R)) = \mathcal{F}(R)$. Using \cite[Theorem 4.1.3.1]{htt}, one may show that if $X$ is a scheme, the above can be computed as
\begin{equation}\label{def_extend2}
\mathcal{F}(X) = \lim_{\Spec(R) \subseteq X} \mathcal{F}(R)
\end{equation}
where the limit is over all affine opens $\Spec(R)$ in $X$. 

For any $\infty$-category $\mathcal{D}$ admitting all limits, the inclusion
\[
\Shv_{\mathcal{D}}(\St_k) \subseteq \Fun(\St_k^\op, \mathcal{D})
\]
admits a left adjoint $\mathcal{F} \mapsto \mathcal{F}^\dagger$ called \emph{sheafification}. If $\mathcal{F}$ is a sheaf one has $\mathcal{F}^\dagger(X) = \mathcal{F}(X)$ for all $k$-stacks $X$. 

\subsection{A survey of animation}
If $A$ is a ring and $M$ is an $A$-module, the functor $- \otimes_A M$ is in general not exact. However, it is exact when restricting to the subcategory of free modules. One may introduce the notion of a free resolution $P_\bullet \to N$ of a general $A$-module $N$, and define
\[
N \otimes^\derL_A M := P_\bullet \otimes_A N
\]
to get a better behaved tensor product. For this to make any sense, one needs a category of free resolutions, and it is well known that $\derD(A)_{\geq 0}$ is a good notion for this category. 

In this section we give a quick survey of \emph{animation}, a technique introduced in \cite{cs19} (based on ideas in \cite{quillen-rings} and \cite[\S 5.5.8]{htt}) to give a general way of achieving the above. Given a $1$-category $\mathcal{C}$ generated under colimits by a full subcategory $\mathcal{C}_0$ of nice (compact and projective) objects, we can form the $\infty$-category $\Ani(\mathcal{C})$ freely generated under sifted colimits of these objects. Given a functor $\mathcal{C} \to \mathcal{D}$ which behaves well on $\mathcal{C}_0$, one then obtains a well-behaved functor $\Ani(\mathcal{C}) \to \mathcal{D}$. 

We will give a quick and by no means complete survey, we urge the reader to read \cite[\S 5.5.8]{htt} and \cite{mao} first. Throughout this section, $n$ can be any natural number or the symbol $\infty$.

\begin{definition}[{\cite[\S 5.5.8]{htt}, \cite[Definition A.18]{mao}}]\label{def_compact}
Let $\mathcal{C}$ be a cocomplete category, and $C \in \mathcal{C}$. We say $C$ is \emph{compact} if the functor $\Map_{\mathcal{C}}(C, -) \colon \mathcal{C} \to \mathcal{S}$ commutes with filtered colimits. 

If $\mathcal{C}$ is a cocomplete $n$-category, we say that $C$ is $n$-projective if the functor $\Map_{\mathcal{C}}(C, -) \colon \mathcal{C} \to \Spaces_{\leq n - 1}$ commutes with geometric realizations. If $n = \infty$, we wil say $C$ is projective. 
\end{definition}
Note that although we work with $n$-projective for general $n$, we shall only be interested in the cases $n = 1$ and $n = \infty$.
\begin{definition}[{\cite[Definition A.22]{mao}}]\label{def_gen}
Let $\mathcal{C}$ be an $n$-category and $S \subseteq \mathcal{C}$ a set of objects in $\mathcal{C}$. We say that $S$ is a \emph{set of compact $n$-projective generators} for $\mathcal{C}$ if
\begin{enumerate}
\item $\mathcal{C}$ is cocomplete.
\item Every $X \in S$ is compact $n$-projective.
\item The set $S$ generates $\mathcal{C}$ under small colimits. 
\end{enumerate}
If there exists a set of compact $n$-projective generators of $\mathcal{C}$, we say that $\mathcal{C}$ is \emph{compact $n$-projectively generated}. If $n = \infty$, we say $\mathcal{C}$ is \emph{compact projectively generated}. 
\end{definition}
\begin{definition}
Let $\mathcal{C}$ a $n$-category which admits finite coproducts. We write
\[
\mathcal{P}_n(\mathcal{C}) := \Fun(\mathcal{C}^\op, \mathcal{S}_{\leq n - 1})
\]
and denote  with $\mathcal{P}_{\Sigma,n}(\mathcal{C}) \subseteq \mathcal{P}_n(\mathcal{C})$ the full subcategory consisting of those functors which preserve finite products. 
\end{definition}

\begin{prop}
 Let $\mathcal{C}$ be an $n$-category, and let $S$ be a set of compact $n$-projective generators for $\mathcal{C}$. Let $\mathcal{C}_0$ be the full subcategory on finite coproducts of objects in $S$. Then the Yoneda embedding $P_{\Sigma, n}(\mathcal{C}_0) \to \mathcal{C}$ is an equivalence.
\end{prop}
\begin{proof}
See \cite[Proposition A.29]{mao}.
\end{proof}
\begin{definition}[{\cite[Definition A.32]{mao}}]
Let $\mathcal{C}$ be a compact $n$-projectively generated $n$-category, let $\mathcal{S}$ be a set of compact and $n$-projective generators, and let $\mathcal{C}_0$ be the full subcategory spanned by finite coproducts of objects in $\mathcal{S}$. We define
\[
\Ani(\mathcal{C}) := \mathcal{P}_{\Sigma, \infty}(\mathcal{C}_0)
\]
the \emph{animation} of $\mathcal{C}$. 
\end{definition}

It is good to observe that $\Ani(\mathcal{C})$ is independent of the choice of compact projective generators for $\mathcal{C}$, see \cite[Remark A.33]{mao}. Note that we have a natural Yoneda embedding $\mathcal{C} \to \Ani(\mathcal{C})$ given by $X \mapsto \Map_{\mathcal{C}}(X, -)$. Moreover, by \cite[Remark 5.5.8.26]{htt}, the Yoneda embedding $\mathcal{C} \to \Ani(\mathcal{C})$ admits a left adjoint $\pi_0$.

\begin{lemma}\label{ani_gen}
Let $\mathcal{C}$ be a compact $n$-projectively generated $n$-category and let $S$ be a set of compact and $n$-projective generators. Then $\Ani(\mathcal{C})$ is compact projectively generated, and $S$ is a set of compact projective generators.
\end{lemma}
\begin{proof}
By definition, for any $X \in S$ the image in $\Ani(\mathcal{C})$ is compact and projective. Since clearly any element in $\Ani(\mathcal{C})$ can be written as a colimit of objects in $S$, the result follows.
\end{proof}
\begin{definition}Let $\mathcal{C}$ and $\mathcal{D}$ be $\infty$-categories. We denote with 
\[
\Fun_\Sigma(\mathcal{C}, \mathcal{D}) \subseteq \Fun(\mathcal{C}, \mathcal{D}) 
\]
the full subcategory of those functors which preserve filtered colimits and geometric realizations. 
\end{definition}
\begin{prop}[{\cite[Proposition 5.5.8.15]{htt}}]\label{left_kan}
Let $\mathcal{C}$ be a cocomplete $\infty$-category, and let $S \subseteq \mathcal{C}$ be a set of compact projective generators for $\mathcal{C}$. Let $\mathcal{C}_0 \subseteq \mathcal{C}$ be the full subcategory spanned by finite coproducts of objects in $S$. Let $\mathcal{D}$ be any $\infty$-category which admits filtered colimits and geometric realizations. 

Then the restriction map
\[
\Psi \colon \Fun_{\Sigma} (\mathcal{C}, \mathcal{D}) \to \Fun(\mathcal{C}_0, \mathcal{D})
\]
induces an equivalence of categories.  Moreover, any $g \in \Fun_{\Sigma} (\mathcal{C}, \mathcal{D})$ commutes with all colimits if and only if $\Psi(g)$ commutes with finite coproducts. Finally, for any $f \in \Fun(\mathcal{C}_0, \mathcal{D})$ the inverse image $\Psi^{-1}(f)$ is given by left Kan extension.
\end{prop}

We end by giving two examples of compact projectively generated categories which are essential to the rest of the text.
\begin{example}
Let $k$ be a ring. The set $S = \{k\}$ is a set of compact $1$-projective generators for the category of discrete $k$-modules $\derD(k)^\heart$. The full subcategory spanned by finite coproducts of objects in $S$ is the $1$-category of finite free $k$-modules. Moreover one has $\Ani(\derD(k)^\heart) \cong \derD(k)_{\geq 0}$. 
\end{example}
\begin{example}
For $k$ a ring, we denote with $\Alg_k$ the $1$-category of commutative $k$-algebras. The set $S = \{k[x]\}$ is a set of compact $1$-projective generators, and the full subcategory $\Poly_k \subseteq \Alg_k$ spanned by coproducts of objects in $S$ is the category of finitely generated polynomial $k$-algebras. The $\infty$-category $\CAlg_k^\an := \Ani(\Alg_k)$ is equivalent to the $\infty$-category of simplicial rings. 
\end{example}

\subsection{Higher algebraic stacks}
We now give an inductive definition of higher algebraic stacks, following Lurie's thesis \cite{thesislurie}.
\begin{definition}
Let $k$ be a ring. A morphism $f \colon X \to Y$ in $\St_k$ is a \emph{relative $0$-stack} if for any $A \in \Alg_k$ and any map $\Spec(A) \to Y$ in $\St_k$, the fiber product $\Spec(A) \times_{Y} X$ is an algebraic space in the sense of \stacksref{025X}. We say that $f$ is smooth if the maps $\Spec(A) \times_{Y} X \to \Spec(A)$ are smooth.

For $n > 0$, a morphism $f \colon X \to Y$ in $\St_k$ is a \emph{relative $n$-stack} if for any $A \in \Alg_k$ and any map $\Spec(A) \to Y$ in $\St_k$, there exists an effective epimorphism $p \colon U \to \Spec(A) \times_Y X$ which is a smooth relative $(n - 1)$-stack, where $U$ is a disjoint union of affine schemes. We will say that a relative $n$-stack $f \colon X \to Y$ is \emph{smooth} if for all $\Spec(A) \to Y$, the cover $U$ can be chosen to be smooth over $\Spec(A)$.

Finally, we define an \emph{algebraic stack} to be a morphism $X \to \Spec(k)$ which is a relative $n$-stack for some $n \in \NN$. 
\end{definition}

We may similarly define open immersions of relative $n$ stacks inductively.
\begin{definition}
A morphism of relative $0$-stacks is an open immersion if it is an open immersion of algebraic spaces. For $n > 0$, we say that a morphism $U \to X$ of relative $n$-stacks is an \emph{open immersion} if there exists a surjective map $T \to X$ which is a relative $(n - 1)$-stack such that $U_T \to X_T$ is an open immersion. In this case, we say $U \subseteq X$ is an \emph{open substack}. 
\end{definition}

For $G$ an affine group scheme (in the classical sense) we can form a simplicial object
\[
\cdots \stack{7} G \times_{\Spec(k)} G \stack{5} G \stack{3} \Spec(k)
\]
in $\sss \St_k$, see \cite[Definition 4.25]{khanstacks}.  We define the \emph{classifying stack} $\mathrm{B}G$ to be the colimit of this diagram in $\St_k$. One may show $\mathrm{B}G$ is an algebraic $1$-stack. By \cite[Theorem 4.28]{khanstacks}, for any $k$-scheme $X$ one may canonically identify the groupoid $\Map_k(X, \mathrm{B}G)$ with the groupoid of $G$-torsors $T \to X$. 

 \subsection{Module categories on stacks}
Let $k$ be a ring. For any stack $X$ over $k$, we denote with $\derD(X)$ the stable $\infty$-category defined as 
\[
\derD(X) := \lim_{\Spec(A) \subseteq X} \derD(A)
\]
For $A \in \Alg_k$, we say that an object $\mathcal{E} \in \derD(A)$ is \emph{perfect} if it is compact. For a general $X$ in $\St_k$, we say that $\mathcal{E} \in \derD(X)$ is perfect if the pullback $f^*(\mathcal{E})$ is perfect for all $f \colon \Spec(A) \to X$. We denote with $\Perf(X)$ the full subcategory of perfect objects.  We will say that an object $\mathcal{E} \in \derD(X)$ is finite locally free (of rank $n$) if $f^*(\mathcal{E})$ is a finite locally free $A$-module  (of rank $n$)  in homological degree $0$ for all maps $f \colon \Spec(A) \to X$. We will denote with $\Vect(X)$ the full subcategory of finite locally free modules, and with $\Vect_n(X)$ the full subcategory of finite locally free modules of rank $n$. We will write 
\[
\Pic(X) := \Vect_1(X)^\simeq
\]
We refer to elements $\mathcal{E} \in \Vect(X)$ as \emph{vector bundles}, and to elements $\mathcal{L} \in \Pic(X)$ as \emph{line bundles}. We will write $\mathrm{K}_0(\Vect(X))$ for the abelian group generated by the vector bundles on $X$ with relations coming from short exact sequences, see \stacksref{0FDE}. For $\iota \colon Z \hookrightarrow X$ a closed immersion of schemes and $\mathcal{E} \in \derD(X)$, we will sometimes write $\mathcal{E} \rvert_Z := \iota^*(\mathcal{E})$. 
\begin{lemma}\label{vect_sheaf2}
Let $k$ be a ring. Let 
\[
\mathcal{F} \colon \St_k^\op \to \Cat_\infty
\]
be an element of $\{\derD(-), \Perf(-), \Vect(-), \Vect_n(-), \Pic\}$. Then $\mathcal{F}$ is a sheaf (for the fppf topology).
\end{lemma}
\begin{proof}
For $\derD(-)$ this follows from \cite[Corollary D.6.3.3]{sag}. Since the condition that an object is perfect or locally free (of rank $n$) is local for the flat topology by \cite[Proposition 2.8.4.2]{sag}, the rest follow by \cite[Corollary 3.3.3.2]{htt}.
\end{proof}
\section{Derived de Rham cohomology}
In this section we give the constructions of derived de Rham cohomology and derived crystalline cohomology, and state various basic results we need. Most definitions and results are due to Bhatt and Mao, see \cite{bhatt-cddrc} \cite{bhatt-padic} \cite{mao}. In Section \ref{ss_comp_dr} we state and prove a result comparing (Hodge completed) derived de Rham cohomology of a surjective ring map $A \to (A / I)$ with the derived completion of $A$ in $I$, a result originally due to Bhatt \cite{bhatt-cddrc}.  
 
Throughout this section, fix a base ring $k$.
\label{sec_cohomology_affine}
\subsection{The derived de Rham complex}

In this section, we briefly recall the definition of the derived de Rham complex, and state some of the properties we need. Recall that classical de Rham cohomology is defined for a morphism of rings $A \to B$. Thus one would like to define derived de Rham cohomology for a morphism of animated rings $A \to B$. To do this, it is helpful to find a set of compact projective generators for the category $\Fun(\Delta^1, \CAlg_k^\an)$. This can be done in big generality as follows.

For $\mathcal{C}$ a compact and $1$-projectively generated $1$-category with ${S}$ a set of compact $1$-projective generators, the $\infty$-category $\Fun(\Delta^p, \Ani(\mathcal{C}))$ is compact projectively generated. By Lemma \ref{lem_fun_proj_gen}, a set of compact projective generators is given by
\[
S_p := \left \{\underbrace{0 \to \dots \to 0}_{i \text{ times}} \to \underbrace{X \to \dots \to X}_{p - i + 1 \text{ times}} \mid \substack{i \in \{0, \dots, p\} \\ X \in S} \right \}
\]
We will write $\Fun(\Delta^p, \Ani(\mathcal{C}))_{\mathrm{gen}}$ for the full subcategory of $\Fun(\Delta^p, \Ani(\mathcal{C}))$ spanned by finite coproducts of elements in $S_p$ (see Definition \ref{def_gen_fun}), note that this depends on a choice of compact $1$-projective generators for $\mathcal{C}$.

We are now ready to define the derived de Rham complex. 
\begin{definition}[Derived de Rham complex]
We define the \emph{derived de Rham complex}
\[
\dR_{- / -} : \Fun(\Delta^1, \CAlg_k^\an) \to \CAlg_\fil(k)
\]
as the left Kan extension of the composition 
\[
\Fun(\Delta^1, \CAlg_k^\an)_\gen \subseteq \Fun(\Delta^1, \Poly_k) \xrightarrow{\Omega^\bullet_{-/-}}  \CAlg(\Ch(k)_\fil) \to \CAlg_\fil(k)
\]
where $\Omega^\bullet_{-/-}$ denotes the classical de Rham complex equipped with the Hodge filtration (\stacksref{0FKL}) and (graded) multiplication. For $p \in \NN$ will write
\[
\LL^p_{- / -} := \gr^p \dR_{- / -}[p]
\]
for the graded pieces, and refer to the completion $\widehat{\dR}_{-/-}$ as the  \emph{Hodge completed derived de Rham complex}. 
\end{definition}

If $A \to B$ is a smooth map of $k$-algebras one has $\LL_{B / A} = \Omega_{B / A}$. It follows that for any map of $k$-algebras $A \to B$, the complex $\LL_{B / A}$ coincides with Illusie's cotangent complex (see \cite{ill71}).

\begin{construction}\label{dr_iso_z}
Let $A \to B$ be a smooth map of $k$-algebras. Then $\Omega_{B / A}$ is a finitely generated $B$ module. Hence the Hodge filtration on $\Omega^\bullet_{B / A}$ is finite, and in particular $\Omega^\bullet_{B / A}$ is complete. We thus get a canonical map
\[
\widehat{\dR}_{B / A} \to \Omega^\bullet_{B / A}
\]
in $\CAlg_\fil(k)$. Applying \stacksref{08R5} to the graded pieces, we get a canonical equivalence
\[
\widehat{\dR}_{B / A} \xrightarrow{\sim} \Omega^\bullet_{B / A}
\]
in $\CAlg_\fil(k)$.
\end{construction}

\begin{lemma}\label{dr_iso_p}
 Let $p$ be a prime, let $k$ be a $\ZZ / p^n \ZZ$-algebra for some $n \in \NN$, and let $A \to B$ a smooth map of $k$-algebras. Then the maps
 \[
{\dR}_{B / A} \to \widehat{\dR}_{B / A} \to \Omega^\bullet_{B / A}
 \]
are  equivalences in $\CAlg_\fil(k)$. 
\end{lemma}
\begin{proof}
It suffices to show the first map is an equivalence, as the second was already shown to be an equivalence in general.  By \cite[Corollary 3.10]{bhatt-padic} the natural map
\begin{equation} \label{dr_iso}
\dR_{B / A} \to \Omega^\bullet_{B / A}
\end{equation}
is an  equivalence in $\derD(k)$. For any $i \geq 0$, we have a commutative diagram
\[
\begin{tikzcd}
\Fil^i \dR_{B / A} \dar  \rar& \dR_{B / A} \dar  \rar& \dR_{B / A} / \Fil^i \dar \\
\Fil^i \Omega^\bullet_{B / A} \rar& \Omega^\bullet_{B / A} \rar& \Omega^\bullet_{B / A} / \Fil^i
\end{tikzcd}
\]
in $\derD(k)$, where the rows are fiber sequences. The middle vertical map is an equivalence by (\ref{dr_iso}). Since $A \to B$ is smooth, the map $\LL^j_{B / A} \to \Omega^j_{B / A}$ is an equivalence for all $j$, hence the right vertical map is an equivalence as well. It follows that the left map is an equivalence for all $i \geq 0$, and thus the map 
\[
\dR_{B / A} \to \widehat{\dR}_{ B / A}
\]
is an equivalence in $\derD(k)_\fil$. Since the map $\CAlg_\fil(k) \to \derD(k)_\fil$ is conservative, we conclude.
\end{proof}
On the contrary, if $\QQ \subseteq k$, one may show using the Poincar\'e lemma that $\dR_{P / A} \cong A$ for any polynomial $A$-algebra $P$. By the lemma below, it follows that $\Fil^0\dR_{B / A} \cong A$ for all $A$-algebras $B$.
\begin{lemma}\label{lem_dr_colim}
The functor 
\[
\dR_{-/-} \colon \Fun(\Delta^1, \CAlg_k^\an) \to \CAlg_\fil(k)
\]
commutes with small colimits.
\end{lemma}
\begin{proof}
By Proposition \ref{left_kan}, it suffices to show the functor
\[
\Fun(\Delta^1, \Poly_k)_\gen \xrightarrow{\Omega^\bullet_{-/-}} \CAlg_\fil(k)
\]
preserves finite coproducts. By  Lemma \ref{lem_alg_coprod} and an induction argument, it suffices to show the natural maps
\begin{align}
\dR_{B / A} \otimes^\Day_{k} k[x] &\to \dR_{B[x] / A[x]} \label{eq_coprod_dr1} \\
\dR_{B / A} \otimes^\Day_{k} \dR_{k[x] / k}  &\to \dR_{B[x] / A} \label{eq_coprod_dr2}
\end{align}
are equivalences in $\CAlg_\fil(k)$ for any $(A \to B) \in \Fun(\Delta^1, \Poly_k)_\gen$. Since the forgetful functor $\CAlg_\fil(k) \to \derD(k)_\fil$ is conservative, it suffices to check (\ref{eq_coprod_dr1}) and (\ref{eq_coprod_dr2}) are equivalences in $\derD(k)_\fil$. 

Now (\ref{eq_coprod_dr1}) is an equivalence by \stacksref{0FL5}. To show that (\ref{eq_coprod_dr2}) is an equivalence in $\derD(k)_\fil$, it suffices to check that the induced map on associated gradeds in $\derD(k)_\gr$ is an isomorphism (note that filtered objects are complete as we are considering finitely generated polynomial algebras). We thus need to confirm the natural map
\[
\Omega^{p - 1}_{B / A} \otimes_k \Omega_{k[x] / k} \oplus \Omega^{p}_{B / k} \otimes_k k[x] \to \Omega^{p}_{B[x] / A}
\]
is an equivalence in $\derD(k)$ for all $p$, which follows by taking wedge powers of the equation
\[
\Omega_{B / A} \otimes_k k[x] \oplus B[x] \mathrm{d} x = \Omega_{B[x] / A} \qedhere
\]
\end{proof}
\begin{corollary}\label{dr_coprod_fil}
The functor 
\[
\dR_{-/-} \colon \Fun(\Delta^1, \CAlg_k^\an) \to \derD(k)_\fil
\]
commutes with sifted colimits, and sends finite coproducts to finite Day convolution products.
\end{corollary}
\begin{proof}
The statement about sifted colimits follows by combining \cite[Corollary 3.2.3.2]{ha} with Lemma \ref{lem_dr_colim}. The statement about finite coproducts follows by combining Lemma \ref{lem_alg_coprod} with Lemma \ref{lem_dr_colim}.
\end{proof}
\begin{corollary}[K\"unneth formula]\label{corol_kunneth}
Let $k$ be a ring, let $X$ be a smooth scheme over $k$, and let $Y$ be any stack over $k$. Then the natural map
\begin{equation}
\dR_{X / k} \otimes \dR_{Y / k} \to \dR_{X \times Y / k} \label{kunneth}
\end{equation}
is an equivalence in $\CAlg_\fil(k)$. 
\end{corollary}
\begin{proof}
If $A$ and $B$ are discrete $k$-algebras and $A$ is smooth, then the underived tensor product $A \otimes_k B$ computes the coproduct in $\CAlg_k^\an$. It follows that the map
\[
\dR_{A / k} \otimes \dR_{B / k} \to \dR_{A \otimes B / k} 
\]
is an equivalence in $\CAlg_\fil(k)$ whenever $A$ is smooth. Using (\ref{def_extend}), it follows that for any stack $Y$ over $k$ and any smooth $k$-algebra $A$ the natural map
\[
\dR_{A / k} \otimes \dR_{Y / k} \to \dR_{\Spec(A) \times Y / k} 
\]
is an equivalence in $\CAlg_\fil(k)$. The result follows using (\ref{def_extend2}). 
\end{proof}

We warn the reader that $\widehat{\dR}$ generally does not commute with colimits. If it would, then since $\widehat{\dR}_{P / \QQ} \simeq \QQ$ for all polynomial $\QQ$-algebras, one would have $\widehat{\dR}_{B / \QQ} \simeq \QQ$ for any $\QQ$-algebra $B$. But since $\QQ[x, \frac{1}x]$ is smooth over $\QQ$, by Construction \ref{dr_iso_z} we would get an equivalence
\[
\QQ \xrightarrow{\sim} \Omega^\bullet_{\QQ[x, \frac{1}{x}] / \QQ}
\]
in $\derD(\QQ)$. This clearly cannot exist since 
\[
\coh^1(\Omega^\bullet_{\QQ[x, \frac{1}{x}] / \QQ}) = \QQ
\]
but $\coh^1(\QQ) = 0$. 

For the reader's convenience, we give an explicit description of both $\dR$ and $\widehat{\dR}$. By \cite[Lemma 5.5.8.13]{htt}, for any $A \in \CAlg_k^\an$, there exists a simplicial ring $A_* \in \Fun(\Delta^\op, \Poly_k)$ such that
\[
A = \underset{\Delta^\op}\colim \ A_*
\]
One may use Lemma \ref{lem_dr_colim} to show the chain complexes
\[
\Tot^\oplus\(\Omega^\bullet_{A_* / k}\) \text{ and } \Tot^\Pi\(\Omega^\bullet_{A_* / k}\) 
\]
are isomorphic to $\Fil^0 \dR_{A / k}$ and $\Fil^0  \widehat{\dR}_{A / k}$ respectively (in $\derD(k)$). 

We end by showing that $\dR_{- / -} / \Fil^p$ and $\widehat{\dR}_{- / -}$ satisfy a derived descent statement, as was first observed by \cite[Remark 2.8]{bhatt-cddrc}. To formulate the statement, we first introduce some notation. 

Recall that for any $\infty$-category $\mathcal{C}$, we denote with $\cc \mathcal{C}$ the $\infty$-category of cosimplicial diagrams in $\mathcal{C}$. For any $A \in \Alg_k$, the functor
\[
\ev_{[0]} \colon \cc \Alg_{A}  \to \Alg_A
\]
commutes with all limits by \cite[Proposition 5.1.2.3]{htt}, and thus admits a left adjoint 
\[
\Cech(A \to -) \colon \Alg_A \to \cc \Alg_A 
\]
Explicitly, for $B \in \Alg_A$ the cosimplicial object $\Cech(A \to B) \in \cc\Alg_{A}$ is given by
\[
[n] \mapsto \underbrace{B \otimes_A \dots \otimes_A B}_{n + 1 \text{ times}}
\]
Moreover, for any element $(A \to A')$ in $\Fun(\Delta^1, \Alg_k)$, one has a commutative diagram
\[
\begin{tikzcd}
\Alg_A & \lar{\ev_{[0]}}  \cc \Alg_A \\
\Alg_{A'}\uar{\mathrm{forget}} & \lar{\\ev_{[0]}}  \cc \Alg_{A'} \uar{\mathrm{forget}}
\end{tikzcd}
\]
For any $B \in \Alg_{A'}$ one has $(\ev_{[0]} \circ \mathrm{forget})(\Cech(A' \to B)) = B$. Thus the counit of the adjunction
\[
\Cech(A \to -) \vdash \ev_{[0]}
\]
induces a natural transformation $\Cech(A \to -) \to \Cech(A' \to -)$. We thus get a functor
\[
\Cech(- \to -) \colon \Fun(\Delta^1, \CAlg_k^\an)_\gen \to \cc \CAlg_k^\an
\]
and by left Kan extension a functor
\[
\Cech(- \to -) \colon \Fun(\Delta^1, \CAlg_k^\an) \to \cc\CAlg_k^\an
\]
\begin{lemma}\label{bhatt_neat}
Let $(A \to B \to C) \in \Fun(\Delta^2, \CAlg_k^\an)$. Then for all $p \geq 1$, 
\begin{align*}
\lim_\Delta \LL^p_{\Cech(A \to B) / A} \cong 0 \\
\lim_\Delta C \otimes_{\Cech(A \to B)} \LL^p_{\Cech(A \to B) / A} \cong 0
\end{align*}
in $\derD(k)$.
\end{lemma}
\begin{proof}
The first statement is \cite[Corollary 2.7]{bhatt-cddrc}. For the second statement, write $B_n = B^{ \otimes_A n }$. Then note that by Lemma \ref{lem_dr_colim} one has
\[
\LL_{B_n/A} \otimes_{B_n} C \cong (\LL_{B / A} \otimes_B C)^{\oplus n}
\]
Thus the second statement follows by taking wedge powers of \cite[Lemma 2.5]{bhatt-cddrc}, where one takes $A$ to be the constant cosimplicial ring $C$ and one takes $M$ to be $\LL_{B / A} \otimes_B C$.
\end{proof}
\begin{corollary}\label{corol_descent_dr2}
For any $(A \to B \to C) \in \Fun(\Delta^2, \CAlg_k^\an)$ and $p \in \ZZ_{\geq 0}$, the natural map
\[
\lim_\Delta \Cech(B \to C) \otimes_B \LL^p_{B / A} \to \lim_\Delta \LL^p_{\Cech(B \to C) / A}
\]
is an equivalence.
\end{corollary}
\begin{proof}
We follow \cite[Theorem 3.1]{integralpadic}. The transitivity sequence for the sequence of cosimplicial rings $A \to B \to \Cech(B \to C)$ yields a short exact sequence 
\[
\Cech(B \to C) \otimes_B \LL_{B / A} \to \LL_{\Cech(B \to C) / A} \to \LL_{\Cech(B \to C) / B}
\]
in $\cc \derD(k)$. Taking (pointwise) wedge powers, we see that $\LL^p_{\Cech(B \to C) / A}$ comes with a natural filtration with graded pieces
\[
\gr^j(\LL^p_{\Cech(B \to C) / A}) = \LL^j_{B / A} \otimes_B \LL^{p - j}_{\Cech(B \to C) / B}
\]
By the first statement in Lemma \ref{bhatt_neat}, all graded pieces except the $j = p$ piece vanish after taking the limit over $\Delta$, which gives the result.
\end{proof}
\begin{prop}\label{prop_dr_sheaf}
Let $k$ be a ring. The functor
\[
\widehat{\dR}_{- / k} \colon \St_k^\op \to \CAlg_\fil(k)
\]
is a sheaf for the fppf topology.
\end{prop}
\begin{proof}
By \cite[Proposition 1.3.1.7]{sag}, it suffices to show the functor
\[
\widehat{\dR}_{- / k} \colon \Alg_k \to \CAlg_\fil(k)
\]
is a sheaf for the fppf topology. Commuting limits with limits, this follows directly from Lemma \ref{corol_descent_dr2}. 
\end{proof}
We also need the following statement, which can best be described as `descent on the base'.
\begin{lemma}\label{lem_descent_dr}
For any $(A \to B \to C) \in \Fun(\Delta^2, \CAlg_k^\an)$ and $p \in \ZZ_{\geq 0}$, the natural map
\[
\dR_{C / A} / \Fil^p \to \lim_\Delta \dR_{C / \Cech(A \to B)} / \Fil^p
\]
is an equivalence.
\end{lemma}
\begin{proof}
It suffices to show the natural map 
\[
\LL^p_{C / A} \to \lim_{\Delta} \LL^p_{C / \Cech(A \to B)}
\]
is an equivalence for all $p$. The transitivity sequence for $A \to \Cech(A \to B) \to C$ induces a filtration on $\LL^p_{C / A}$ with 
\[
\gr^j(\LL^p_{C / A}) = (\LL^j_{\Cech(A \to B) / A} \otimes_{\Cech(A \to B)} C) \otimes_C \LL^{p - j}_{C / \Cech(A \to B)}
\]
for $j \in \{0, \dots, p\}$. By the second part of Lemma \ref{bhatt_neat} and an Eilenberg-Zilber argument, after taking the limit over $\Delta$ all terms vanish except the $j = 0$ term, which proves the result.
\end{proof}
Finally, we will need the following nil-invariance result.
\begin{theorem}[Nil-invariance of derived de Rham cohomology]\label{theorem_goodwillie_dr}
Let $k$ be a ring such that $\QQ \subseteq k$, and let $A \to B$ be a morphism in $\CAlg^\an_k$ such that $\pi_0(A) \to \pi_0(B)$ is surjective. If $\ker(\pi_0(A) \to \pi_0(B))$ is a nilpotent ideal in $\pi_0(A)$, then the natural map
\begin{equation}
\widehat{\dR}_{A / k} \to \widehat{\dR}_{B / k} \label{eq_goodwillie_dr}
\end{equation}
is an equivalence in $\CAlg_k$. 
\end{theorem}
\begin{proof}
One may prove this using the analogous result for Harthshorne's algebraic de Rham cohomology (see \cite{bhatt-cddrc}) or by comparing with periodic Hochschild homology (see \cite{pridham}). An alternative direct proof may be found in \cite{rienks}.
\end{proof}

\subsection{The derived crystalline complex}
In this section, we give a short survey of the Mao's construction of \emph{derived crystalline cohomology} \cite{mao}. We start by generalizing our definition of de Rham cohomology to morphisms of divided power rings. Throughout this section, $k$ can be any commutative ring. Before we can give the definition, we need some preliminaries.

We will (as in \cite{mao}) denote with $\PDPair_k$ the $1$-category of PD-rings $(R, I, \gamma)$ such that $R$ is a commutative $k$-algebra, see \stacksref{07GU}.  We write $I^{[p]} \subseteq R$ for the $p$-th \emph{divided power ideal}, see \stacksref{07HQ}. We will often omit $\gamma$ from the notation and denote a PD-ring $(R, I, \gamma)$ with $(R \to R / I)$ instead. For $A$ a ring, we will denote with $(A\langle x_1, \dots, x_n \rangle \to A)$ the PD-ring freely generated on $n$ variables $x_1, \dots, x_n$, see \stacksref{07H4}. For $(R' \to R) \to (A' \to A)$ a morphism of PD rings, we will denote with $\Omega_{(A'\to A) / (R' \to R)}$ the $A'$-module of \emph{divided power differentials} over $R'$, see \stacksref{07HQ} (note that it only depends on $(A' \to A)$ and $R'$). Finally we will write 
\[
\Omega^p_{(A'\to A) / (R' \to R)} := \bigwedge^p_{A'} \Omega_{(A'\to A) / (R' \to R)}
\]

Unfortunately, the category $\PDPair_k$ is not compact projectively generated, but we can remedy the situation. Define $\PDPair_{k,\gen}$ to be the full subcategory of f.g. free PD-rings over f.g. polynomial algebras, i.e. the full subcategory on  objects of the form $k[y_1, \dots, y_m]\langle x_1, \dots, x_n \rangle$. Then \cite[Lemma 3.13]{mao} shows there exists a fully faithful embedding $\PDPair_k \hookrightarrow \mathcal{P}_{\Sigma, 1}(\PDPair_{k, \gen})$. Following \cite{mao}, we set
\[
\AniPDPair_k := \Ani(\PDPair_{k,\gen})
\]
the category of \emph{animated divided power algebras}.

\begin{definition}[PD-de Rham cohomology]
Let 
\[
\begin{tikzcd}
A' \arrow[->>]{r} & A\\
R' \uar \arrow[->>]{r} &R \uar
\end{tikzcd}
\]
be an object of $\Fun(\Delta^1, \PDPair_k)$. Write $J = \ker(A' \to A)$. We define the \emph{PD de Rham complex}
by 
\begin{align*}
 \Omega^\bullet_{(A' \to A) / (R' \to R)} := \left[A' \to \Omega_{(A' \to A) / (R' \to R)} \to \Omega^2_{(A' \to A) / (R' \to R)} \to \dots\right] \in \Ch(k)
\end{align*}
It comes with a filtration $\Fil_{\pdadic}^\bullet$ given by
\[
\Fil^i \Omega^p_{(A' \to A) / (R' \to R)} := \begin{cases}
J^{[i - p]} \otimes_{A'} \Omega^p_{(A' \to A) / (R' \to R)} & i \geq p \\
\Omega^p_{(A' \to A) / (R' \to R)} & i < p
\end{cases}
\]
and a canonical (graded) multiplication. We define the \emph{derived PD filtered de Rham complex} 
\[
\Fil^\bullet_\pdadic \dR_{(- \to -) / (- \to -)} \colon \Fun(\Delta^1, \AniPDPair_k) \to \CAlg_\fil(k) 
\]
as the left Kan extension of the functor 
\[
\Fil^\bullet_\pdadic \Omega^\bullet_{(- \to -) / (- \to -)}  \colon \Fun(\Delta^1, \PDPair_k)_\gen \to \CAlg_\fil(k) 
\]
\end{definition}
Observe that $\Omega^\bullet_{(A' \to A) / (R' \to R)}$ only depends on $A' \to A$ and $R'$. However, $\dR_{(A' \to A) / (R' \to R)}$  does depend on the pair $(R' \to R)$. Clearly, for any element $(A \to B) \in \Fun(\Delta^1, \Alg_k)$ one has
\[
\dR_{(B \to B) / (A \to A)} = \dR_{B / A}
\]
in $\CAlg_\fil(k)$. 
\begin{lemma}\label{lem_extended_dr_colim}
The functor 
\[
\Fil^\bullet_\pdadic \dR_{(- \to -) / (- \to -)} \colon \Fun(\Delta^1, \AniPDPair_k) \to \CAlg_\fil(k) 
\]
commutes with small colimits.
\end{lemma}
\begin{proof}
Similarly as in the proof of Lemma \ref{lem_dr_colim}, by Proposition \ref{left_kan} and Lemma \ref{lem_alg_coprod} it suffices to show that for any element 
\[
((R' \to R) \to (A' \to A)) \in \Fun(\Delta^1, \PDPair_k)_\gen
\]
the maps
\begin{align}
\underset{p + q \geq n}{\colim} \ \Fil^p \Omega^\bullet_{T / Q} \underset{k}{\otimes} \Fil^q \Omega^\bullet_{k[x] / k[x]} &\to \Fil^n \Omega^\bullet_{(A'[x] \rangle \to A[x]) / (R'[x] \to R[x])} \label{crys_coprod2} \\
\underset{p + q \geq n}{\colim} \  \Fil^p \Omega^\bullet_{T / Q} \underset{k}{\otimes} \Fil^q \Omega^\bullet_{k[x] / k} &\to \Fil^n \Omega^\bullet_{(A'[x] \rangle \to A[x]) / (R' \to R)} \label{crys_coprod4} \\
\underset{p + q \geq n}{\colim} \  \Fil^p \Omega^\bullet_{T / Q} \underset{k}{\otimes} \Fil^q \Omega^\bullet_{(k \langle x \rangle \to k) / (k \langle x \rangle \to k)} &\to \Fil^n \Omega^\bullet_{(A'\langle x \rangle \to A) / (R'\langle x \rangle \to R)} \label{crys_coprod1} \\
\underset{p + q \geq n}{\colim} \  \Fil^p \Omega^\bullet_{T / Q} \underset{k}{\otimes} \Fil^q \Omega^\bullet_{(k \langle x \rangle \to k) / (k \to k)} &\to \Fil^n \Omega^\bullet_{(A'\langle x \rangle \to A) / (R' \to R)} \label{crys_coprod3} 
\end{align}
are equivalences in $\derD(k)$ for all $n \geq 0$, where for typographical reasons we use the shorthand notation $T := (A' \to A)$ and $Q := (R' \to R)$. One immediately sees (\ref{crys_coprod2}) is an equivalence after observing that
\[
\Omega_{(A'[x] \to A[x]) / (R'[x] \to R[x])} \cong \Omega_{(A' \to A) / (R' \to R)} \otimes_k k[x]
\]
(see \stacksref{07HS}). Similarly, (\ref{crys_coprod4}) is seen to be an equivalence after observing that
\[
\Omega_{(A'[x] \to A[x]) / (R' \to R)} \cong \Omega_{(A' \to A) / (R' \to R)} \otimes_k k[x] \oplus A'[x]dx
\]
(see \stacksref{07HS}). For (\ref{crys_coprod1}), write $J = \ker(A' \to A)$, write $I = \ker(k\langle x \rangle \to k)$ and write $H = \ker(A'\langle x \rangle \to A)$. One observes first that
\[
\Fil^q \Omega^\bullet_{(k \langle x \rangle \to k) / (k \langle x \rangle \to k)} = I^{[q]}
\]
so it suffices to show 
\[
\underset{p + q \geq n}{\colim} J^{[p]} \otimes_k I^{[q]} = H^{[n]}
\]
holds in $\derD(k)$ for all $n \geq 0$. By a cofinality argument, one may reduce this to the finite colimit diagram 
\[
\underset{\substack{p + q \geq n \\ n \geq p \\ n \geq q}}{\colim} J^{[p]} \otimes_k I^{[q]} = H^{[n]}
\]
In this diagram, all objects are discrete $k$-modules and all maps are cofibrations for the standard model structure on chain complexes. Hence one may compute the colimit in the $1$-category of discrete $k$-modules, and reduce to the statement
\[
\underset{p + q \geq n}{\sum} J^{[p]}  I^{[q]} = H^{[n]}
\]
which is classical.
Finally for (\ref{crys_coprod3}), a computation shows that
\[
\Fil^q \Omega^\bullet_{(k \langle x \rangle \to k) / (k \to k)} = \begin{cases}
k[0] & q = 0 \\
0 & q > 0
\end{cases}
\]
hence it suffices to show the natural map
\[
\Fil^n \Omega^\bullet_{(A' \to A) / (R' \to R)} \to \Fil^n \Omega^\bullet_{(A'\langle x \rangle \to A) / (R' \to R)}
\]
is a quasi-isomorphism. This can be easily achieved by constructing an explicit homotopy between the composition 
\[
\Fil^n \Omega^\bullet_{(A'\langle x \rangle \to A) / (R' \to R)} \xrightarrow{x \mapsto 0} \Fil^n \Omega^\bullet_{(A' \to A) / (R' \to R)} \to \Fil^n \Omega^\bullet_{(A'\langle x \rangle \to A) / (R' \to R)}
\] 
and the identity map (see \cite[Theorem 6.13]{berthelot-yellow-book}). 
\end{proof}
\begin{lemma}[Filtered Poincar\'e lemma]\label{lem_poincare}
Let $(A' \to A) \in \PDPair_{k,\gen}$, and let $I = \ker(A' \to A)$. The natural map
\[
\Fil^p_\pdadic \Omega^\bullet_{(A' \langle x_1, \dots, x_n \rangle \to A) / (A' \to A)} \to I^{[p]}[0]
\]
is a quasi-isomorphism. 
\end{lemma}
\begin{proof}
By Lemma \ref{lem_extended_dr_colim}, we may (by factoring into coproducts) reduce to the case $n = 1$, which again can be easily done by constructing an explicit homotopy (see \cite[Theorem 6.13]{berthelot-yellow-book}). 
\end{proof}
The input to the classical crystalline cohomology functor is a PD-ring $(A, I, \gamma)$ and a morphism of rings $A / I \to R$. The following $\infty$-category, introduced by Mao \cite[p.49]{mao}, thus gives a natural input category for derived crystalline cohomology.
\begin{definition}For $k$ a ring, we define the $\infty$-category
\[
\CrysCon_k := \AniPDPair_k \times_{\CAlg_k^\an} \Fun(\Delta^1, \CAlg_k^\an)
\]
where the functor $\PDPair_k \to \CAlg_k^\an$ is given by $(A \to A') \mapsto A'$, and the functor $\Fun(\Delta^1, \CAlg_k^\an) \to \CAlg_k^\an$ is given by $(A \to B) \mapsto A$. 
\end{definition} 
By \cite[p. 52]{mao}, a set of compact projective generators is given by objects of the form
\[
\begin{tikzcd}
& k[x_1, \dots, x_n, y_1, \dots, y_m] \\
k\langle z_1, \dots, z_\ell \rangle[x_1, \dots, x_n] \rar& k[x_1, \dots, x_n] \uar 
\end{tikzcd}
\]
We write $\CrysCon_{k,\gen}$ for the full subcategory spanned by these objects. Note that the forgetful functor $\AniPDPair_k \to \Fun(\Delta^1, \CAlg_k^\an)$ induces a functor
\[
\CrysCon_k \to \Fun(\Delta^2, \CAlg_k^\an)
\]
sending $((A' \to A), (A \to R)) \mapsto (A' \to A \to R)$. 
\begin{lemma}\label{coprod_cryscon}
The functor
\[
\CrysCon_k \to \Fun(\Delta^2, \CAlg_k^\an)
\]
commutes with colimits.
\end{lemma}
\begin{proof}
By \cite[Lemma 5.4.5.5]{htt} it suffices to show the forgetful functor
\[
\AniPDPair_k \to \Fun(\Delta^1, \CAlg_k^\an)
\]
preserves colimits, which follows from Lemma \ref{forget_colim} and \cite[Proposition 3.34]{mao}.
\end{proof}

It is also shown in \cite[p. 52]{mao} that we have a forgetful functor 
\[
L \colon \Fun(\Delta^1, \AniPDPair_k) \to \CrysCon_k
\]
informally given by
\[
\begin{tikzcd}[column sep = small, row sep = tiny]
A' \arrow[->>]{r} & A\\
R' \uar \arrow[->>]{r} &R \uar
\end{tikzcd} \mapsto 
\begin{tikzcd}[column sep = small, row sep = tiny]
 & A\\
R' \arrow[->>]{r} &R \uar
\end{tikzcd}
\]
with right adjoint 
\[
R \colon \CrysCon_k \to \Fun(\Delta^1, \AniPDPair_k)
\]
informally given by
\[
\begin{tikzcd}[column sep = small, row sep = tiny]
 & A\\
R' \arrow[->>]{r} &R \uar
\end{tikzcd} \mapsto
\begin{tikzcd}[column sep = small, row sep = tiny]
A \arrow[->>]{r} & A\\
R' \uar \arrow[->>]{r} &R \uar
\end{tikzcd}
\]
The following might be somewhat surprising. 
\begin{lemma}\label{lem_R_colim}
The functor $R \colon \CrysCon_k \to \Fun(\Delta^1, \AniPDPair_k)$ preserves small colimits.
\end{lemma}
\begin{proof}
By Lemma \ref{coprod_cryscon} and \cite[Proposition 3.34]{mao} it suffices to show the functor
\[
\Fun(\Delta^2, \CAlg_k^\an) \to \Fun(\Delta^1 \times \Delta^1, \CAlg_k^\an) 
\]
sending
\[
(A \to B \to C) \mapsto \begin{tikzcd}[column sep = small, row sep = tiny]
C \rar & C\\
A \uar \rar &B \uar
\end{tikzcd}
\]
commutes with colimits, which follows directly from \cite[Proposition 5.1.2.3]{htt}.
\end{proof}
By the above lemma, the functor $R$ also admits a right adjoint. We hope to study this adjoint in future work, we believe it to be related to the functor $\GG_a^\#$ defined in \cite[Definition 2.4.1]{bhattgauges}. 
\begin{definition}
Let $k$ be a ring. We define the \emph{PD-adic filtered derived crystalline cohomology} functor
$\Crys_{- / (- \to -)} \colon \CrysCon \to \CAlg_\fil(k)$
as 
\[
\Crys_{- / (- \to -)} :=  \Fil_{\pdadic} \dR_{(- \to -) / (- \to -)} \circ R
 \]
We will write
\[
\LL^p_{- / (- \to -)} \colon \CrysCon \to \derD(k)
\]
for the $p$-th suspension of the $p$-th graded piece. 
\end{definition}
By definition, for any $(R \to A) \in \Fun(\Delta^1, \CAlg_k^\an)$ we have 
\[
\Crys_{A / (R \to R)} = \dR_{A / R}
\]
in $\CAlg_\fil(k)$. 
\begin{lemma}\label{lem_crys_colim}
The functor 
\[
\Crys_{- / (- \to -)} \colon \CrysCon \to \CAlg_\fil(k) 
\]
commutes with small colimits.
\end{lemma}
\begin{proof}
Combine Lemma \ref{lem_extended_dr_colim} with Lemma \ref{lem_R_colim}. 
\end{proof}
The following proposition basically states that the crystalline cohomology of $A$ over $(R' \to R)$ can be computed as the de Rham cohomology of a lift $A'$ over $R'$. 
\begin{prop}\label{prop_iso_crys}
Let $k$ be a ring, and 
\[
(R' \to R) \to (A' \to A) \in \Fun(\Delta^1, \AniPDPair_k)
\]
The unit of the adjunction $L \dashv R$ induces an equivalence
\[
\Fil_\pdadic \dR_{(A' \to A) / (R' \to R)} \to \Crys_{A / (R' \to R)}
\]
in $\CAlg_\fil(k)$.
\end{prop}
\begin{proof}
See also \cite[Proposition 4.16]{mao}, we give some more details. Since all functors involved commute with small colimits it suffices to show this for the four types of generators of $\Fun(\Delta^1, \AniPDPair_k)$ given by Lemma \ref{lem_fun_proj_gen}, i.e. we need to show the maps
\begin{align}
\Fil_\pdadic \dR_{(k\langle x \rangle \to k) / (k \langle x \rangle \to k)} &\to \Fil_\pdadic \dR_{(k \to k) / (k\langle x \rangle \to k)} \label{eq_qiso_1}\\
\Fil_\pdadic \dR_{(k\langle x \rangle \to k) / (k \to k)} &\to \Fil_\pdadic \dR_{(k \to k) / (k \to k)} \label{eq_qiso_2}\\
\Fil_\pdadic \dR_{(k[x] \to k[x]) / (k[x] \to k[x])} &\to \Fil_\pdadic \dR_{(k[x] \to k[x]) / (k[x] \to k[x])}\label{eq_qiso_3} \\
\Fil_\pdadic \dR_{(k[x] \to k[x]) / (k \to k)} &\to \Fil_\pdadic \dR_{(k[x] \to k[x]) / (k \to k)} \label{eq_qiso_4}
\end{align}
are filtered quasi-isomorphisms. Now (\ref{eq_qiso_3}) and (\ref{eq_qiso_4}) are evidently quasi-isomorphisms, and (\ref{eq_qiso_2}) is a quasi-isomorphism by Lemma \ref{lem_poincare}. To compute the right hand side of (\ref{eq_qiso_1}), we need to find a simplicial resolution of $(k \to k)$ over $(k\langle x \rangle \to k)$. A construction analogous to \cite[Construction 4.16]{simplicial-res} gives a simplicial resolution $C_\bullet \to (k \to k)$ over $(k\langle x \rangle \to k)$ with 
\[
C_n = (k \langle x, x_1, \dots, x_n \rangle \to k)
\]
If we write $I = \ker(k\langle x \rangle \to k)$, then Lemma \ref{lem_poincare} tells us the natural map
\[
\Fil^p_\pdadic \dR_{C_n / (k\langle x \rangle \to k)} \to I^{[p]}
\]
is a quasi-isomorphism for all $n$, which shows (\ref{eq_qiso_1}) is a quasi-isomorphism, establishing the result. 
\end{proof}
\subsection{Comparison with derived completions}\label{ss_comp_dr}
In this section we define for any surjection of rings $A \to A / I$ the \emph{derived completion} $\Comp(A \to A / I)$ also known as the \emph{Adams completion}, see \cite{bhatt-cddrc}. If $A$ is Noetherian this completion agrees with the usual completion, however for general $A$ it can be different. 

Moreover, Bhatt \cite[Remark 4.5]{bhatt-cddrc} shows that for any surjection of $\QQ$-algebras $A \to A / I$, there exists a canonical equivalence 
\[
\Comp(A \to A / I) \cong \widehat{\dR}_{(A / I) / A}
\]
in $\CAlg_\fil(k)$.

Before we can begin we need a good source category for the derived completion functor, which the definition below gives for $p = 1$. We will consider this in slightly bigger generality and consider a composition of multiple surjective ring maps, as we will need this later.

\begin{definition}
We denote with $\Fun(\Delta^p, \CAlg_k^\an)_\surj$ the full subcategory of $\Fun(\Delta^p, \CAlg_k^\an)$ consisting of objects $A_0 \to  \dots \to  A_p$ such that $\pi_0(A_0) \to \pi_0(A_i)$ is surjective for all $i$.
\end{definition}

By Corollary \ref{corol_gen_surj}, a set of compact projective generators can be described as follows. For $i \in \{0, \dots, p\}$, let $F_i \colon \Delta^p \to \Alg_k$ be the unique functor satisfying
\[
F_i(j) := \begin{cases}
k[x] & j \leq i \\
k  & j > i 
\end{cases}
\]
where the maps $k[x] \to k$ are given by $x \mapsto 0$, and all other maps are the identity. Then the set $S_p := \{F_0, \dots, F_p\}$ is a set of compact projective generators for $\Fun(\Delta^p, \CAlg_k^\an)_\surj$. For example, if $p = 2$ we have 
\[
S_2 = \left\{\begin{array}{c} k[x] \to k \to k, \\ k[x] \to k[x] \to k,\\ k[x] \to k[x] \to k[x]\end{array}\right\} 
\]
We shall write $\Fun(\Delta^p, \Poly_k)_{\surj, \gen}$ for the full subcategory spanned by coproducts of objects in $S_p$ (see Definition \ref{def_gen_surj}).
 
Following \cite{mao}, we will denote with $\Pair_k$ the $1$-category of surjections $R \to R'$ of (discrete) commutative $k$-algebras. We warn the reader that $\Pair_k$ is not compact $1$-projectively generated, however \cite[Lemma 3.7]{mao} shows there does exist a fully faithful embedding $\Pair_k \hookrightarrow \Fun(\Delta^1, \CAlg_k^\an)_\surj$. We will often abuse notation by writing (see also Definition \ref{def_anipair})
\[
\AniPair_k := \Fun(\Delta^1, \CAlg_k^\an)_\surj
\]
there is no chance for confusion as the left hand side is a priori not well-defined. 

Note that if $F \colon \Fun(\Delta^1, \CAlg_k^\an) \to \mathcal{D}$ preserves (sifted) colimits, then so does $F \colon \AniPair_k  \to \mathcal{D}$, by Lemma \ref{forget_colim}. 

\begin{definition}\label{aba}
We define the \emph{derived divided power envelope} functor
\[
(-)^\Lenv \colon \AniPair_k \to \AniPDPair_k 
\]
as the left Kan extension of the composition
\[
\Pair_{k,\gen} \xrightarrow{\mathrm{env}} \PDPair_{k,\gen} \subseteq \AniPDPair_k
\]
where $\mathrm{env}$ is the functor sending the surjective ring map 
\[
k[x_1, \dots, x_n, y_1, \dots, y_m] \twoheadrightarrow k[x_1, \dots, x_n]
\]
to the element $k[x_1, \dots, x_n]\langle y_1, \dots, y_m \rangle \twoheadrightarrow k[x_1, \dots, x_n]$ in $\PDPair_{k,\gen}$. 
\end{definition}
By \cite[Corollary 2.2]{mao} the derived divided power envelope admits a right adjoint $\AniPDPair_k \to \PDPair_k$ to which we will refer as the \emph{forgetful functor}. 

We now wish to discuss filtrations.

\begin{definition}\label{derived_fil_adic}
Let $(A \to A / I) \in \Pair_k$ where $A$ is a $k$-algebra and $I$ is an ideal of $A$. Then the rule $n \mapsto I^{n}[0]$ defines an object in the $\infty$-category $(\derD(k)_{\geq 0})_\fil = \Fun(\NN^\op, \derD(k)_{\geq 0})$. Using the multiplication on $A$ this defines a functor
\begin{align}
\Fil_{\adic} \colon \Pair_k &\to \CAlg_\fil(k)_{\geq 0} \label{eq_underived_fil}\\
(A \to A / I) &\mapsto \{I^{n}\}_n \nonumber
\end{align}
We define the \emph{derived adic filtration functor}
\begin{align*}
\derL \Fil_\adic \colon \AniPair_k \to \CAlg_\fil(k)_{\geq 0}
\end{align*}
as the left Kan extension of $\Fil_{\adic}$ restricted to $\Fun(\Delta^1, \Poly_k)_{\surj, \gen}$. We shall sometimes write
\[
I^{(n)} := \derL \Fil_\adic^n(A \to A / I)
\]
and refer to it as the \emph{derived $n$-th power} of $I$. 
\end{definition}

\begin{lemma}\label{lem_fil_colim}
The functor
\[
\derL \Fil_\adic \colon \AniPair_k \to \CAlg_\fil(k)_{\geq 0}
\]
preserves small colimits.
\end{lemma}
\begin{proof}
By Proposition \ref{left_kan} it suffices to show the map
\[
\Fil_{\adic} \colon \Pair_{k, \gen} \to \CAlg_\fil(k)
\]
commutes with coproducts. By an induction argument, we thus need to show that for any $t \in \NN$ and any $(P \to Q) \in \Pair_{k, \gen}$ the natural maps
\begin{align*}
\Fil^t (\Fil^\bullet_\adic(P \to Q) \otimes \Fil^\bullet_\adic (k[z] \to k)) \to \Fil^t_\adic(P[z] \to Q) \\
\Fil^t (\Fil^\bullet_\adic(P \to Q) \otimes \Fil^\bullet_\adic (k[z] \to k[z])) \to \Fil^t_\adic(P[z] \to Q[z])
\end{align*}
are equivalences. We shall only give a proof for the first map, the second is similar but easier. Write $P = k[x_1, \dots, x_n, y_1, \dots, y_m]$ and $Q = k[x_1, \dots, x_n]$ so that the map $P \to Q$ is given by $y_i \mapsto 0$. 
By Lemma \ref{lem_alg_coprod} it suffices to show
\[
\underset{p + q \geq t}\colim \ (y_1, \dots, y_m)^p \otimes_k (z)^q \to (y_1, \dots, y_m, z)^t
\]
is an equivalence in $\derD(k)$. By a similar argument as in the proof of Lemma \ref{lem_extended_dr_colim}, one may reduce this to showing that
\[
\sum_{p + q \geq t} (y_1, \dots, y_m)^p \cdot (z)^q = (y_1, \dots, y_m, z)^t
\]
as discrete ideals in $P[z]$, which we leave for the reader to verify.
\end{proof}
\begin{remark}
Explicitly, using \cite[Lemma 5.5.8.13]{htt} and Corollary \ref{corol_gen_surj} one can show that if $k$ is a commutative $\QQ$-algebra, one may represent any object $(A \to A / I) \in \AniPair_k$ by a simplicial ring $A_*$ with a simplicial ideal $I_*$, such that $A_i$ is a polynomial algebra for all simplicial degrees $i$, and $I_i \subseteq A_i$ is generated by a subset of the generators of $A_i$. Then $\derL \Fil_{\adic}^n (A \to A / I)$ is isomorphic to the image of $I_*^n$ in $\derD(k)$ under the Dold-Kan correspondence.
\end{remark}

\begin{construction}\label{cons_comp_map_adfil}
Denote with $\tau_{\leq 0} \colon \CAlg_\fil(k) \to \CAlg_\fil(k)$ the truncation map induced by the t-structure on $\derD(k)$. Write $\tau_{\leq 0} \derL \Fil_\adic = \tau_{\leq 0} \circ \derL \Fil_\adic$. The natural transformation $\id \to \tau_{\leq 0}$ of functors $\CAlg_\fil(k) \to \CAlg_\fil(k)$ induces a natural transformation
\[
\derL \Fil_\adic \to \tau_{\leq 0} \derL \Fil_\adic
\]
of functors $\AniPair_k \to \CAlg_\fil(k)$. For any $(A \to B) \in \Pair_k$ one has 
\[
\tau_{\leq 0} \Fil_\adic(A \to B) = \Fil_\adic(A \to B)
\]
and thus we get canonical natural transformations
\[
\derL \Fil_\adic \to \tau_{\leq 0} \derL \Fil_\adic \to \Fil_\adic
\]
of functors $\Pair_k \to \CAlg_\fil(k)$.
\end{construction}
\begin{example}\label{example_derived_ideal_powers}We warn the reader that even in characteristic $0$, for a general surjective map $A \to A / I$ of discrete rings, it is not generally the case that $\tau_{\leq 0}\derL \Fil_\adic = \Fil_\adic$. For example, let $k$ be a ring of characteristic $0$, $A = k[x] / (x^2)$ and $I = (x)$. Applying the resolution from \cite[Construction 4.16]{simplicial-res} to the regular element $(t - x^2) \in k[t, x]$ one obtains a simplicial resolution for $k[x] / (x^2)$. Using this resolution, one may show that $\tau_{\leq 0}(\derL \Fil^n_{\adic}) \neq 0$ for all $n \geq 0$, even though $I^{n} = 0$ for $n \geq 2$.
\end{example}

We now give a divided power analogue of Definition \ref{derived_fil_adic}. 
\begin{definition}\label{derived_pdadic_fil}
Let $A$ be a $k$-algebra, $I$ an ideal of $A$, and $\gamma$ a PD-structure on $I$, so that $(A \to A/I) \in \PDPair_k$. Then the rule $n \mapsto I^{[n]}[0]$ defines an object in $(\derD(k)_{\geq 0})_\fil = \Fun(\NN^\op, \derD(k)_{\geq 0})$. Using the multiplication on $A$ this defines a functor
\begin{align}
\Fil_{\pdadic} \colon \PDPair_k &\to \CAlg_\fil(k)_{\geq 0} \label{eq_underived_pdfil}\\
(A \to A / I) &\mapsto \{I^{[n]}\}_n \nonumber
\end{align}
We define the \emph{derived PD-adic filtration functor}
\begin{align*}
\derL \Fil_\pdadic \colon \AniPDPair_k \to \CAlg_\fil(k)_{\geq 0}
\end{align*}
as the left Kan extension of $\Fil_{\pdadic}$ restricted to the subcategory ${\PDPair_{k,\gen}}$. 
\end{definition}
\begin{lemma}\label{lem_pdfil_colim}
The functor
\[
\derL \Fil_\pdadic \colon \AniPDPair_k \to \CAlg_\fil(k)_{\geq 0}
\]
preserves small colimits.
\end{lemma}
\begin{proof}
Analogous to the proof of Lemma \ref{lem_extended_dr_colim}. 
\end{proof}
We shall sometimes abuse notation by writing 
\[
\derL \Fil_{\adic} \colon \AniPDPair_k \to \CAlg_\fil(k)
\]
for the composition
\[
\AniPDPair_k  \to \AniPair_k \xrightarrow{\derL \Fil_{\adic} } \CAlg_\fil(k)
\]
Note that we have a natural transformation 
\begin{equation} \label{abb}
\derL \Fil_{\adic} \to \derL \Fil_{\pdadic}
\end{equation}
 of functors $\AniPDPair_k \to \CAlg_\fil(k)$ induced by the inclusion $I^\ell \subseteq I^{[\ell]}$ on $\PDPair_{k, \gen}$. 
\begin{lemma}
If $k$ is a $\QQ$-algebra, the natural functor
\[
\derL \Fil_{\adic} \to \derL \Fil_{\pdadic}
\]
is an equivalence of functors $\AniPDPair_k \to \CAlg_\fil(k)$.
\end{lemma}
\begin{proof}
It suffices to check this on $\PDPair_{k, \gen}$ where the statement is obvious.
\end{proof}
Note that for $(A \to A_0) \in \AniPair_k$, the unit of the adjunction 
\[
(-)^\Lenv \vdash \forget
\]
induces a canonical map 
\[
\derL\Fil_\adic(A \to A_0) \to \derL\Fil_\adic((A \to A_0)^\Lenv)
\] 
\begin{lemma}\label{compare_01}
Let $k$ be a ring, and $(A \to A_0) \in \AniPair_k$. Then the composition
\begin{equation} \label{abc}
\derL\Fil_\adic(A \to A_0) \to \derL\Fil_\adic((A \to A_0)^\Lenv) \to \derL\Fil_\pdadic((A \to A_0)^\Lenv)
\end{equation}
induces an equivalence
\[
\derL \gr^i_{\adic}(A \to A_0) \xrightarrow{\sim}  \derL \gr^i_\pdadic((A \to A_0)^\Lenv)
\]
for $i = 0, 1$. 
\end{lemma}
\begin{proof}
Since all functors commute with colimits it suffices to check this for the elements $k[x] \to k[x]$ and $k[x] \to k$. The only nontrivial thing to check is that the map
\[
(x) / (x^2) \to \gr^1_{\pdadic}(k\langle x \rangle \to k)
\]
is an isomorphism in $\derD(k)^\heart$, which we leave for the reader to verify.
\end{proof}

\begin{construction}\label{construction_dr_adic_compare}
Let $k$ be a ring. Let $(A \to A_0) \in \AniPair_k$. Then the maps 
\[
(A \to A) \to (A \to A_0) \to (A_0 \to A_0)
\]
in $\AniPair_k$ induce maps
\[
(A \to A) \to (A \to A_0)^\Lenv \to (A_0 \to A_0)
\]
in $\AniPDPair_k$. We thus get maps
\begin{align*}
\dR_{A_0 / A} &\xleftarrow{\sim} \dR_{(A \to A_0)^\Lenv / (A \to A)} \\&\rightarrow \dR_{(A \to A_0)^\Lenv / (A \to A_0)^\Lenv} \\&\cong \derL \Fil_\pdadic ((A \to A_0)^\Lenv)
\end{align*}
where the first arrow is an equivalence by Proposition \ref{prop_iso_crys}. Inverting the first arrow, we obtain a map
\begin{equation} \label{drpdcompare}
\dR_{A_0 / A} \to \derL \Fil_\pdadic ((A \to A_0)^\Lenv)
\end{equation}
in $\CAlg_\fil(k)$, functorial in $(A \to A_0) \in \AniPair_k$. 

%

\end{construction}

\begin{prop}\label{surj_dr_adfil}
Let $k$ be a ring. Then the map  (\ref{drpdcompare}) induces an equivalence
\[
\dR_{- / -} \xrightarrow{\sim} \derL \Fil_{\pdadic}((- \to -)^\Lenv)
\] 
of functors $\AniPair_k \to \CAlg_\fil(k)$. 
\end{prop}
\begin{proof}
See \cite[Proposition 4.64]{mao}.
\end{proof}

\begin{corollary}
Let $k$ be a ring, and $(A \to A_0) \in \AniPair_k$. The natural map $A \to \dR_{A_0 / A}$ induces an equivalence
\[
A / \derL\Fil^2_\adic(A \to A_0) \to \dR_{A_0 / A} / \Fil^2
\]
in $\CAlg_\fil(k)$. 
\end{corollary}
\begin{proof}
Combine Proposition \ref{surj_dr_adfil} and Lemma \ref{compare_01}.
\end{proof}

If $k$ is a ring such that $\QQ \subseteq k$, the categories $\AniPDPair_k$ and $\AniPair_k$ are canonically isomorphic, and we get the following result.
\begin{corollary}
Let $k$ be a ring such that $\QQ \subseteq k$. Then the map (\ref{drpdcompare}) induces an equivalence
\[
\dR_{- / -} \xrightarrow{\sim} \derL \Fil_{\adic}(- \to -)
\]
of functors $\AniPair_k \to \CAlg_\fil(k)$. 
\end{corollary}

\begin{definition}
We define the \emph{derived completion} functor
\[
\Comp(- \to -) \colon \Fun(\Delta^1, \CAlg_k^\an)_\surj \to \CAlg_\fil(k)
\]
as the composition
\[
\Fun(\Delta^1, \CAlg_k^\an)_\surj \xrightarrow{\derL \Fil_\adic^\bullet} \CAlg_\fil(k) \xrightarrow{\widehat{(-)}}  \CAlg_\fil(k) 
\]
We will refer to the filtration on $\Comp(A \to A / I)$ as the \emph{derived adic filtration}.
\end{definition}
One may think of $\Comp(A \to A / I)$ as the filtered $\EE_\infty$-algebra whose $p$-filtered piece is given by
\[
\lim_{n \to \infty} \cofib\(I^{(n)} \to  I^{(p)}\)
\]
using the suggestive notation from Definition \ref{derived_fil_adic}. 

\begin{remark}\label{rem_compute_derived_comp}
More explicitly, one may compute $\Comp(A \to A / I)$ as follows. Start by taking a simplicial ring $A_*$ with a simplicial ideal $I_*$ such that for all simplicial degrees $i$, $A_i$ is a polynomial algebra, $I_i \subseteq A_i$ is generated by a subset of the generators of $A_i$, and 
\[
A \to (A / I) = \underset{i \in \Delta^\op}{\colim}\  A_i \to A_i / I_i
\]
in $\AniPair_k$. 

Then
\[
\Comp(A \to A / I) = \lim_{n \to \infty}  \underset{i \in \Delta^\op} \colim \ A_i / I_i^n
\]
in $\CAlg_k^\an$.
\end{remark}
\begin{remark}
Using Proposition \ref{surj_dr_adfil},  \cite[Proposition 8.5]{quillen-rings} and \cite[Corollary 10.4(iii)]{quillen-rings} one may in fact show that if $k$ is of characteristic zero, $A_*$ is a simplicial ring and $I_*$ is a (termwise) quasi-regular ideal, then the comparison map
\[
\Comp\(\underset{i \in \Delta^\op}{\colim}\  A_i \to A_i / I_i\) \to \lim_{n \to \infty} \underset{i \in \Delta^\op}{\colim}\ A_i \to A_i / I_i^n
\]
is an equivalence in $\CAlg_k^\an$. In particular, if $A$ is a discrete ring and $I \subseteq A$ is a quasi-regular ideal, the derived completion agrees with the usual completion. 
\end{remark}
\begin{corollary}\label{surj_dr_comp}
If $\QQ \subseteq k$, the natural transformation
\[
\widehat{\dR}_{-/-} \to \Comp(- \to -)
\]
of functors
\[
\AniPair_k \to \CAlg_\fil(k)
\]
induced by Construction \ref{construction_dr_adic_compare} is an equivalence.
\end{corollary}
\begin{proof}
Follows directly by applying Proposition \ref{surj_dr_adfil}.
\end{proof}

\section{Chern classes in derived de Rham cohomology}
Let $k$ be a ring. Let $X$ be a scheme over $k$ with the resolution property (\stacksref{0F85}). In this section, we define for any $\mathcal{E} \in \Perf(X)$ and any $i \geq 0$ a Chern class
\[
c_i(\mathcal{E}) \in \coh^{2i}(\Fil^i \widehat{\dR}_{X / k})
\]
Moreover, if $i!$ is invertible in $k$, we will construct the $i$th part of the Chern character
\[
\ch_i(\mathcal{E}) \in \coh^{2i}(\Fil^i \widehat{\dR}_{X / k})
\]
Then we will show our construction is uniquely determined by additivity, functoriality and its value on line bundles, see Proposition \ref{prop_chern_unique}. 

We mostly follow the approach of Bhatt and Lurie from \cite[\S7, \S 9.2]{bhatt-lurie}. To avoid needing to introduce syntomic cohomology, we adapt their construction of higher Chern classes to the case of derived de Rham cohomology.  The first three sections merely provide results needed to make the machinery work, the entire construction is contained in the last section. 
\label{sec_drc_stacks}
\subsection{Relative derived de Rham cohomology}
In this section, we define derived de Rham cohomology relative to an open subset, and construct a cup product map for relative derived de Rham cohomology in a very general way.
\begin{definition}
Let $k$ be a ring, let $X$ be an algebraic stack over $k$, and let $U \subseteq X$ be an open substack. Then we define the \emph{relative derived de Rham cohomology}
\[
\widehat{\dR}_{(X, U) / k} := \fib(\widehat{\dR}_{X / k} \to \widehat{\dR}_{U / k})
\]  
in $\CAlg_\fil(k)$. 
\end{definition}
The limit exists by \cite[Proposition 3.2.2.1]{ha} and can be computed in $\derD(k)_\fil$. Note that since this is a filtered algebra, we immediately get a canonical map
\[
\Fil^p \widehat{\dR}_{(X, U) / k} \otimes_k \Fil^q \widehat{\dR}_{(X, U) / k} \to \Fil^{p + q} \widehat{\dR}_{(X, U) / k}
\]
The remainder of this section is devoted to constructing a map 
\[
\Fil^p \widehat{\dR}_{(X, U) / k} \otimes_k \Fil^q \widehat{\dR}_{(X, V) / k} \to \Fil^{p + q} \widehat{\dR}_{(X, U \cup V) / k}
\]
for different open substacks $U, V \subseteq X$. We start with a lemma.

\begin{lemma}\label{pushout}
Let $k$ be a ring. Let $X$ be an algebraic stack over $k$, and let $U, V \subseteq X$ be open substacks which cover $X$ (i.e. the map $U \coprod V \to X$ is an effective epimorphism in $\St_k$). Then the diagram
\[
\begin{tikzcd}
\dar U \times_X V \rar& U \dar \\
V \rar& X 
\end{tikzcd}
\]
is a pushout diagram in $\St_k$. 
\end{lemma}
\begin{proof}
We thank Dhyan Aranha for teaching us the following argument. We prove the following more general statement: For any relative $n$-strack $X \to Y$ and $U, V \subseteq X$ open relative $n$-substacks which cover $X$, the diagram
\[
\begin{tikzcd}
\dar U \times_X V \rar& U\dar  \\
V \rar& X
\end{tikzcd}
\] 
is a pushout square.

If $X$ is an algebraic space, then this follows from the analogous statement in the category of sheaves of sets. Now suppose the statement is known for any relative $(n - 1)$-stack, and pick a relative $n$-stack $f \colon X \to Y$. Let $U, V \subseteq X$ be relative $n$-substacks which cover $X$. Let $Z$ be defined by the pushout diagram
\[
\begin{tikzcd}
\dar U \times_X V \rar& U\dar  \\
V \rar& Z
\end{tikzcd}
\]
It suffices to show the natural map $Z \to X$ is an open immersion. Let $T \to X$ be a smooth cover of $X$ which is a relative $(n - 1)$-stack. Since $T \to X$ is a smooth cover, it suffices to show the map $Z_T \to T$ is an open immersion. But $Z_T$ fits in a pushout diagram
\[
\begin{tikzcd}
\dar U_T \times_T V_T \rar& U_T \dar  \\
V_T \rar& Z_T
\end{tikzcd}
\]
so the result follows by the induction hypothesis.
\end{proof}

\begin{prop}\label{prop_fib_stable_infty}
Let $\mathcal{C}^\otimes$ be a symmetric monoidal stable $\infty$-category. Let $X$ be any $\infty$-topos, and let $\mathcal{F}, \mathcal{G}$ be $\mathcal{C}$-valued sheaves on $X$. Let $U, V \in X$, and define $U \cup V$ as the unique object sitting in a pushout square 
\[
\begin{tikzcd}
\dar U \times_X V \rar & U \dar \\
V \rar& U \cup V
 \end{tikzcd}
\] 
in $X$. Write
\[
\mathcal{F}(X, U) := \fib(\mathcal{F}(X) \to \mathcal{F}(U))
\]
Then there exists a map
\[
\mathcal{F}(X, U) \otimes \mathcal{G}(X, V) \to (\mathcal{F} \otimes \mathcal{G})(X, U \cup V)
\]
fitting in a commutative diagram
\[
\begin{tikzcd}
 \dar \mathcal{F}(X, U) \otimes \mathcal{G}(X, V) 
 \rar& \dar (\mathcal{F} \otimes \mathcal{G})(X, U \cup V) \\ 
\mathcal{F}(X) \otimes \mathcal{G}(X) \rar& (\mathcal{F} \otimes \mathcal{G})(X)
\end{tikzcd}
\]
in $\mathcal{C}$.
\end{prop}
\begin{proof}
Define $T \in \mathcal{C}$ as the unique object fitting in a cartesian diagram
\[
\begin{tikzcd}
T  \dar \arrow[dr, phantom, "\square"]  \rar& \mathcal{F}(X) \otimes \mathcal{G}(V) \dar \\ 
 \mathcal{F}(U) \otimes \mathcal{G}(X) \rar&  \mathcal{F}(U) \otimes \mathcal{G}(V)
\end{tikzcd}
\]
We then have a diagram
\begin{equation}
\begin{tikzcd}
T  \dar \rar& \mathcal{F}(V) \otimes \mathcal{G}(V) \dar \\ 
 \mathcal{F}(U) \otimes \mathcal{G}(U) \rar&  \mathcal{F}(U \times_X V) \otimes \mathcal{G}(U  \times_X V)
\end{tikzcd}\label{diagramt}
\end{equation}
By the sheaf property of $\mathcal{F} \otimes \mathcal{G}$ we have a pullback diagram
\[
\begin{tikzcd}
(\mathcal{F} \otimes \mathcal{G})(U \cup V)  \dar \rar& (\mathcal{F} \otimes \mathcal{G})(V) \dar \\ 
( \mathcal{F} \otimes \mathcal{G})(U) \rar&  (\mathcal{F} \otimes \mathcal{G})(U  \times_X V)
\end{tikzcd}
\]
Thus the diagram (\ref{diagramt}) induces a map 
\[
T \to (\mathcal{F} \otimes \mathcal{G})(U \cup V)
\]
fitting in a commutative diagram
\[
\begin{tikzcd}
\dar \mathcal{F}(X) \otimes \mathcal{G}(X) \rar& T \dar \\
(\mathcal{F} \otimes \mathcal{G})(X) \rar& (\mathcal{F} \otimes \mathcal{G})(U \cup V)
\end{tikzcd}
\]
By Lemma \ref{lem_fib_stable_infty} the canonical map
\[
\mathcal{F}(X, U) \otimes \mathcal{G}(X, V) \to \fib(\mathcal{F}(X) \otimes \mathcal{G}(X) \to T)
\]
is an equivalence, the result follows.
\end{proof}
\begin{corollary}\label{prop_relative_map}
Let $X$ be an algebraic stack, and let $U, V \subseteq X$ be open substacks. Then there exists a map
\[
\Fil^p \widehat{\dR}_{(X, U) / k} \otimes \Fil^q \widehat{\dR}_{(X, V) / k} \to \Fil^{p + q} \widehat{\dR}_{(X, U \cup V) / k}
\]
fitting in a commutative diagram
\[
\begin{tikzcd}
\Fil^p \widehat{\dR}_{(X, U) / k} \otimes \Fil^q \widehat{\dR}_{(X, V) / k} \dar \rar& \dar \Fil^{p + q} \widehat{\dR}_{(X, U \cup V) / k} \\
\Fil^p \widehat{\dR}_{X / k} \otimes \Fil^q \widehat{\dR}_{X / k} \rar& \Fil^{p + q} \widehat{\dR}_{X / k}
\end{tikzcd}
\]
in $\derD(k)$.
\end{corollary}
\begin{proof}
Since $\Fil^p \widehat{\dR}$ is a sheaf for any $p$ by Proposition \ref{prop_dr_sheaf}, we may apply Lemma \ref{pushout}  and Proposition \ref{prop_fib_stable_infty} to find a map
\[
\Fil^p \widehat{\dR}_{(X, U) / k} \otimes \Fil^q \widehat{\dR}_{(X, V) / k} \to \(\Fil^{p} \widehat{\dR}_{(X, U \cup V) / k} \otimes \Fil^{q} \widehat{\dR}_{(X, U \cup V) / k}\)
\]
The result follows by composing with the multiplication map.
\end{proof}
\subsection{The classifying stack of short exact sequences}
\label{section_torsors}
For $n \in \NN$, denote with denote with $\GL_{n}$ the group scheme of invertible matrices over $\Spec(\ZZ)$, that is
\[
\GL_{n} := \Spec\(\ZZ[x_{ij} \mid i, j \in \{1, \dots, n\}]_{\det(x_{ij})}]\)
\]
The following identifies $\mathrm{B}\GL_{n}$ as the classifying stack of vector bundles of rank $n$.
\begin{lemma}\label{prop_vectngln}
Let $k$ be a ring and let $X$ be a stack over $k$. Then there exists an equivalence of spaces
\[
\Vect_n(X)^\simeq \simeq \Map_{\St_k}(X, \mathrm{B} \GL_n)
\]
functorial in $X$.
\end{lemma}
\begin{proof}
Since both sides are sheaves on $\St_k$ (see Lemma \ref{vect_sheaf2}), it suffices to construct the equivalence of functors when restricted to  $(\Alg_k)^\op$. In this case, both sides are $1$-groupoids, and the construction is classical (see e.g. \cite[Example 4.32]{khanstacks}). 
\end{proof}
We now wish to generalize the above to short exact sequences of vector bundles. We start by introducing the corresponding group scheme. For $m, n \in \NN$, we denote $\matP_{n, m}$ the group scheme of upper triangular block matrices in $\GL_{n + m}$. Explicitly
\[
\matP_{n, m} := \Spec\(\ZZ[x_{ij} \mid i, j \in \{1, \dots, n + m\}]_{\det(x_{ij})}] / I_{n,m} \)
\]
where $I_{n,m}$ is the ideal generated by all $x_{ij}$ with $i > n$ and $j \leq n$.  

For any stack $X$ over $k$, we denote with $\mathrm{ExtVect}_{n, m}(X)$ the category of exact triangles 
\[
A \to B \to C \xrightarrow{+1}
\]
in $\derD(X)$ such that $A \in \Vect_n(X)$, $B \in \Vect_{n + m}(X)$ and $C \in \Vect_m(X)$.
\begin{lemma}\label{sheaf_vects}
The functor
\[
\mathrm{ExtVect}_{n, m}(-)^\simeq \colon \St_k^\op \to \mathcal{S}
\]
is a sheaf for the fppf topology.
\end{lemma}
\begin{proof}
For any $k$-stack $X$ and any $\mathcal{F} \in \Fun(\Delta^1 \times \Delta^1, \derD(X))$, the condition that $\mathcal{F}(i, j)$ is a vector bundle of a certain rank for some $(i, j) \in \{0, 1\}^2$  is local for the flat topology. Moreover, the condition that $\mathcal{F}$ is a pullback square is also local for the flat topology. It follows that the functor
\[
\mathrm{ExtVect}_{n, m} \colon \St_k^\op \to \Cat_\infty
\]
is a sheaf for the fppf topology. The result follows after observing that $(-)^\simeq$ preserves limits.
\end{proof}

\begin{prop}\label{prop_bpnm}
Let $k$ be a ring and let $X$ be a stack over $k$. Then there exists an equivalence of spaces
\[
\mathrm{ExtVect}_{n, m}(X)^\simeq \simeq \Map_{\St_k}(X, \mathrm{B} \matP_{n, m})
\]
functorial in $X$.
\end{prop}
\begin{proof}
Again, since both sides are sheaves on $\St_k$ by Lemma \ref{sheaf_vects}, it suffices to construct the equivalence of functors when restricted to  $(\Alg_k)^\op$, in which case both sides are $1$-groupoids. Consider the standard short exact sequence 
\[
S := \left[ 0 \to \catO_X^{\oplus n} \to \catO_X^{\oplus(n + m)} \to \catO_X^{\oplus m} \to 0\right]
\]
on $X$. Since any short exact sequence $T \in \Vect_{n,m}(X)$ is locally isomorphic to $S$, and moreover $\mathrm{Aut}(S) \simeq \matP_{n, m}(\catO_X)$, the result follows.
\end{proof}
\subsection{A result of Totaro}
In this section, we slightly adapt a theorem on the de Rham cohomology classifying spaces that is originally due to Totaro. The motivation for this excursion is to provide a crucial ingredient for the Cartan formula for Chern classes in derived de Rham cohomology in the next section. We start by recalling the result, see \cite[Theorem 6.1]{totaro} for the proof.
\begin{theorem}[Totaro]\label{thm_totaro}
Let $k$ be a field, and let $P$ be a parabolic subgroup of a reductive group $G$ over a field $k$. Let $L$ be the Levi quotient of $\matP$. Then the restriction 
\[
\derR \Gamma(\mathrm{B}P , \Omega^j) \to \derR \Gamma(\mathrm{B}L, \Omega^j)
\]
is an equivalence for all $j$. 
\end{theorem}
We now wish to extend this result slightly, and show $k$ can be any ring. 
\begin{lemma}
Let $A \in \derD(\ZZ)$. If $A \otimes \QQ \cong 0$ and $A \otimes (\ZZ / p \ZZ) \cong 0$ for all prime numbers $p$, then $A \cong 0$. 
\end{lemma}
\begin{proof}
This is well known, we follow \cite[Lemma E.9.3.1]{sag}. By induction on the number of prime divisors of $n$, one may first show that $A \otimes_\ZZ (\ZZ / n \ZZ) \cong 0$ for all $n \in \NN$, by choosing a prime divisor $p \mid n$ and tensoring the exact triangle
\[
(\ZZ / p \ZZ) \to (\ZZ / n \ZZ) \to (\ZZ / (n / p) \ZZ) \xrightarrow{+1}
\] 
with $A$. 

Next, note that we have
\[
\QQ / \ZZ = \underset{N \in \NN} \colim  \ \ZZ / N \ZZ 
\]
in $\derD(\ZZ)^\heart$. Since taking homology of chain complexes commutes with filtered colimits, it follows that the above equality also holds in $\derD(\ZZ)$. Commuting tensor products with colimits it follows that
\[
A \otimes_\ZZ \QQ / \ZZ\cong 0
\]
in $\derD(\ZZ)$. The result now follows by tensoring the exact triangle
\[
\ZZ \to \QQ \to (\QQ / \ZZ) \xrightarrow{+1}
\]
with $A$. 
\end{proof}
\begin{corollary}\label{corol_qiso_z}
If $f \colon A \to B$ is a morphism in $\derD(\ZZ)$ such that $f \otimes \QQ$ is an equivalence,  and $f \otimes_\ZZ (\ZZ / p \ZZ)$ is an equivalence for all primes $p$, then $f$ is an equivalence. 
\end{corollary}
\begin{proof}
Follows by applying the previous lemma to the cone of $f$. 
\end{proof}
With this, we can state and prove the version of Totaro's theorem that we need.
\begin{prop}\label{corol_totaro}
Let $k$ be a ring. Let $m, n \in \NN$. Write $\matP = \matP_{n, m}$, and $\mathrm{L} = \GL_n \times \GL_m$. Then for any $p \geq 0$, the map 
\[
\Fil^p \widehat{\dR}_{\mathrm{BP}_k / k} \to \Fil^p \widehat{\dR}_{\mathrm{BL}_k / k}
\]
is an equivalence in $\derD(k)$. 
\end{prop}
\begin{proof}
Note that $\mathrm{BP}$ and $\mathrm{BL}$ can be written as a limit of smooth affine stacks (over a cosimplicial diagram), so by Lemma \ref{dr_iso_p} it suffices to show that the natural map 
\[
\derR\Gamma(\mathrm{BP}_k, \Omega^{\geq p}_{- / k}) \to \derR \Gamma(\mathrm{BL}_k, \Omega^{\geq p}_{- / k})
\]
is a quasi-isomorphism. By the (convergent) spectral sequence for hypercohomology it suffices to show that 
\[
\derR \Gamma(\mathrm{BP}_k, \Omega^i_{- / k}) \to\derR \Gamma(\mathrm{BL}_k, \Omega^i_{- / k})
\]
is a quasi-isomorphism for all $i$. Note that $\mathrm{BP}$ is perfect by \cite[Corollary 3.22]{bzfrna}. Moreover, $\mathrm{BP}$ is smooth over $\Spec(\ZZ)$, hence $\mathrm{BP}_k$ is the derived pullback 
\[
\mathrm{BP} \times_{\Spec(\ZZ)} \Spec(k)
\]
computed in $\St_\ZZ$. Finally the sheaf $\Omega^i_{- / \ZZ}$ on $\mathrm{BP}$ is flat since $\mathrm{BP}$ is smooth over $\ZZ$, hence the derived pullback is equal to $\Omega^i_{- / k}$. Thus, by \cite[Proposition 3.10]{bzfrna} the base change map
\[
\derR \Gamma(\mathrm{BP}, \Omega^i_{- / \ZZ}) \otimes_\ZZ^{\derL} k \to \derR \Gamma(\mathrm{BP}_k, \Omega^i_{- / k})
\]
is an equivalence (similarly for $\mathrm{BL}$), so it suffices to prove the result in the case $k = \ZZ$. This follows directly from Corollary \ref{corol_qiso_z} and Theorem \ref{thm_totaro}. 
\end{proof}
\subsection{Chern classes in derived de Rham cohomology}
In this section we define Chern classes in Hodge-completed derived de Rham cohomology, and show that they satisfy the usual axioms. To avoid having to introduce syntomic cohomology, we adapt the results from \cite[\S 9,2]{bhatt-lurie} from the syntomic case. 
\begin{definition}
Define the functor
\begin{align*}
\GG_m \colon \Alg_k &\to \derD(\ZZ)^\heart \subseteq \derD(\ZZ) \\
R &\mapsto R^\times
\end{align*}
By right Kan extension to $\St_k$ and sheafification, we get an induced functor
\[
\derR\Gamma(-, \GG_m) \in \Shv_{\fppf}(\St_k, \derD(\ZZ))
\]
\end{definition}
Since $\GG_m$ is smooth, for any scheme $X$ one has
\[
\derR\Gamma(X, \GG_m)  = \derR \Gamma(X_{\mathrm{et}}, \GG_m)
\]
where the right hand side denotes the cohomology of $\GG_m$ on the étale site of $X$. Moreover, one has a canonical isomorphism of spaces
\[
\Pic(X) \simeq \tau_{\geq 0}(\derR\Gamma(X, \GG_m)[1])
\]
for any $k$-stack $X$ (here we think of the right hand side of a space via the Dold-Kan construction). 

\begin{lemma}\label{gm_extend}
The functor 
\begin{align*}
\GG_m \colon \Alg_k &\to \Ab \subseteq \derD(\ZZ) \\
R & \mapsto R^\times
\end{align*}
is the left Kan extension from its restriction to smooth $k$-algebras.
\end{lemma}
\begin{proof}
Since
\[
\GG_m(R) = \Hom_{\Alg_k}(k[x, \frac{1}{x}], R)
\] 
this is a consequence of the following much more general statement: If $\mathcal{C}$ is any category, $\mathcal{C}_0 \subseteq \mathcal{C}$ is a full subcategory and $X \in \mathcal{C}_0$, the left Kan extension of the functor $\Hom_{\mathcal{C}_0}(X, -)$ along the inclusion is given by $\Hom_{\mathcal{C}}(X, -)$. 
\end{proof}
\begin{construction}[First Chern class for line bundles]
Let $k$ be a ring. For any ring $R$ which is smooth over $k$, we have a commutative diagram
\[
\begin{tikzcd}
R^\times \dar{\mathrm{d} \log} \rar& 0 \dar  \rar& \dots \\
\Omega_{R / k} \rar& \Omega^2_{R / k} \rar& \dots  
\end{tikzcd}
\]
defining a functor
\begin{align*}
\Alg_k &\to \Fun(\Delta^1, \derD(\ZZ)) \\
R &\mapsto (R^\times[-1] \to \Fil^1 \Omega^\bullet_{R / k})
\end{align*}
By Construction \ref{dr_iso_z}, we thus get for any smooth ring $R$ over $k$ a canonical map
\begin{align*}
\GG_m(R)[-1] \to \Fil^1 \widehat{\dR}_{R / k}
\end{align*}
functorial in $R$. By Lemma \ref{gm_extend} the functor $\GG_m$ is the left Kan extension of its restriction to smooth $k$-algebras. By \cite[Proposition 4.3.2.17]{htt} we thus get a canonical map
\[
\GG_m(R)[-1] \to \Fil^1 \widehat{\dR}_{R / k}
\]
functorial in $R \in \Alg_k$. 

By Proposition \ref{prop_dr_sheaf}, the functor $R \mapsto \Fil^1 \widehat{\dR}_{R / k}$ is a sheaf and thus
we get a natural map
\[
\csmall_1^{\widehat{\dR}} \colon  \Gamma(\Spec(R), \GG_m[-1]) \to  \Fil^1 \widehat{\dR}_{R / k}
\]
for any $R \in \Alg_k$. By right Kan extension, this induces a functor 
\[
\St_k \to \Fun(\Delta^1, \derD(\ZZ))
\]
given on any $k$-stack $X$ by the map
\begin{equation}
\csmall_1^{\widehat{\dR}} \colon  \Gamma(X, \GG_m[-1]) \to  \Fil^1 \widehat{\dR}_{X / k} \label{chern_first_map}
\end{equation}
in $\derD(\ZZ)$. We call this the \emph{first Chern class in derived de Rham cohomology}. 
\end{construction}
\begin{definition}\label{defc1dr}
Let $k$ be a ring, let $X$ be a stack over $k$, and let $\mathcal{L}$ be a line bundle on $X$ corresponding to an element $[\mathcal{L}] \in \coh^1(\derR \Gamma(X, \GG_m))$. We define its first Chern class 
\[
\csmall^{\widehat{\dR}}_{1}(\mathcal{L}) \in \coh^2 (\Fil^1 \widehat{\dR}_{X /k })
\]
to be the image of $[\mathcal{L}]$.
\end{definition}
\begin{remark}\label{chern_pullback}
Since (\ref{chern_first_map}) is functorial in $X$, the construction of $\csmall^{\widehat{\dR}}_{1}$ commutes with pullbacks, i.e. for any morphism of stacks $f \colon X \to Y$ and $\mathcal{L} \in \Pic(Y)$ one has
\[
f^* \csmall^{\widehat{\dR}}_{1} (\mathcal{L}) = \csmall^{\widehat{\dR}}_{1}(f^* \mathcal{L})
\]
\end{remark}

\begin{lemma}[Projective space bundle formula]
Let $k$ be a ring, let $X$ be a stack over $k$, and let $\mathcal{E}$ be a vector bundle on $X$ of constant rank $r$.  Set
\[
t := -\csmall_1^{\widehat{\dR}}(\mathcal{O}(1)) \in \coh^2\(\Fil^1 \widehat{\dR}_{\PP(\mathcal{E}) / k}\)
\]
Then for all $m$, the map
\begin{equation}\label{eq_proj_bundle}
\bigoplus_{i = 0}^{r - 1} \Fil^{m - i} \widehat{\dR}_{X / k}[-2i] \xrightarrow{(1, t, \dots, t^{r - 1})}  \Fil^m \widehat{\dR}_{\PP(\mathcal{E}) / k}
\end{equation}
is an equivalence in $\derD(k)$. 
\end{lemma}
\begin{proof}
Since we have a map, it suffices to give a proof locally on $X$, in which case $\PP(\mathcal{E}) \cong X \times_{\Spec(k)} \PP^r_k$. By the Künneth formula (Corollary \ref{corol_kunneth}) one may reduce to the case $X = \Spec(k)$, which is the statement of \stacksref{0FMJ}. 
\end{proof}
In particular, setting $m = 0$ we get an equivalence
\[
\bigoplus_{i = 0}^{r - 1} \widehat{\dR}_{X / k}[-2i] \xrightarrow{(1, t, \dots, t^{r - 1})} \widehat{\dR}_{\PP(\mathcal{E}) / k}
\]
in $\derD(k)$.
\begin{definition}[Higher Chern classes]
Let $k$ be a ring, let $X$ be a stack over $k$, and let $\mathcal{E}$ be a vector bundle on $X$ of constant rank $r$.  Write $t = -\csmall_1^\dR(\mathcal{O}(1))$. We define the $i$-th \emph{Chern class}
\[
\csmall^{\widehat{\dR}}_i(\mathcal{E}) \in \coh^{2i}\(\Fil^i \widehat{\dR}_{X / k}\)
\] 
as the $(r - i)$-th component of the image of $-t^r$ under the inverse of the isomorphism (\ref{eq_proj_bundle}). 

For a general vector bundle $\mathcal{E} \in \Vect(X)$, we define the $i$-th Chern class by decomposing $X$ into components on which $\mathcal{E}$ has constant rank. 
\end{definition}
\begin{theorem}[Cartan formula]\label{thm_cartan}
Let $k$ be a ring, let $X$ be a stack over $k$, and let
\[
0 \to \mathcal{E} \to \mathcal{F} \to \mathcal{G} \to 0
\]
be a short exact sequence of vector bundles on $X$. Then
\[
c_i(\mathcal{F}) = \sum_{j = 0}^{i} c_j(\mathcal{E}) c_{i - j}(\mathcal{G})
\]
in $\coh^{2i} \( \Fil^i \widehat{\dR}_{X / k}\)$. 
\end{theorem}
\begin{proof}
After decomposing $X$ into pieces where the vector bundles have constant rank, by Proposition \ref{prop_bpnm} we only need to verify this for $X = \mathrm{BP}_{n, m}$ for all $n, m \in \NN$, so in particular we may assume that $X$ is an algebraic stack.

Base changing along the map $\mathrm{BGL}_{n} \times \mathrm{BGL}_m \to \mathrm{BP}_{n, m}$, by Totaro's theorem (Proposition \ref{corol_totaro}) and Lemma \ref{prop_vectngln} it suffices to show the equality holds for the universal (split) short exact sequence on $\mathrm{BGL}_{n} \times \mathrm{BGL}_{m}$. We may thus reduce to the case where $\mathcal{F} \cong \mathcal{E} \oplus \mathcal{G}$, where $\mathcal{E}$ and $\mathcal{G}$ are vector bundles of constant rank $n$ and $m$ respectively. 

Let $\pi \colon \PP(\mathcal{F}) \to X$ be the projection map, and let $\catO_{\PP(\mathcal{F})}(-1)$ be the tautological subbundle of $\pi^*(\mathcal{F})$, with Chern class $t := c_1(\catO_{\PP(\mathcal{F})}(-1))$. Let $U_\mathcal{E} \subseteq \PP(\mathcal{F})$ be the open subset of $\PP(\mathcal{F})$ for which the composite
\[
\catO_{\PP(\mathcal{F})}(-1) \to \pi^*(\mathcal{F}) \to \pi^*(\mathcal{E})
\]
is the inclusion of a subbundle, and similarly for $U_{\mathcal{G}}$. Clearly $U_\mathcal{E} \cup U_{\mathcal{G}} = \PP(\mathcal{F})$. 

Note that we have a commutative diagram
\[
\begin{tikzcd}
U_\mathcal{E} \dar{p_\mathcal{E}} \rar{\pi}& X \\
\PP(\mathcal{E}) \arrow{ur}
\end{tikzcd}
\]
and moreover, 
\[
p_{\mathcal{E}}^*(\catO_{\PP(\mathcal{E})}(-1)) = \catO_{\PP(\mathcal{F})}(-1)|_{U_{\mathcal{E}}}
\] 
It follows that the element
\[
\sum_{i = 0}^{n} \pi^*(c_i(\mathcal{E})) (-t)^i \in \coh^{2n}\(\Fil^{n} \widehat{\dR}_{U_{\mathcal{E}} / k}\) 
\]
is equal to zero. Hence there exists an element $\eta \in \coh^{2n}(\Fil^n \widehat{\dR}_{(X, U_\mathcal{E}) / k})$ in the relative cohomology group mapping to 
\[
\sum_{i = 0}^{n} \pi^*(c_i(\mathcal{E})) (-t)^i \in \coh^{2n}\(\Fil^{n} \widehat{\dR}_{X / k}\) 
\]
Similarly, we find the existence of an element $\eta'  \in \coh^{2m}(\Fil^{m} \widehat{\dR}_{(X, U_\mathcal{G}) / k})$ in the relative cohomology group mapping to 
\[
\sum_{i = 0}^{m} \pi^*(c_i(\mathcal{G})) (-t)^i \in \coh^{2m}\(\Fil^{m} \widehat{\dR}_{X / k}\) 
\]
Using Corollary \ref{prop_relative_map}, we therefore find the existence of an element
\[
\eta \cdot \eta'  \in \coh^{2(n + m )}(\Fil^{n + m} \dR_{(X, X) / k})
\]
mapping to 
\[
\(\sum_{i = 0}^{n} \pi^*(c_i(\mathcal{E})) (-t)^i\)\cdot \(\sum_{i = 0}^{m} \pi^*(c_i(\mathcal{G})) (-t)^i\)
\]
Since clearly $\Fil^{n + m} \dR_{(X, X) / k} = 0$ we find that this last expression is zero, which implies the theorem. 
\end{proof}
\begin{corollary}
Let $k$ be a ring. The map $\csmall_i$ defined above induces for any $k$-stack $X$ a unique map 
\[
\csmall_i \colon \mathrm{K}_0(\mathrm{Vect}(X)) \to \coh^{2i}\(\Fil^i \widehat{\dR}_{X / k}\)
\]
satisfying $c_i(f^* \mathcal{E}) = f^* c_i(\mathcal{E})$. 
\end{corollary}
\begin{definition}[Chern character]\label{def_chern_character}
Let $k$ be a ring such that $i!$ is invertible in $k$. Let $X$ be a stack over $k$. For $i > 0$, let $\sigma_i \in k[x_1, \dots, x_n]$ be the $i$-th symmetric polynomial, and let $\theta_i \in k[\sigma_1, \dots, \sigma_i]$ be the unique polynomial such that
\[
\theta_i(\sigma_1, \dots, \sigma_i) = x_1^i + \dots + x_n^i
\]
We define the $i$-th \emph{Chern character}
\[
\ch_i \colon \mathrm{K}_0(X) \to \coh^{2i} \(\Fil^i  \widehat{\dR}_{X / k}\)
\]
by
\[
\ch_i(\mathcal{E}) := \frac{\theta_i(c_1(\mathcal{E}), \dots, c_i(\mathcal{E}))}{i!} 
\]
using the algebra structure on 
\[
\bigoplus_{j} \coh^{2j}(\Fil^j  \widehat{\dR}_{X / k})
\]
induced by the filtered $\EE_\infty$-algebra structure on $\widehat{\dR}_{X / k}$. 
\end{definition}

We now wish to generalize to perfect complexes. Let $X$ be a quasi-compact and quasi-separated scheme over $k$ which has the resolution property \stacksref{0F8D}. By \stacksref{0F8E} any $\mathcal{E} \in \Perf(X)$ can be represented by a bounded complex of vector bundles, so that    we may talk about its image $[\mathcal{E}] \in \mathrm{K}_0(\Vect(X))$.

\begin{definition}
Let $k$ be a ring, and let $X$ be a quasi-compact and quasi-separated scheme over $k$ which has the resolution property. For $\mathcal{E} \in \Perf(X)$, define the $i$-th \emph{Chern class} 
\[
\csmall_i (\mathcal{E}) \in  \coh^{2i}\(\Fil^i  \widehat{\dR}_{X / k}\)
\]
as $\csmall_i (\mathcal{E}) := \csmall_i ([\mathcal{E}])$. If $i!$ is invertible in $k$, we define $\ch_i(\mathcal{E}) := \ch_i([\mathcal{E}])$ for any $\mathcal{E} \in \Perf(X)$. 
\end{definition}

Using the natural map $\Fil^i  \widehat{\dR}_{X / k} \to \LL^i_{X / k}[-i]$, we also obtain Chern classes and characters in $\coh^i(\LL^i_{X / k})$. 

\begin{prop}\label{prop_chern}
Let $k$ be a ring, and let $i$ be a number such that $i!$ is invertible in $k$. The $i$th Chern character $\ch_i$ satisfies the following properties:
\begin{enumerate}
\item For any quasi-compact and quasi-separated scheme $X$ over $k$ with the resolution property and any exact triangle
\[
\mathcal{E} \to \mathcal{F} \to \mathcal{G} \xrightarrow{+1}
\]
in $\Perf(X)$, one has
\[
{\ch}_i(\mathcal{F}) = {\ch}_i(\mathcal{E}) + {\ch}_i(\mathcal{G})
\]
\item If $X, Y$ are quasi-compact and quasi-separated schemes over $k$ with the resolution property and $f \colon X \to Y$ is a morphism of schemes over $k$, then 
\[
{\ch}_i(f^* \mathcal{E}) = f^* {\ch}_i(\mathcal{E})
\]
for all $\mathcal{E} \in \Perf(X)$.
\item For any scheme $X$ over $k$ and any line bundle $\mathcal{L}$ on $X$ one has 
\[
{\ch}_i(\mathcal{L}) = \frac{c_1(\mathcal{L})^i}{i!}
\]
\end{enumerate}
\end{prop}
\begin{proof}
Statement (2) follows from Remark \ref{chern_pullback}. Statement (1) and (3) follow from elementary identities between symmetric polynomials combined with Theorem \ref{thm_cartan}.
\end{proof}
We now show these properties characterize the Chern character uniquely. As we will later need this for Hodge cohomology, we formulate the statement for both de Rham and Hodge cohomology.
\begin{prop}[Uniqueness of Chern character]\label{prop_chern_unique}
Let $k$ be a ring such that $i!$ is invertible in $k$. Let $A^i \in \{\LL^i_{- / k}[-i], \Fil^i \widehat{\dR}_{- / k}\}$. 

Suppose that for every quasi-compact and quasi-separated scheme $X$ over $k$ which has the resolution property and any $\mathcal{E} \in \Perf(X)$, we are given an element
\[
\widetilde{\ch}_i(\mathcal{E}) \in \coh^{2i}(A^i_{X / k})
\]
satisfying the properties from Proposition \ref{prop_chern}. 

Then $\widetilde{\ch}_i(\mathcal{E}) = \ch_i(\mathcal{E})$ for all quasi-compact and quasi-separated schemes $X$ over $k$ with the resolution property, and all $\mathcal{E} \in \Perf(X)$. 
\end{prop}

\begin{proof}
Let $X$ be a quasi-compact and quasi-separated scheme over $k$ with the resolution property, and $\mathcal{E} \in \Perf(X)$. We wish to show that $\widetilde \ch_i(\mathcal{E}) = \ch_i(\mathcal{E})$. Since $X$ has the resolution property we may represent $\mathcal{E}$ by a boundex complex of locally free sheaves, hence using (1) we may reduce to the case where $\mathcal{E}$ is a locally free sheaf on $X$. Since the pullback map 
\[
\coh^{2i}(A^i_{X / k}) \to \coh^{2i}(A^i_{\PP(\mathcal{E}) / k})
\]
is injective for all $i \geq 0$, we may reduce to the case where $\mathcal{E}$ has a filtration with graded quotients given by line bundles. The result then follows by applying (1) and (3). 
\end{proof}

\subsection{Chern classes without Hodge completion}
In this section we construct Chern classes in uncompleted derived de Rham cohomology in the $p$-adic case. The following lemma is the main ingredient in the construction.

\begin{lemma}\label{iso_bgln}
Let $p$ be a prime number, and let $k$ be a ring over $\ZZ / p^n \ZZ$ for some $n \geq 1$. Let $G$ be a smooth affine group scheme over $k$.

Then for any $i \geq 0$, the map
\[
\coh^{2i}(\Fil^i \dR_{\mathrm{B}G / k}) \xrightarrow{\sim} \coh^{2i}(\Fil^i \widehat{\dR}_{\mathrm{B}G / k}) 
\]
is an isomorphism.
\end{lemma}
\begin{proof}
Since $\mathrm{B}G$ can be written as a colimit of smooth affine stacks over a simplicial diagram and we defined de Rham cohomology of stacks by right Kan extension, its de Rham cohomology can be computed by computing it for the affine schemes and taking the limit over the cosimplicial diagram. Thus the result follows from Lemma \ref{dr_iso_p}.
\end{proof}
\begin{construction}
Let $p$ be a prime number, and let $k$ be a ring over $\ZZ / p^n \ZZ$ for some $n \geq 1$. Let $\mathcal{E}_{\mathrm{univ}}$ be the universal rank $r$ vector bundle on $\mathrm{BGL}_{r, k}$. Let 
\[
c^\mathrm{univ}_i \in \coh^{2i}(\Fil^i \dR_{\mathrm{BGL}_{r, k}})
\]
be the inverse image of $c_i^{\widehat{\dR}}(\mathcal{E}_{\mathrm{univ}})$ under the isomorphism 
\[
\coh^{2i}(\Fil^i {\dR}_{\mathrm{BGL}_{r, k}}) \xrightarrow{\sim} \coh^{2i}(\Fil^i \widehat\dR_{\mathrm{BGL}_{r, k}})
\]
from Lemma \ref{iso_bgln}. 

For $X$ any stack over $k$ and $\mathcal{E}$ a vector bundle of rank $r$ on $X$ corresponding to a map 
\[
f_{\mathcal{E}} \colon X \to \mathrm{BGL}_{r, k}
\]
(see Lemma \ref{prop_vectngln}), we define $c_i(\mathcal{E}) := f_{\mathcal{E}}^*(c_i^{\mathrm{univ}})$. For general $\mathcal{E} \in \Vect(X)$ we define its Chern class by decomposing $X$ into pieces on which $\mathcal{E}$ has constant rank. If $i!$ is invertible in $k$, we define the $i$-th Chern character
\[
\ch_i(\mathcal{E}) := \frac{\theta_i(c_1(\mathcal{E}), \dots, c_i(\mathcal{E}))}{i!} 
\]
When $X$ is a quasi-compact and quasi-separated scheme over $k$ which has the resolution property, and $\mathcal{E} \in \Perf(X)$ is a perfect complex, define the $i$-th \emph{Chern class} as $\csmall_i (\mathcal{E}) := \csmall_i ([\mathcal{E}])$. If $i!$ is invertible in $k$, we define $\ch_i(\mathcal{E}) := \ch_i([\mathcal{E}])$ for any $\mathcal{E} \in \Perf(X)$. 
\end{construction}
We leave it to the reader to verify that the above definition is the (unique) construction satisfying the properties from Proposition \ref{prop_chern} (the only nontrivial thing to check is the Cartan formula, which can be done by applying Lemma \ref{iso_bgln} to $\mathrm{BP}_{n,m,k}$).

\label{sec_def_theory_hodge}
\section{Kodaira--Spencer classes and variations of Hodge structure}
Let $A$ be a local Artinian $\CC$-algebra and let $X$ be a smooth and proper variety over $A$. Let 
\[
X_0 = X \times_{\Spec(A)} \Spec(\CC)
\]
By \cite[Lemma 5.5.3]{deligne68} the maps
\begin{align*}
A \to  \Omega^\bullet_{X^\an / A} \\
 \CC \to \Omega^\bullet_{X_0^\an / \CC}
\end{align*}
are quasi-isomorphisms of complexes of sheaves on $X^\an$. Combined with GAGA we obtain isomorphisms
\begin{align*}
\coh^*(X^\an, A) &\cong \coh^*(X, \Omega^\bullet_{X / A}) \\
\coh^*(X_0^\an, \CC) &\cong \coh^*(X_0, \Omega^\bullet_{X_0 / \CC})
\end{align*}
Denote with $\varphi$ the composition
\begin{equation}\label{eq_stratifying_map}
\coh^*(X_0, \Omega^\bullet_{X_0 / \CC}) \otimes_\CC A \xrightarrow{\sim} \coh^*(X_0^\an, \CC) \otimes_\CC A \cong \coh^*(X^\an, A) \xrightarrow{\sim} \coh^*(X, \Omega^\bullet_{X / A})
\end{equation}
Given $i \geq 0$ and an element 
\[
v \in \Fil^i \coh^{2i}(X_0, \Omega^\bullet_{X_0 / \CC})
\]
we want to determine whether or not $\varphi(v)$ lies in the $i$-th part of the Hodge filtration.

Bloch \cite{bloch} showed this can be studied using the Gauss--Manin connection, however his procedure only works with conditions on the base $A$. The goal of this section is to generalize his method to general bases (and even to mixed characteristic). The main idea is to replace the isomorphism $\varphi$ with its algebraic analogue (\ref{eq_goodwillie_dr}), an idea originally due to Pridham \cite{pridham}. This will allow us to generalize Bloch's algebraic computation to all $A$.

\subsection{Hodge--theoretic obstructions for completed derived de Rham cohomology}
\label{sec_bloch_dr}
In this section, we rephrase Bloch's problem in terms of a more algebraic problem. We first define an algebraic analogue of the map (\ref{eq_stratifying_map}) for completed derived de Rham cohomology, using nil-invariance. We then define an obstruction class that measures whether or not a cohomology class that sits within the Hodge filtration over the base remains within the Hodge filtration when that smaller base is enlarged by a nilpotent thickening. 

We start by introducing a more general notion of a local Artinian $\CC$-algebra. Note that we will only be considering \emph{discrete} thickenings. 
\begin{definition}
Let $k$ be a ring. A \emph{nilpotent thickening} of $k$ is a commutative $k$-algebra $R$ and a nilpotent ideal $I \subseteq R$ such that the composition 
\[
k \to R \to (R / I)
\]
is an isomorphism. A \emph{morphism of nilpotent thickenings} is a commutative diagram
\[
\begin{tikzcd}
& k \arrow{dr} \arrow{dl} \\
R' \arrow{dr} \arrow[two heads]{rr} & & R \arrow{dl} \\
& k 
\end{tikzcd}
\]
such that $R' \to R$ is surjective. Finally, we say that a morphism of nilpotent thickenings is \emph{square zero} if $J  =\ker(R' \to R)$ satisfies $J^2 = 0$. 
\end{definition}
\begin{example}
Any local Artinian $\CC$-algebra is a nilpotent thickening of $\CC$. 
\end{example}
\begin{remark}\label{aak}
If $k$ is any ring such that $\QQ \subseteq k$ and $R$ is a nilpotent thickening of $k$, for any smooth scheme $X$ over $R$ with $X_0 := X \times_{\Spec(R)} \Spec(k)$, the map
\[
\widehat{\dR}_{X / R} \to \widehat{\dR}_{X_0 / R}
\]
is an equivalence in $\derD(R)$ by globalizing Theorem \ref{theorem_goodwillie_dr} using (\ref{def_extend}).
\end{remark}
\begin{definition}\label{aal}
Let $k$ be a ring such that $\QQ \subseteq k$, let $R$ be a nilpotent thickening of $k$, and let $X$ be a smooth scheme over $R$. Let $X_0 = X \times_{\Spec(R)} \Spec(k)$. Define the \emph{stratifying map} 
\[
\varphi_{\widehat{\dR},X} \colon 
\widehat{\dR}_{X_0 / k} \otimes_k R  \to \widehat{\dR}_{X  / R }
\]
as the composition
\[
\begin{tikzcd}
\widehat{\dR}_{X_0 / k} \otimes_k R  \rar&  \widehat{\dR}_{X_0 / R } & \lar{\sim} \widehat{\dR}_{X  / R }
\end{tikzcd}
\]
after inverting the right equivalence.
\end{definition}
Before we continue, we verify that our stratifying map agrees with the map (\ref{eq_stratifying_map}), so that there can be no confusion about the map $\varphi_{X \times Y}$ in Theorem \ref{thm_zero}. 
\begin{lemma}\label{lem_strat_agree}
Let $R$ be a local Artinian $\CC$-algebra, and let $X$ be a smooth and proper scheme over $R$. Denote with $X_0 := X \times_{\Spec(R)} \Spec(\CC)$. Then the diagram
\[
\begin{tikzcd}
\dar{\sim} \coh^*(X_0, \widehat{\dR}_{X_0 / \CC}) \otimes_\CC R  \rar{\varphi_{\widehat{\dR},X }}& \coh^*(X , \widehat{\dR}_{X  / R }) \dar{\sim} \\
\coh^*(X_0, \Omega^\bullet_{X_0 / \CC}) \otimes_\CC R  \rar{(\ref{eq_stratifying_map})}& \coh^*(X , \Omega^\bullet_{X  / R })
\end{tikzcd}
\]
commutes. 
\end{lemma}
\begin{proof}
For $Y \to \Spec(B)$ a morphism of schemes of finite type $\CC$ such that $Y$ can be embedded in a smooth $B$-scheme, write $\coh_\har^*(Y / B)$ for Harthshorne's algebraic de Rham cohomology \cite[\S II.1]{harthshornedr}, and $\coh_\har^*(Y^\an/B^\an)$ for the holomorphic (analytic) version \cite[\S IV.1]{harthshornedr}. The main properties we need of Harthshorne's theory are the following. 
\begin{itemize}
\item The cohomology groups $\coh_\har^*(Y / B)$ and $\coh_\har^*(Y^\an/B^\an)$ are functorial in the pair $(Y, B)$.
\item Whenever $Y \to \Spec(B)$ is smooth, there are functorial isomorphisms $\coh_\dR^*(Y / B) \cong \coh_\har^*(Y / B)$ and $\coh_\dR^*(Y^\an / B^\an) \cong \coh_\har^*(Y^\an/B^\an)$. 
\end{itemize}
We get a commutative diagram
\[
\begin{tikzcd}
\coh^*(\widehat{\dR}_{X  / R }) \rar{\sim} \dar{\sim} & \coh_\har^*(X / R ) \rar{\sim} \dar & \coh_\har^*((X)^\an / (R)^\an)  \dar &\lar{\sim} \coh^*(X^\an , R ) \dar{\sim} \\
\coh^*(\widehat{\dR}_{X_0 / R } )\rar{\sim} & \coh_\har^*(X_0 / R ) \rar{\sim} & \coh_\har^*(X_0^\an / (R )^\an)  &\lar{\sim} \coh^*(X_0^\an, R ) \\
\coh^*(\widehat{\dR}_{X_0 / \CC}) \uar \rar{\sim} &  \uar \coh_\har^*(X_0 / \CC) \rar{\sim} &  \uar \coh_\har^*(X_0^\an / \CC^\an)  & \uar\lar{\sim} \coh^*(X_0^\an, \CC)
\end{tikzcd}
\]
By \cite[Corollary 4.27]{bhatt-cddrc} the left horizontal arrows are isomorphisms. Then note that $X_0 \to R$ is proper. Moreover $\coh_\har^i(X_0/ R)$ is finite for all $i$, for example by comparing with $\coh_\har^*\(X / R\)$ via (\ref{eq_goodwillie_dr}). Thus by \cite[Proposition 4.1]{harthshornedr} the middle horizontal arrows are isomorphisms. Finally by \cite[Lemma 5.5.3]{deligne68} the right horizontal arrows are isomorphisms.

The top left vertical arrow is an isomorphism by (\ref{eq_goodwillie_dr}). It follows that all top vertical arrows are isomorphisms.  The result now follows by carefully chasing through the diagram after inverting all the relevant arrows: Going straight up from the bottom left to the top left gives $\varphi_{\widehat{\dR}, X}$, going all the way right-up-left gives (\ref{eq_stratifying_map}). 
\end{proof}

The following lemma is the algebraic version of the statement that Chern classes are horizontal for the Gauss--Manin connection.
\begin{lemma}[Horizontality of Chern classes]\label{lem_chern_horizontal}
Let $k$ be a ring and let $R$ be a nilpotent thickening of $k$. Let $X$ be a smooth and proper scheme over $R$. Let $\mathcal{E} \in \Perf(X)$, and let $\mathcal{E}_0 = \mathcal{E}\rvert_{X_k}$. Then
\[
\varphi_{\widehat{\dR}, X }\(\ch_i(\mathcal{E}_0) \otimes 1\) = \ch_i(\mathcal{E})
\]
in $\coh^{2i}(\widehat{\dR}_{X  / R })$. 
\end{lemma}
\begin{proof}
Since the diagram
\[
\begin{tikzcd}
\mathrm{K}_0(X_k) \dar{\ch_i} \arrow[equals]{r} &  \dar{\ch_i} \mathrm{K}_0(X_k) & \lar \mathrm{K}_0({X}) \dar{\ch_i}   \\
\coh^{2i}( \widehat{\dR}_{X_k / k}) \rar& \coh^{2i}(\widehat{\dR}_{X_k / R }) & \lar \coh^{2i}( \widehat{\dR}_{X  / R })
\end{tikzcd}
\]
commutes, this follows immediately by definition of the stratifying map.
\end{proof}


We now finally give our definition of the obstruction class, as promised. 
\begin{definition}\label{our_obs_class}
Let $k$ be a ring and let $R$ be a nilpotent thickening of $k$. Let $X$ be a smooth and proper scheme over $R$.  Let $v_0 \in \coh^{2i}(\Fil^i \widehat{\dR}_{X_k / k})$. 

We define the \emph{obstruction class to $v_0$ staying in the Hodge filtration}
\[
\ob^{\widehat{\dR}}_{X  / R }(v_0) \in \coh^{2i}\(\widehat{\dR}_{X / R } / \Fil^i  \)
\]
as the image of $v_0 \otimes 1$ under the composition
\[
\coh^{2i} (\widehat{\dR}_{X_k / k}) \otimes_k R \xrightarrow{\varphi_{\widehat{\dR},X }} \coh^{2i} (\widehat{\dR}_{X / R} ) \to \coh^{2i} (\widehat{\dR}_{X / R } / \Fil^i) 
\]
\end{definition}
Almost by definition, we see that the obstruction class to $v_0$ staying in the Hodge filtration vanishes if and only if $v_0$ lands in the $i$-th parth of the Hodge filtration on $\coh^{2i}(\widehat{\dR}_{X / R }  )$, which explains the terminology. 
\subsection{Comparison with Bloch's construction}
\label{sec_compare_bloch}
In this section, we compare the construction of the obstruction class from Definition \ref{our_obs_class} with the classical construction of Bloch \cite{bloch}. Let $k = \CC$, and suppose $R' \to R$ is a square zero morphism of nilpotent thickenings with $I = \ker(R' \to R)$. Let $X'  \xrightarrow{f} \Spec(R')$ be a smooth and proper morphism, and set
\begin{align*}
X &:= X'  \times_{\Spec(R' )} \Spec(R) \\
X_0 &:= X'  \times_{\Spec(R' )} \Spec(\CC)
\end{align*}
Suppose $v_0 \in \coh^{2i}_{\dR}(X_0 / \CC)$ is such that 
\[
\ob^{\widehat{\dR}}_{X  / R }(v_0) \in \coh^{2i}_{\dR}(X / R) / \Fil^i
\]
vanishes, so that we may consider the horizontal lift
\[
v \in \Fil^i \coh^{2i}_{\dR}(X / R)
\]
In this case, Bloch \cite{bloch} defined an obstruction class in
\[
\ob^{\mathrm{Bloch}}_{X' / R'}(v_0) \in (\coh^{2i}_{\dR}(X' / R') / \Fil^i) \otimes_{R'} \Omega_{R' / \CC}
\]
as follows. 
\begin{construction}
For $a \in I$ and $\omega \in \Fil^i \coh^{2i}(X' / R')$, the Gauss--Manin connection
\[
\coh^{2i}_{\dR}(X' / R') \xrightarrow{\nabla} \coh^{2i}_{\dR}(X' / R') \otimes_{R'} \Omega_{R' / \CC}
\]
satisfies $\nabla(a \cdot \omega) = a \cdot \nabla(\omega) + \mathrm{d}a \wedge \omega$, hence 
\[
\nabla(a \cdot \omega) \equiv 0 \pmod M
\] 
where $M \subseteq \coh^{2i}_{\dR}(X' / R') \otimes_{R'} \Omega_{R' / \CC}$ is the submodule generated by the submodules $\Fil^i \coh^{2i}_{\dR}(X' / R') \otimes_{R'} \Omega_{R' / \CC}$ and $I \cdot \coh^{2i}_{\dR}(X' / R') \otimes_{R'} \Omega_{R' / \CC}$. 

Identifying 
\begin{align*}
\coh^{2i}_{\dR}(X' / R') / (I \cdot \coh^{2i}_{\dR}(X' / R')) &\cong \coh^{2i}_{\dR}(X / R) \\
\coh^{2i}_{\dR}(X' / R') \otimes_{R'} \Omega_{R' / \CC} / M &\cong (\coh^{2i}_{\dR}(X / R) / \Fil^i) \otimes_{R'} \Omega_{R' / \CC}
\end{align*}
one obtains a map
\[
\Fil^i \coh^{2i}_{\dR}(X / R) \xrightarrow{\overline{\nabla}} (\coh^{2i}_{\dR}(X / R) / \Fil^i) \otimes_{R'} \Omega_{R' / \CC}
\]
Bloch's obstruction class can then be defined as 
\[
\ob^{\mathrm{Bloch}}_{X' / R'}(v_0)  := \overline{\nabla}(v)
\]
\end{construction}
In \cite[Proposition 4.2]{bloch}, Bloch showed the following. 
\begin{lemma}[Bloch] Assume the composition 
\begin{equation}\label{aaaaa}
I \to R'  \xrightarrow{\mathrm{d}} \Omega_{R'  / \CC} \to \Omega_{R' / \CC} \otimes_{R'} R
\end{equation}
is injective. Then the following are equivalent:
\begin{enumerate}
\item $\varphi(v_0) \in \Fil^i \coh^{2i}_{\dR}(X' / R')$ 
\item $\ob^{\mathrm{Bloch}}_{X' / R'}(v_0) = 0$
\end{enumerate}
\end{lemma}
That is, Bloch's obstruction class measures whether or not $v$ remains wihtin the $i$-th part of the Hodge filtration when deforming to $X'$ as long as the map (\ref{aaaaa}) is injective (this condition on $I$ also appears in the paper of Buchweitz and Flenner, see \cite[5.6]{buchflen}). The condition on $I$ is satisfied for many important examples, and in particular it is satisfied for the ideal $(t^n) \subseteq \CC[t] / (t^{n + 1})$. 

However, this is not sufficient to handle all Artinian $\CC$-algebra's $A$, as is shown by the following example.
\begin{example}[See \cite{mohan}]
Consider the polynomial
\[
f = x^2y^2 + x^5 + y^5 \in \CC[x, y]
\]
Note $A = \CC[x, y] / (f_x, f_y)$ is an Artinian local $\CC$-algebra. One may verify that $f \not \in (f_x, f_y)$. Thus $df = 0$ in $\Omega_{A / \CC}$ but $f \neq 0$ in $A$. In particular, there exists no ideal $I \subseteq A$ containing $f$ for which the map (\ref{aaaaa}) is injective. 
\end{example}

For completeness, we end this section by relating Bloch's obstruction class to our obstruction class.
\begin{lemma}
Denote with $\mathrm{d}$ the composition
\[
I \hookrightarrow R' \xrightarrow{\mathrm{d}} \Omega_{R' / \CC} \to \Omega_{R' / \CC} \otimes_{R'} R
\]
Then the map
\begin{equation} \label{xxyyzz}
(\coh^{2i}_{\dR}(X / R) / \Fil^i) \otimes_{R} I \xrightarrow{\id \otimes \mathrm{d}}  \coh^{2i}_{\dR}(X / R) / \Fil^i) \otimes_{R} \Omega_{R' / \CC}
\end{equation}
maps $\ob^{\widehat{\dR}}_{X' / R'}(v_0) \mapsto -\ob^{\mathrm{Bloch}}_{X' / R'}(v_0)$. 
\end{lemma}
\begin{proof}
Write 
\[
v' = \varphi_{\widehat{\dR},X'}(v_0) \in \coh^{2i}_{\dR}(X' / R'), \qquad v = \varphi_{\widehat{\dR},X}(v_0) \in \coh^{2i}_{\dR}(X / R)
\]
Since $v \in \Fil^i \coh^{2i}_{\dR}(X / R)$, we may write 
\[
v' = w_1 + a \cdot w_2
\]
where $w_1 \in \Fil^i \coh^{2i}_{\dR}(X' / R')$, $a \in I$ and $w_2 \in \coh^{2i}_{\dR}(X' / R')$. Note that by definition one has 
\[
\ob^{\mathrm{Bloch}}_{X' / R'}(v_0) \equiv \overline{\nabla}(\overline{w_1}) \pmod M
\]
Since $v'$ is horizontal, we have $\nabla(v') = 0$, thus 
\[
\nabla(w_1) = -\nabla(a \cdot w_2) = -a \cdot \nabla(w_2) + w_2 \mathrm{d}a 
\] 
Reducing modulo $I$ we obtain
\[
\ob^{\mathrm{Bloch}}_{X' / R'}(v_0) = \overline{\nabla}(\overline{w_1}) =  -w_2 \mathrm{d}a
\]
in $\coh^{2i}_{\dR}(X / R) / \Fil^i) \otimes_{R'} \Omega_{R' / \CC}$, which proves the result since the right hand side is precisely the image of $-\ob^{\widehat{\dR}}_{X' / R'}(v_0)$ under the map \ref{xxyyzz} by definition.
\end{proof}

\subsection{Hodge--theoretic obstructions for derived crystalline cohomology}
In this section we give definitions analogous to those in Section \ref{sec_bloch_dr}. Note however that here we are working with uncompleted theories, which are not well behaved in characteristic zero. We start by giving the crystalline analogue of the map (\ref{eq_stratifying_map}). 
\begin{definition}\label{abk}
Let $k$ be a ring. Let $(R \to R_0) \in \PDPair_k$ and let $X \to \Spec(R) $ be a smooth scheme. Write $X_0 = X \times_{\Spec(R)} \Spec(R_0)$. We define the \emph{crystalline stratification map} 
\[
\varphi_{\Crys,X} \colon {\dR}_{X_0  / k}  \to {\dR}_{X / R}
\]
in $\derD(k)$ as the map obtained by inverting the equivalences in the diagram
\[
\dR_{X_0 / k} \to {\Crys}_{X_0 / (R \to R_0)} \xleftarrow{\sim}  {\Crys}_{X / (R \to R)} \otimes_{R} \Crys_{R_0 / (R \to R_0)} \xleftarrow{\sim} {\dR}_{X / R}
\]
where the middle map is an equivalence  by Lemma \ref{lem_crys_colim}, and the last map is an equivalence since \[
\Crys_{R_0 / (R \to R_0)} \cong R
\]
We will denote with 
\begin{equation}
\alpha_\Crys \colon \Crys_{X_0 / (R \to R_0)} \xrightarrow{\sim} \dR_{X / R} \label{abh}
\end{equation}
 the map obtained by inverting the two equivalences in the diagram above. 
\end{definition}
\begin{remark}
By \cite[Proposition 3.25]{bhatt-padic} or \cite[Proposition 4.66, 4.87, 4.90]{mao}, if $R$ is a $(\ZZ / p^n \ZZ)$-algebra we have a diagram
\[
\begin{tikzcd}
\coh^*(\Crys_{X_0 / (R \to R_0)}) \dar{\sim}  \rar{\alpha_\Crys} & 
\coh^*(\dR_{X / R}) \dar{\sim} \\
\coh^*_{\mathrm{cris}}(X_0 / (R \to R_0)) \rar{\sim}&  \coh^*_{\mathrm{dR}}(X / R)
\end{tikzcd}
\]
Thus our construction agrees with more classical constructions. 
\end{remark}
\begin{definition}
Let $k$ be a ring. Let $(R \to R_0) \in \PDPair_k$, and let $X$ be a smooth and proper scheme over $R$. Write $X_0 := X \times_{\Spec(R)} \Spec(R_0)$. Let $v_0 \in \coh^{2i}(\Fil^i {\dR}_{X_0 / k})$.

We define the \emph{obstruction class to $v_0$ staying in the Hodge filtration}
\[
\ob^\Crys_{X / R}(v_0) \in \coh^{2i}\( {\dR}_{X / R} / \Fil^i  \)
\]
as the image of $v_0$ under the composition
\begin{align*}
\coh^{2i}(\Fil^i \dR_{X_0 / k}) &\to \coh^{2i} \( \Crys_{X_0 / (R \to R_0)}\) \\& \xrightarrow{\varphi_{{\Crys},X }} \coh^{2i} \( {\dR}_{X / R} \) \\ &\to \coh^{2i} \( {\dR}_{X / R} / \Fil^i\) 
\end{align*}
\end{definition}

 \subsection{The (derived) Kodaira--Spencer map}

Let $k$ be a commutative ring. In this section we define the derived analogue of the Kodaira--Spencer map. That is, for any surjection of $k$-algebras $A' \to A$ with kernel $I$, we wish to construct a map
\[
\LL_{A / k}[-1] \to I / I^2
\]
 in $\derD(A)$. Moreover, this map should be functorial in the pair $(A' \to A)$. In fact, we believe this to be the map in \stacksref{0GPT}, however we will not verify this.  The starting point is the following construction.
\begin{construction}\label{cons_ks}
Let $k$ be a ring. Combining Proposition \ref{surj_dr_adfil} and Lemma \ref{compare_01}, we get for any any $(A' \to A) \in \AniPair_k$ an equivalence
\begin{equation}
\dR_{A / A'} / \Fil^2 \xrightarrow{\sim} A' / \derL\Fil_\adic^2  \label{dr_filadic}
\end{equation}
in $\CAlg_\fil(k)$. If we now assume $A' \to A$ is a map of discrete $k$-algebras with kernel $I$, we can consider the composition
\begin{equation}
\dR_{A / A'} / \Fil^2 \xrightarrow{\sim} A' / \derL\Fil_\adic^2  \to A' / \Fil_\adic^2 \label{aaa}
\end{equation}
in $\CAlg_\fil(k)$. Applying $\gr^1(-)$ we get a map
\[
\kappa_{A / A' / k} \colon \LL_{A / A'}[-1] \to I / I^2
\]
We will also denote the composition
\[
\LL_{A /k}[-1] \to \LL_{A / A'}[-1] \to \gr^1_\adic(A' \to A) = I / I^2
\]
with $\kappa_{A / A' / k}$, and refer to it as the \emph{Kodaira--Spencer map}.
\end{construction}
\begin{remark}\label{aai}
Note that we may have chosen $k = A'$ in the above construction, so that we get a commutative diagram
\begin{equation}
\begin{tikzcd}
\dar A'  \rar& A' / \Fil^2_\adic \\
\dR_{A / A'}  \rar & \dR_{A / A'} / \Fil^2\uar{(\ref{aaa})} 
\end{tikzcd}
\end{equation}
in $\derD(A')$. 
\end{remark}

 We now wish to lift $\kappa_{A / A' / k}$ to a map in $\derD(A)$, functorial in $A$. For this, we first need to construct a good target category.
\begin{definition}
We denote with $\AlgMod_k^\heart$ the $1$-category of pairs $(A, M)$ where $A$ is a discrete commutative $k$-algebra and $M$ is a discrete $A$ module. A morphism $(A, M) \to (B, N)$ is given by a map of $k$-algebras $A \to B$ (which gives $N$ the structure of an $A$-module), and a map of $A$-modules $M \to N$. 
\end{definition}

A set of compact $1$-projective generators for $\AlgMod_k^\heart$ is given by the set $S = \{(k[x], 0), (k, k)\}$. The full subcategory spanned by coproducts of elements in $S$ is the $1$-category of pairs $(A, M)$ such that $A$ is a finitely generated polynomial algebra over $k$ and $M$ is a finite free $A$-module. 

\begin{definition}
We define $\AlgMod_k^\an := \Ani(\AlgMod_k^\heart)$. 
\end{definition}

One may show that for any $A \in \Alg_k$ one has 
\[
\AlgMod_k^\an \times_{\CAlg_k^\an} \{A\} \cong \derD(A)_{\geq 0}
\] 
We denote with 
\begin{align*}
p_0 \colon \AlgMod_k^\an \to \CAlg_k^\an
\end{align*}
the functor informally given by $(A, M) \mapsto A$, and with
\begin{align*}
p_1 \colon \AlgMod_k^\an \to \derD(k)
\end{align*}
the functor informally given by $(A, M) \mapsto M$ (these functors are easily constructed by animating). Write $\ev_1 \colon \AniPair_k \to \CAlg_k^\an$ for the morphism informally given by $(A' \to A) \mapsto A$, and 
\[
\const \colon \AniPair_k \to \Fun(\Delta^1, \AniPair_k)
\]
for the functor informally given by $A \mapsto (A \xrightarrow{\id} A)$.

Informally, the following lemma shows there exists a functor sending an (animated) surjective ring map $A' \twoheadrightarrow A$ to a morphism in $\derD(A)_{\geq 0}$ lifting the morphism 
\[
\LL_{A / A'}[-1] \to \derL \gr^1_\adic(A' \to A)
\]
in $\derD(k)$. 
\begin{lemma}\label{lem_hard_lifting_ks_map}
Let $k$ be a ring. Denote with $\Theta \colon \AniPair_k \to \Fun(\Delta^1, \derD(k)_\fil)$ the functor sending an object $(A' \to A)$ to the map (\ref{dr_filadic}).

There exists a unique colimit-preserving functor 
\[
\psi \colon \AniPair_k \to \Fun(\Delta^1, \AlgMod_k^\an)
\]
such that  $p_0 \circ \psi \simeq \const \circ \mathrm{ev}_1$ as functors $\AniPair_k \to \Fun(\Delta^1, \CAlg_k^\an)$ and $p_1 \circ \psi  \simeq \gr^1(\Theta)$ as functors $\AniPair_k \to \Fun(\Delta^1, \derD(k))$, 
\end{lemma}
\begin{proof}
Let
\begin{align*}
\phi \colon \Fun(\Delta^1, \CAlg_k^\an)_\surj &\to \Fun(\Delta^1, \derD(k))
\end{align*}
be defined by $\phi :=  \gr^1\(\Theta\) $. Informally, $\phi$ is given by
\[
(A' \to A) \mapsto \(\LL_{A / A'}[-1] \xrightarrow{\Theta} \derL \gr^1_\adic(A' \to A)\)
\]
By \stacksref{08SI}, we may restrict to get a functor
\[
\phi \colon \Fun(\Delta^1, \Poly_k)_{\surj, \gen} \to \Fun(\Delta^1, \derD(k)^\heart)
\]
In particular, for any standard surjective ring map between polynomial algebras
\[
k[x_1, \dots, x_n, y_1, \dots, y_m] \xrightarrow{x_i \mapsto 0} k[y_1, \dots, y_m]
\]
from now on denoted $P \to Q$ with kernel $I := (x_1, \dots, x_n)$, we obtain a map
\begin{equation} \label{eq_comp_map_mod}
\LL_{Q / P}[-1] \xrightarrow{\phi_{Q / P}} I / I^2
\end{equation}
in $\derD(k)^\heart$, functorial in $P \to Q$. 

We claim that for a fixed standard surjection of polynomial $k$-algebras $P \to Q$ the map (\ref{eq_comp_map_mod})  is a map of $Q$-modules. Indeed, the entire construction above is functorial in $k$, and we can consider the surjection $P \to Q$ as a standard surjection of polynomial $Q$-algebras, so we may simply have chosen $k = Q$ at the beginning.

We thus have a unique lift
\begin{align*}
\tilde{\psi} \colon \Fun(\Delta^1, \Poly_k)_{\surj, \gen} \to \Fun(\Delta^1, \AlgMod_k^\heart) \\
(Q \to P) \mapsto \((Q, \LL_{Q / P}) \xrightarrow{\phi_{Q / P}} (Q, I / I^2)\)
\end{align*}
By animating $\tilde{\psi}$, we obtain our desired functor
\[
\begin{tikzcd}
\psi \colon \AniPair_k \to \Fun(\Delta^1, \AlgMod_k^\an)
\end{tikzcd}
\]
one easily checks that it satisfies the compatibilities outlined in the statement. Uniqueness is clear as the colimit-preserving property implies that $\psi$ is determined by $\tilde{\psi}$.
\end{proof}
 We now can construct our lifted Kodaira--Spencer map.
 
\begin{construction}\label{ks_map_mod}
Let $k$ be a ring. Using animation, one may construct a functor given informally by 
\begin{align*}
\psi_0 \colon \AniPair_k &\to \Fun(\Delta^1, \AlgMod_k^\an) \\
(A' \to A) &\mapsto ((A, \LL_{A / k}) \to (A, \LL_{A / A'}))
\end{align*}
Animating the functor
\begin{align*}
\tilde{\psi_1} \colon \Pair_k &\to \AlgMod_k^\heart \\
(A \to A / I) &\mapsto (A / I, I / I^2)
\end{align*}
we set $\psi_1 \colon \Ani(\tilde{\psi_1}) \to \tilde{\psi_1}$ to be the canonical natural transformation. Thus $\psi_1$ is a functor
\begin{align*}
\psi_1 \colon \Pair_k &\to \Fun(\Delta^1, \AlgMod_k^\an) \\
(A \to A / I) &\mapsto ((A, \derL \gr^1_\adic(A \to A / I)) \to (A, I / I^2))
\end{align*}
We define the functor
\begin{align*}
\kappa \colon \Pair_k \to \Fun(\Delta^1, \AlgMod_k^\an)
\end{align*}
defined as the (pointwise) composition $\psi_1 \circ \psi \circ \psi_0$. 
\end{construction}
For $R' \to R$ a surjective map of $k$-algebras with kernel $I$, we thus get a map 
\[
(p_1 \circ \kappa)(R' \to R) \colon \LL_{R / k}[-1] \to I / I^2
\] 
in $\derD(R)_{\geq 0}$, which maps to $\kappa_{R / R' / k}$ under the forgetful functor to $\derD(k)$.

\begin{lemma}\label{c_k_compare}
Let $R' \to R$ be a surjective map with kernel $I$. Denote with $c_{R / R'} \colon I[1] \to \LL_{R / R'}$ the map induced on cofibers of the commutative square
\[
\begin{tikzcd}
\dar R' \rar& R \dar \\
\dR_{R / R'} / \Fil^2 \rar& R
\end{tikzcd}
\]
Then the composition
\[
I[1] \xrightarrow{c_{R / R'}} \LL_{R/R'} \xrightarrow{\kappa_{R / R' / k}} I / I^2[1]
\]
in $\derD(R')$ is (the suspension of) the natural projection $I \to I / I^2$. 
\end{lemma}
\begin{proof}
First observe that the composition
\[
R' \to \dR_{R / R'} / \Fil^2 \to R' / (\derL \Fil^2_\adic (R' \to R)) \to R' / I^2
\]
is the natural projection. The result follows by applying the functor $\cofib(- \to R)$ to the composition above.
\end{proof}

\begin{lemma}\label{lem_kodaira_base_change}
Let $k$ be a ring, and let $R' \to R$ be a surjective ring map of $k$-algebras with kernel $I$. For $A' \in \CAlg_{R'}^\an$, write $J = A'  \otimes_{R'} I$ and $A = A' \otimes_{R'} R$, so that $(A' \to A) \in \AniPair_k$. There exists a commutative diagram
\[
\begin{tikzcd}
\dar{\id \otimes \kappa_{R / R' / R'}} A \otimes_{R} \LL_{R / R'}[-1] \rar{\sim} & \LL_{A / A'}[-1] \dar{\kappa_{A / A' / R'}}  \\
A \otimes_{R} I \rar{\sim} & J
\end{tikzcd}
\]
in $\AlgMod_{R'}^\an$, functorial in $A' \in \CAlg_{R'}^\an$. 
\end{lemma}
\begin{proof}
Follows since the equivalence in Proposition \ref{surj_dr_adfil} commutes with coproducts. 
\end{proof}
\begin{definition}
Let $k$ be a ring, $R' \to R$ a surjective ring map with kernel $I$ and $X'$ a smooth scheme over $R'$. Write $X = X' \times_{\Spec(R')} \Spec(R)$. For any smooth $R'$-algebra $A'$ and any map $\Spec(A') \to X'$ the Kodaira--Spencer map from Construction \ref{ks_map_mod} defines a map
\[
\LL_{A / k}[-1] \to I / I^2 \otimes_R A
\]
in $\derD(A)$, where $A := R' \otimes_R A$, functorial in $A'$. Hence we obtain a map
\[
\LL_{X / k}[-1] \to I / I^2 \otimes_R \catO_X
\]
in $\derD(X)$. We obtain a class
\[
\kappa_{X / X' / k} \in \Ext^1_X(\LL_{X / k}, I / I^2 \otimes_R \catO_X)
\]
which we call the \emph{Kodaira--Spencer class}.
\end{definition}
\subsection{The computation for a square-zero extension}
In this section we do the main computation relating the Kodaira--Spencer map with an abstract stratifying map to handle both the characteristic zero and the $p$-adic case at once. We start by introducing the latter, for this we need some setup.

\begin{situation}[Square-zero deformation context]\label{sit_sqzero}
Let $k$ be a ring, and let $R' \to R$ be a surjective map of $k$-algebras such that $I = \ker(R' \to R)$ satisfies $I^2 = 0$. Let 
\[
X' \xrightarrow{f} \Spec(R')
\]
 be a smooth morphism of schemes over $k$, and set 
\[
X := X' \times_{\Spec(R')} \Spec(R)
\]
Finally write $\mathcal{I} := I \otimes_R \catO_X$. 
\end{situation}
In this situation, the map
\[
\widehat{\dR}_{X' / R'} \otimesr \widehat{\dR}_{R / R'} \xrightarrow{\sim}  \widehat{\dR}_{X / R'} 
\]
is an equivalence in $\derD(k)_\fil$ by Corollary \ref{corol_kunneth}. We will denote with 
\begin{equation}
K_{X / R'} \colon  \widehat{\dR}_{X / R'}\xrightarrow{\sim}   \widehat{\dR}_{X' / R'} \otimesr \widehat{\dR}_{R / R'} \label{kunneth_sqz}
\end{equation}
the inverse  in $\derD(k)_\fil$.
\begin{definition}\label{defalpha}
In Situation \ref{sit_sqzero}. Denote with $\alpha$ the composition
\[
\widehat{\dR}_{R / R'} \to R' / \Fil^2_\adic = R'
\]
in $\derD(k)_\fil$, where the first map is (\ref{aaa}), $R'$ is equipped with the adic filtration, and the last identity holds since $I^2 = 0$. 
\end{definition}
By definition we may identify
\begin{align}
\gr^0(\alpha) &\simeq (\id \colon R \to R) \nonumber \\ 
\gr^1(\alpha) &\simeq (\kappa_{R / R' / k} \colon \LL_{R / R'}[-1] \to I / I^2) \label{aab}
\end{align}
in $\derD(k)$. 
\begin{definition}\label{acc}
In Situation \ref{sit_sqzero}. Define the \emph{square-zero stratification map}
\[
\alpha_\sqzero \colon \widehat{\dR}_{X / R'} \to \widehat{\dR}_{X' / R'}
\]
in $\derD(k)_\fil$ as the composition
\[
\widehat{\dR}_{X / R'} \xrightarrow{K_{X / R'}} \widehat{\dR}_{X' / R'} \otimesr \widehat{\dR}_{R / R'} \xrightarrow{\id \otimes \alpha} \widehat{\dR}_{X' / R'} \otimesr R' \to \widehat{\dR}_{X' / R'}
\]
where $R'$ is equipped with the adic filtration.
\end{definition}
\begin{lemma}\label{aam}
Let $k$ be a $\QQ$-algebra. Then $\alpha_\sqzero$ is inverse to the isomorphism
\[
\widehat{\dR}_{X' / R'} \to \widehat{\dR}_{X / R'}
\]
in $\derD(k)$ from Remark \ref{aak}. 
\end{lemma}
\begin{proof}
Identifying $\dR_{R' / R'} \simeq R'$ and $R' / \Fil^2_\adic \simeq R'$ (since $I^2 = 0$), by Remark \ref{aai} we get a commutative diagram
\[
\begin{tikzcd}
\dar \dR_{R' / R'}  \rar& R' \\
\dR_{R / R'}  \rar & \dR_{R / R'} / \Fil^2\uar{(\ref{aaa})} 
\end{tikzcd}
\]
Since the bottom map factors through $\widehat{\dR}_{R / R'}$ we get a commutative diagram 
\[
\begin{tikzcd}
\dar \widehat{\dR}_{R' / R'}  \rar& R' \\
\widehat{\dR}_{R / R'}  \rar & \dR_{R / R'} / \Fil^2\uar{(\ref{aaa})} 
\end{tikzcd}
\]
in $\derD(k)$. It follows that the map $\alpha \colon \widehat{\dR}_{R / R'} \to R'$ is inverse to the composition
\[
R' \simeq \widehat{\dR}_{R' / R'} \to \widehat{\dR}_{R / R'}
\]
Now consider the commutative diagram
\[
\begin{tikzcd}
\widehat{\dR}_{X' / R'} \rar& \widehat{\dR}_{X / R'} \\
\uar \widehat{\dR}_{X' / R'} \otimesr \widehat{\dR}_{R' / R'} \rar& \widehat{\dR}_{X' / R'} \otimesr \widehat{\dR}_{R / R'} \uar 
\end{tikzcd}
\]
in which all arrows are isomorphisms. Inverting the bottom and rightmost arrow, and identifying $\widehat{\dR}_{R' / R'} \simeq R'$  we get a commutative diagram
\[
\begin{tikzcd}
\widehat{\dR}_{X' / R'} \rar& \widehat{\dR}_{X / R'} \dar{K_{X / R'}}\\
\uar \widehat{\dR}_{X' / R'} \otimesr R' & \lar{\id \otimes \alpha} \widehat{\dR}_{X' / R'} \otimesr \widehat{\dR}_{R / R'}
\end{tikzcd}
\]
in which all arrows are isomorphisms. Since the composition down--left--up is $\alpha_\sqzero$ by definition, the result follows.
\end{proof}

We now wish to relate the square--zero stratification map to the Kodaira--Spencer map. We isolate an essential ingredient in the following lemma.
\begin{lemma}\label{abstract_sqzero}
Let $\mathcal{C}^\otimes$ be a symmetric monoidal stable $\infty$-category whose tensor product preserves colimits in each variable separately. Assume that $A, B \in \derD(\mathcal{C})_\fil$ are such that $\Fil^j(B) = 0$ for all $j \geq 2$. 

Then for any $i \in \NN$, there exists a commutative diagram
\[
\begin{tikzcd}
\Fil^i(A \otimes B) \dar \rar& \Fil^0 A \otimes  \Fil^0  B \rar& \gr^{[0, i)} A \otimes \Fil^0 B \\
\gr^i(A \otimes B)\rar{\pi^\gr_1} & \gr^{i - 1}(A) \otimes \gr^{1}(B) \rar{\id \otimes \sigma} & \gr^{i - 1}(A) \otimes \Fil^0 B \uar
\end{tikzcd}
\]
in $\mathcal{C}$, where $\pi_1^\gr$ is defined in (\ref{proj_i}),  $\gr^{[0, i)}$ in (\ref{aae}), and $\sigma$ denotes the composition
\[
\gr^{1}(B) \simeq \Fil^1 B \to \Fil^0 B
\]
in $\mathcal{C}$. 
\end{lemma}
\begin{proof}
For any $\ell, j \in \NN$ with $\ell + j \geq i + 1$, one either has $\ell \geq i$ or $j \geq 2$. Thus either $\ell \geq i$ or $\Fil^j(B) = 0$. It follows that the composition
\[
\Fil^\ell(A) \otimes \Fil^j(B)  \to \Fil^0 A \otimes \Fil^0 B \to \gr^{[0, i)} A \otimes \Fil^0 B
\]
is zero. Since
\[
 \Fil^{i + 1}(A \otimes B) = \underset{\ell + j \geq i + 1}{\colim} \Fil^\ell(A) \otimes \Fil^j(B)
\]
it follows that the composition 
\[
\Fil^{i + 1}(A \otimes B)  \to \Fil^0 A \otimes \Fil^0 B \to \gr^{[0, i)} A \otimes \Fil^0 B
\]
is zero. We thus get a diagram
\[
\begin{tikzcd}
\Fil^i(A \otimes B) \dar \rar&  \Fil^0 A \otimes \Fil^0 B \dar  \\
\gr^i(A \otimes B)\rar& \gr^{[0, i)} A \otimes \Fil^0 B
\end{tikzcd}
\]
Since 
\[
\gr^i(A \otimes B) = (\gr^i(A) \otimes \gr^0(B)) \oplus (\gr^{i - 1}(A) \otimes \gr^1(B))
\]
the result follows.
\end{proof}
The following lemma contains the essential computation, relating the square--zero stratification map with the Kodaira--Spencer map.
\begin{lemma}\label{main_comput}
In Situation \ref{sit_sqzero}. There exists a commutative diagram
\begin{equation}
\begin{tikzcd}
\Fil^i\(\widehat{\dR}_{X' / R'} \otimesr \widehat{\dR}_{R / R'}\) \dar  \rar & \gr^i\(\widehat{\dR}_{X' / R'} \otimesr \widehat{\dR}_{R / R'}\) \dar{\pi^\gr_1} & \\
\widehat{\dR}_{X' / R'} \otimesr  \widehat{\dR}_{R / R'}  \dar{q_i \otimes \alpha} &   \LL^{i - 1}_{X' / R'}[1 - i] \underset{R'}{\otimes} \LL_{R / R'} [-1] \arrow{d}{\id \otimes \kappa} \\
\widehat{\dR}_{X' / R'} / \Fil^i \otimesr R' & \LL^{i - 1}_{X' / R'}[1 - i] \underset{R'}{\otimes} I \lar 
\end{tikzcd} \label{aac}
\end{equation}
in $\derD(k)$, where $\kappa$ is shorthand for $\kappa_{R / R' / k}$, and $q_i$ is the quotient map.
\end{lemma}
\begin{proof}
Applying Lemma \ref{abstract_sqzero} (and transposing the diagram) to $A = \Fil^\bullet \widehat{\dR}_{X'/R'}$ and $B = \Fil^\bullet_\adic R'$ we obtain a commutative diagram
\begin{equation}
\begin{tikzcd}
\Fil^i\(\widehat{\dR}_{X' / R'} \otimesr R'\) \dar  \rar & \gr^i\(\widehat{\dR}_{X' / R'} \otimesr R'\)  \dar{\pi^\gr_1}  \\
\widehat{\dR}_{X' / R'} \otimesr R' \dar & \LL^{i - 1}_{X' / R'}[1 - i] \underset{R'}{\otimes} I \dar \\
\widehat{\dR}_{X' / R'} / \Fil^i \otimesr R' & \LL^{i - 1}_{X' / R'}[1 - i] \underset{R'}{\otimes} R' \lar
\end{tikzcd} \label{6231}
\end{equation}
We thus see the existence of the following commutative diagram
\begin{equation}
\begin{tikzcd}
\dar{\id \otimes \alpha}  \Fil^i\(\widehat{\dR}_{X' / R'} \otimesr \widehat{\dR}_{R / R'}\)  \rar &  \gr^i\(\widehat{\dR}_{X' / R'} \otimesr \widehat{\dR}_{R / R'}\) \dar{\id \otimes \alpha}\\
\dar  \rar \Fil^i\(\widehat{\dR}_{X' / R'} \otimesr R'\)  \dar \ar[rd, phantom, "(\ref{6231})" label] & \gr^i\(\widehat{\dR}_{X' / R'} \otimesr R'\)\dar{\pi^\gr_1} \\
\widehat{\dR}_{X' / R'} / \Fil^i \otimesr R' &  \LL^{i - 1}_{X' / R'}[1 - i] \underset{R'}{\otimes} R' \lar
\end{tikzcd}  \label{aad}
\end{equation}
By general properties of the tensor products of filtered objects, the composition of the left vertical arrows in (\ref{aad}) is equivalent to the composition of the left vertical arrows in (\ref{aac}). Since 
\[
\pi_1^\gr \circ (\id \otimes \alpha) \simeq (\id \otimes \gr^1(\alpha)) \circ \pi^1_\gr \simeq (\id \otimes \kappa_{R / R' / k}) \circ \pi_1^\gr
\]
(the second equivalence follows from (\ref{aab})), the result follows.
\end{proof}

We now wish to slightly tweak the above result, to improve our understanding of the map $\pi_1^\gr$. We start by defining the antisymmetrization map 
\[
\Delta_{i - 1} \colon \bigwedge^i M \to  \( \bigwedge^{i - 1}M\)  \otimes M
\]
for any (animated) ring $A$ and any connective $A$-module $M$. 
\begin{definition} \label{def_delta} Let $k$ be a ring. Define the \emph{antisymmetrization map} functor 
\[
\Fun(\Delta^1, \AlgMod_k^\an) \to \Fun(\Delta^1, \AlgMod_k^\an)
\]
as the left Kan extension of the functor
\begin{align*}
\Fun(\Delta^1, \AlgMod_k^\heart)_\gen &\to \Fun(\Delta^1,  \AlgMod_k^\heart)  \\
(A, M) &\mapsto \(A, \bigwedge^i(M) \to  \bigwedge^{i - 1}(M)  \otimes M\)
\end{align*}
where $\Delta_{i - 1}$ is given by
\begin{align*}
m_1 \wedge \dots \wedge m_i &\mapsto \sum_{k = 1}^i (-1)^k m_1 \wedge \dots \wedge \widehat{m_k} \wedge \dots \wedge m_i \otimes m_k
\end{align*}
where by $\widehat{m_k} $ we mean that $m_k$ does not appear in the wedge product.
\end{definition}
For any fixed ring $R'$, by right Kan extension we obtain for any stack $X$ over $k$ a map
\[
\Delta_{i - 1} \colon \LL^i_{X / R'} \to \LL^{i - 1}_{X / R'} \underset{\catO_X}{\otimes} \LL_{X / R'}
\]
in $\derD(X)$. 

\begin{definition}
Let $k$ be a ring. For any two objects $A, B \in \CAlg_k^\an$, write $C = A \otimes_k B$ and denote with $\beta$ the composition
\[
\LL^i_{A / k} \underset k \otimes \LL_{B / k} \xrightarrow{\sim} (\LL^{i - 1}_{A / k} \underset A \otimes C) \underset C \otimes (C \underset B \otimes \LL_{B /k}) \xrightarrow{\sim} \LL_{C / B}^{i - 1} \underset C \otimes \LL_{C / A}
\]
clearly $\beta$ is an isomorphism. Define the map 
\[
\beta_{C / k} \colon \LL^{i - 1}_{C / k} \underset{C}{\otimes} \LL_{C / k} \to \LL^{i - 1}_{A / k} \underset{k}\otimes \LL_{B / k}
\]as the unique map fitting in a commutative diagram
\begin{equation}
\begin{tikzcd}
\LL^{i - 1}_{C / k} \underset{C} \otimes \LL_{C / k} \dar \rar{\beta_{C / k}} & \LL^{i - 1}_{A / k} \underset{k} \otimes \LL_{B / k} \arrow{dl}{\beta}[swap]{\sim} \\
\LL^{i - 1}_{C / B} \underset{C} \otimes \LL_{C / A}
\end{tikzcd}\label{aaf}
\end{equation}
in $\derD(k)$. 
\end{definition}
If $A = k[x_1, \dots, x_n]$ and $B = k[x_{n + 1}, \dots, x_{n + m}]$ are polynomial algebras over $k$, then a $k$-basis for the module $\LL_{C / k}^i \otimes \LL_{C / k}$ can be given by the set
\[
\left \{ \(\prod_{k = 1}^{n + m} x_k^{a_k} \) \cdot \mathrm{d} x_{v_1} \wedge \dots \wedge \mathrm{d}x_{v_i} \otimes \mathrm{d}x_\ell \ \bigg \rvert \ \substack{a_i \geq 0 \\
0 \leq v_1 < \dots < v_i \leq n + m \\
0 \leq \ell \leq n + m} \right \}
\]
One may check the map $\beta_{C / k}$ is the unique $k$-linear map sending
\[
\(\prod_{k = 1}^{n + m} x_k^{a_k} \) \cdot \mathrm{d} x_{v_1} \wedge \dots \wedge \mathrm{d}x_{v_i} \otimes \mathrm{d}x_\ell 
\]
to
\[
\(\(\prod_{k = 1}^{n} x_k^{a_k} \) \cdot \mathrm{d} x_{v_1} \wedge \dots \wedge \mathrm{d}x_{v_i}\) \otimes \(\(\prod_{k = n + 1}^{n + m} x_k^{a_k}\) \mathrm{d}x_\ell\)
\]
if $v_k \leq n$ for all $k \in \{1, \dots, n\}$ and $\ell \geq n + 1$, and sending everything else to $0$. 

By right Kan extending from (\ref{aaf}) and using (\ref{def_extend2}), we obtain for $k$-algebra $R$ and any smooth scheme $X$ over $k$ a diagram
\[
\begin{tikzcd}
\LL^{i - 1}_{X_R / k} \underset{\catO_{X_R}} \otimes \LL_{X_R / k} \dar \rar{\beta_{X_R / k}} & \LL^{i - 1}_{X / k} \underset{k} \otimes \LL_{R / k} \arrow{dl}{\beta}[swap]{\sim} \\
\LL^{i - 1}_{X_R / R} \underset{\catO_{X_R}} \otimes \LL_{X_R / X}
\end{tikzcd}
\]
in $\derD(k)$. In particular, in Situation \ref{sit_sqzero} we obtain a diagram
\begin{equation}
\begin{tikzcd}
\LL^{i - 1}_{X / R'}[1 - i] \underset{\catO_X} \otimes \LL_{X / R'}[-1] \dar \rar{\beta_{X / R'}} & \LL^{i - 1}_{X' / R'}[1 - i] \underset{R'} \otimes \LL_{R / R'}[-1] \arrow{dl}{\beta}[swap]{\sim} \\
\LL^{i - 1}_{X / R}[1 - i] \underset{\catO_X} \otimes \LL_{X / X'}[-1]
\end{tikzcd}\label{aag}
\end{equation}
in $\derD(k)$. 
\begin{lemma}\label{beta_lemma}
In Situation \ref{sit_sqzero}. There exists a commutative diagram
\[
\begin{tikzcd}
\LL^i_{X / R'}[-i] \dar{\Delta_{i - 1}} & \lar{\sim}[swap]{\gamma} \dar{\pi_1^\gr} \gr^i\(\widehat{\dR}_{X' / R'} \otimesr \widehat{\dR}_{R / R'}\)  \\
\LL^{i - 1}_{X / R'}[1 - i] \underset{\catO_X}{\otimes} \LL_{X / R'}[-1] \rar{\beta_{X / R'}}& \LL^{i - 1}_{X' / R'}[1 - i] \underset{R'}{\otimes} \LL_{R / R'}[-1]
\end{tikzcd}
\]
in $\derD(k)$, where $\gamma := \gr^i(K_{X / R'})$ is the $i$-th graded piece of the K\"unneth isomorphism (\ref{kunneth_sqz}). 
\end{lemma}
\begin{proof}
Since all corners of the square define $\derD(R')$-valued sheaves on $\St_{R'}$ it suffices to construct the diagram in the case $X = \Spec(A')$ for some smooth $R'$-algebra $A'$, functorially in $A'$. 

It thus suffices to construct a functor
\[
\Fun(\Lambda^2_0, \CAlg_{k}^\an) \to \Fun(\Delta^1 \times \Delta^1, \derD(k))
\]
given on $(A' \leftarrow R' \rightarrow R) \in \Fun(\Lambda^2_0, \CAlg_{k}^\an)$ by 
\[
\begin{tikzcd}
\LL^i_{A / R'} \dar{\Delta_{i - 1}} &  \lar{\sim} \bigoplus_{k = 0}^p \LL^{i - k}_{A' / R'} \underset{R'}\otimes \LL^k_{R / R'}\dar{\pi^\gr_1} \\
\LL^{i - 1}_{A / R'} \underset{A}\otimes \LL_{A / R'} \rar{\beta_{A / R'}} & \LL^{i - 1}_{A' / R'} \otimes_{R'} \LL_{R / R'}
\end{tikzcd}
\]
where we write $A = A' \otimes_{R'} R$. 

We leave it to the reader to verify that for any
\[
(A' \leftarrow R' \to R) \in \Fun(\Delta^1, \Poly_k)_\gen \times_{\Poly_k} \Fun(\Delta^1, \Poly_k)_\gen
\]
the diagram
\[
\begin{tikzcd}
\Omega^i_{A / R'} \dar{\Delta_{i - 1}} &  \lar{\sim} \bigoplus_{k = 0}^i \Omega^{i - k}_{A' / R'} \underset{R'}\otimes \Omega^k_{R / R'}\dar{\pi^\gr_1} \\
\Omega^{i - 1}_{A / R'} \underset{A}\otimes \Omega_{A / R'} \rar{\beta_{A / R'}} & \Omega^{i - 1}_{A' / R'} \otimes_{R'} \Omega_{R / R'}
\end{tikzcd}
\]
in $\derD(k)^\heart$ commutes, where again $A = A' \otimes_{R'} R$. The desired functor is then obtained by left Kan extension.
\end{proof}
\begin{prop}\label{kappa_comp_sqz}
In Situation \ref{sit_sqzero}. There exists a commutative diagram
\[
\begin{tikzcd}
\dar \Fil^i \widehat{\dR}_{X / R'} \rar & \LL^i_{X / R'}[-i] \dar{\Delta_{i - 1}}  \\ 
\dar{\alpha_\sqzero} \widehat{\dR}_{X / R'} & \LL_{X / R}^{i - 1}[1 - i] \underset{\catO_X}{\otimes} \LL_{X / X'} [-1] \dar{\id \otimes \kappa}\\
\widehat{\dR}_{X' / R'} \dar & \LL^{i- 1}_{X / R}[1 - i] \underset{\catO_X}{\otimes} \mathcal{I} \dar{\sim}\\
\widehat{\dR}_{X' / R'} / \Fil^i &  \LL^{i- 1}_{X / R}[1 - i] \underset{R}{\otimes} I \lar
\end{tikzcd}
\]
in $\derD(k)$, where $\kappa = \kappa_{X / X' / k}$. 
\end{prop}
\begin{proof}
Globalizing Lemma \ref{lem_kodaira_base_change} and combining it with Lemma \ref{beta_lemma} and (\ref{aag}), we obtain a commutative diagram
\begin{equation}
\begin{tikzcd}
\rar{\sim} \gr^i\(\widehat{\dR}_{X' / R'} \otimesr \widehat{\dR}_{R / R'}\) \dar{\pi^\gr_1} \ar[rd, phantom, "\text{Lemma \ref{beta_lemma}}" label]  & \LL^i_{X / R'} \dar{\Delta_{i - 1}}  \\
\LL^{i - 1}_{X' / R'}[1 - i] \underset{R'} \otimes \LL_{R / R'}[-1] \arrow{dr}{\beta}[swap]{\sim} \arrow{dd}{\id \otimes \kappa_{R / R' / k}} &  \LL^{i - 1}_{X / R'}[1 - i] \underset{\catO_X} \otimes \LL_{X / R'}[-1] \dar \lar[swap]{\beta_{X / R'}}\\
 \ar[rd, phantom, "\text{Lemma \ref{lem_kodaira_base_change}}" label] & \LL^{i - 1}_{X / R}[1 - i] \underset{\catO_X} \otimes \LL_{X / X'}[-1] \dar{\id \otimes \kappa_{X / X' / k}} \\
\LL^{i - 1}_{X / R}[1 - i] \underset{R} \otimes I \rar{\sim} & \LL^{i - 1}_{X / R}[1 - i] \underset{\catO_X} \otimes \mathcal{I} \end{tikzcd}\label{aah}
\end{equation}
in $\derD(k)$. Additionally, by Lemma \ref{main_comput} we see the existence of the commutative diagram
\[
\begin{tikzcd} 
\Fil^i \widehat{\dR}_{X / R'} \rar \dar{K_{X / R'}}  & \LL^i_{X / R'}[-i] \dar{K_{X / R'}} \\
\ar[dr, phantom, "\text{Lemma \ref{main_comput}}" label] \Fil^i \(\widehat{\dR}_{X' / R'} \otimesr \widehat{\dR}_{R / R'}\)  \rar \dar & \gr^i\(\widehat{\dR}_{X' / R'} \otimesr \widehat{\dR}_{R / R'}\) \dar{\pi^\gr_1} \\
\ar[dr, phantom, "" label] \widehat{\dR}_{X' / R'} \otimesr \widehat{\dR}_{R / R'} \dar{q_i \otimes \alpha} & \LL^{i - 1}_{X' / R'}[1 - i] \underset{R'}{\otimes} \LL_{R / R'} [-1]  \arrow{d}{\id \otimes \kappa} \\
\widehat{\dR}_{X' / R'} / \Fil^i \otimesr R' & \LL^{i - 1}_{X' / R'}[1 - i] \underset{R'}{\otimes} I \lar \\
\end{tikzcd}
\]
in $\derD(k)$. The result follows by combining the two above diagrams, together with the definition of $\alpha_\sqzero$. 
\end{proof}

\subsection{The obstruction class as a cup product with the Kodaira--Spencer class in characteristic zero}
\label{sec_compute_char}
In this section we refine Bloch's computation, expressing the Hodge-theoretic obstruction class as a cup product with a (derived) Kodaira--Spencer class. The following essentially is a refinement of \cite[Proposition 3.6]{bloch}.
\begin{prop}\label{label_chern_result}
Let $k$ be a ring such that $\QQ \subseteq k$, and suppose that $R' \to R$ is a square zero morphism of nilpotent thickenings with $I = \ker(R' \to R)$. Let 
\[
X'  \xrightarrow{f} \Spec(R')
\]
be a smooth and proper morphism, and set
\begin{align*}
X &:= X'  \times_{\Spec(R' )} \Spec(R) \\
X_0 &:= X'  \times_{\Spec(R' )} \Spec(k)
\end{align*}
Write $\mathcal{I} := I \otimes_{R} \catO_{X}$. 
Let $v \in \coh^{2i}(\Fil^i \widehat{\dR}_{X / k})$. Let $v_0 \in \coh^{2i}(\Fil^i \widehat{\dR}_{X_0 / k})$ be the image of $v_0$.

Then the composition
\begin{align*}
\coh^{2i} \(\Fil^i \widehat{\dR}_{X / k}\) &\to  \coh^{i} \(\LL^i_{X / k}\) \\
&\to  \coh^{i} \(\LL^i_{X / R'}\) \\
&\xrightarrow{\Delta_{i - 1}}  \coh^{i} \(\LL^{i - 1}_{X / R'} \underset{\catO_X} \otimes \LL_{X / R'}\) \\
&\to  \coh^{i} \(\Omega^{i - 1}_{X / R} \underset{\catO_X} \otimes \LL_{X / X'}\) \\
&\xrightarrow{\id \otimes {\kappa_{X / X'  / R'}}} \coh^{i + 1}\( \Omega^{i - 1}_{X / R} \underset{\catO_X} \otimes \mathcal{I}\)\\
&\to \coh^{2i}\( \widehat{\dR}_{X' / R'} / \Fil^i\) \underset{R'} \otimes I
\end{align*}
maps $v$ to $\ob^{\widehat{\dR}}_{X' / R'}(v_0)$. 
\end{prop}
The reader should keep the following example in mind (which will be the only application for us): $\mathcal{E}$ is a vector bundle on $X$ with Chern character $v = \ch_i(\mathcal{E})$, so that $v_0 = \ch_i(\mathcal{E}\rvert_{X_0})$. 
\begin{proof}[Proof of Proposition \ref{label_chern_result}]
Consider the commutative diagram
\[
\begin{tikzcd}
\widehat{\dR}_{X / R'} \rar{\sim} & \widehat{\dR}_{X_0 / R'} \\
&  \widehat{\dR}_{X' / R'} \uar{\sim} \arrow{ul}{\sim}
\end{tikzcd}
\]
where all arrows are seen to be isomorphisms by Remark \ref{aak}. Inverting the two vertical arrows (using Lemma \ref{aam}), we obtain a commutative diagram
\[
\begin{tikzcd}
\Fil^i \widehat{\dR}_{X / R'} \rar & \widehat{\dR}_{X / R'} \arrow{dr}{\sim}[swap]{\alpha_\sqzero} \rar{\sim} & \arrow{d}{\sim} \widehat{\dR}_{X_0 / R'} & \lar{\sim} \dR_{X_0 / k} \otimes_k R' \arrow{dl}{\varphi_{\widehat{\dR}, X'}} \\
& & \widehat{\dR}_{X' / R'}
\end{tikzcd}
\]
where the right triangle exists by definition of $\varphi_{\widehat{\dR}, X'}$ (see Definition \ref{aal}). 

Applying $\coh^{2i}(-)$ and chasing $v$ through the diagram we see that the obstruction class $\ob^{\widehat{\dR}}_{X' / R'}(v_0)$ is equal to the image of $v$ under the composition
\[
\coh^{2i} (\Fil^i \widehat{\dR}_{X / R'}) \to \coh^{2i} (\widehat{\dR}_{X / R'}) \xrightarrow{\alpha_\sqzero} \coh^{2i}( \widehat{\dR}_{X' / R'}) \to \coh^{2i}(\widehat{\dR}_{X' / R'} / \Fil^i)
\]
Thus the result follows from Proposition \ref{kappa_comp_sqz}.
\end{proof}

\subsection{The obstruction class as a cup product with the Kodaira--Spencer class in the $p$-adic case}
In this section we state the analogue of the result in Section \ref{sec_compute_char} for the $p$-adic case. 
\begin{prop}\label{label_chern_result_pd}
Let $k$ be a ring over $\ZZ / p^n \ZZ$ for some $n \geq 1$, and let $R_0$ be a $k$-algebra. Let 
\[
(R' \to R_0, \gamma' ) \to (R \to R_0, \gamma)
\]
 be a morphism in $\PDPair_k$ such that $R' \to R$ is a surjection with kernel $I$ and $I^{[2]} = 0$. 

Let $X'  \xrightarrow{f} \Spec(R')$ be a smooth and proper morphism, and set
\begin{align*}
X &:= X'  \times_{\Spec(R' )} \Spec(R) \\
X_0 &:= X'  \times_{\Spec(R' )} \Spec(R_0)
\end{align*}
Write $\mathcal{I} := I \otimes_{R} \catO_{X}$. Let $v \in \coh^{2i}(\Fil^i \dR_{X / k})$, write $v_0 \in \coh^{2i}(\Fil^i \dR_{X_0 / k})$ for the image of $v_0$.
 
Then the composition
\begin{align*}
\coh^{2i} \(\Fil^i \dR_{X / k}\) &\to  \coh^{i} \(\LL^i_{X / k}\) \\
 &\to  \coh^{i} \(\LL^i_{X / R'}\) \\
&\xrightarrow{\Delta_{i - 1}}  \coh^{i} \(\LL^{i - 1}_{X / R'} \underset{\catO_X} \otimes \LL_{X / R'}\) \\
&\to  \coh^{i} \(\Omega^{i - 1}_{X / R} \underset{\catO_X} \otimes \LL_{X / X'}\) \\
&\xrightarrow{\id \otimes {\kappa_{X / X'  / R'}}} \coh^{i + 1}\(\Omega^{i - 1}_{X / R} \otimes \mathcal{I}\) \\
&\to \coh^{2i}\( \widehat{\dR}_{X' / R'} / \Fil^i \otimes \mathcal{I}\)
\end{align*}
maps $v$ to $\ob^\Crys_{X' / R'}(v_0)$. 
\end{prop}
\begin{proof}
Recall that the derived divided power envelope functor (see Definition \ref{aba})
\[
(-)^\Lenv \colon \AniPair_k \to \AniPDPair_k 
\]
admits a right adjoint 
\[
\forget \colon \AniPDPair_k \to \AniPair_k
\]
We will denote the unit for this adjunction with $\eta$ and the counit with $\epsilon$. 

We may consider $(R' \to R)$ as an object in $\PDPair_k \subseteq \AniPDPair_k$ by giving it the trivial PD--structure (since $I^{[2]} = 0$). Write
\[
(T' \to R) := (\forget(R' \to R))^\Lenv
\]
which lives in $\AniPDPair_k$. 

By the triangle identity for an adjunction, the composition of the counit and unit 
\[
\forget(R' \to R) \xrightarrow{\eta} \forget(T' \to R) \xrightarrow{\epsilon} \forget(R' \to R)
\]
is equivalent to the identity in $\AniPair_k$, so that the composition
\[
\derL\Fil_\adic(R' \to R) \xrightarrow{\eta} \derL\Fil_\adic(T' \to R) \xrightarrow{\epsilon} \derL\Fil_\adic(R' \to R)
\]
is equivalent to the identity in $\CAlg_\fil(k)$. We thus get a commutative diagram
\[
\begin{tikzcd}
\derL \Fil_\pdadic (T' \to R) \dar{\epsilon} & \lar[swap]{(\ref{abb})} \dar{\epsilon} \derL \Fil_\adic (T' \to R) & \derL \Fil_\adic (R' \to R) \lar[swap]{\eta} \arrow{dl}{\id} \\
\derL \Fil_\pdadic (R' \to R) & \lar[swap]{(\ref{abb})} \derL \Fil_\adic (R' \to R)
\end{tikzcd}
\]
in $\CAlg_\fil(k)$, inducing a commutative diagram
\[
\begin{tikzcd}
\derL \gr^{[0, 2)}_\pdadic (T' \to R) \dar{\epsilon} & \derL \gr^{[0, 2)}_\adic (R' \to R) \lar{\sim}[swap]{(\ref{abc})} \arrow{dl}{(\ref{abb})} \\
\derL \gr^{[0, 2)}_\pdadic (R' \to R)
\end{tikzcd}
\]
in $\derD(k)$, where the horizontal map is an equivalence by Lemma \ref{compare_01}. We thus get a commutative diagram
\begin{equation}
\begin{tikzcd}
\dar{(\ref{drpdcompare})}[swap]{\sim}\dR_{R / R'} / \Fil^2 \arrow{dr}{(\ref{dr_filadic})}  \\
  \derL \gr^{[0, 2)}_\pdadic (T' \to R) \dar{\epsilon} & \derL \gr^{[0, 2)}_\adic (R' \to R) \lar{\sim}[swap]{(\ref{abc})} \arrow{dl}{(\ref{abb})} \\
 \derL \gr^{[0, 2)}_\pdadic (R' \to R)
\end{tikzcd}\label{abe}
\end{equation}
in $\derD(k)$. 

On the other hand, the commutative diagram
\[
\begin{tikzcd}[column sep = tiny]
{\dR}_{R / R'} \dar & \lar[swap]{\sim} \dar \dR_{(T' \to R) / (R' \to R')} \rar& \dar \dR_{(T' \to R) / (T' \to R)}  \rar{\sim} &  \derL \Fil_\pdadic (T' \to R)\dar{\id}  \\
\Crys_{R / (R' \to R)} \rar{\id} & \Crys_{R / (R' \to R)} \rar{\eta} & \Crys_{R / (T' \to R)}   \rar{\sim} & \derL \Fil_\pdadic (T' \to R)
\end{tikzcd}
\]
induces a commutative diagram
\[
\begin{tikzcd}
{\dR}_{R / R'} \dar \rar{(\ref{drpdcompare})} &  \derL \Fil_\pdadic (T' \to R)\dar{\id}  \\
\Crys_{R / (R' \to R)} \rar{\eta} &  \derL \Fil_\pdadic (T' \to R)
\end{tikzcd}
\]
and hence a commutative diagram
\begin{equation}
\begin{tikzcd}
{\dR}_{R / R'} \dar \rar{(\ref{drpdcompare})} &  \derL \Fil_\pdadic (T' \to R)\dar{\epsilon}  \\
\Crys_{R / (R' \to R)} \rar{\sim} &  \derL \Fil_\pdadic (R' \to R)
\end{tikzcd}\label{abf}
\end{equation}
in $\CAlg_\fil(k)$. Applying $\gr^{[0, 2)}(-)$ to the diagram (\ref{abf}) and combining it with the diagram (\ref{abe}), we get a commutative diagram
\begin{equation}
\begin{tikzcd}
\arrow{r}{(\ref{dr_filadic})} \dR_{R / R'} / \Fil^2 \dar & \derL \gr^{[0, 2)}_\adic (R' \to R) \arrow{d}{(\ref{abb})} \\
\rar{\sim} \Crys_{R / (R' \to R)} / \Fil^2 & \derL \gr^{[0, 2)}_\pdadic (R' \to R)
\end{tikzcd}
\end{equation}
in $\derD(k)$, and hence a commutative diagram
\begin{equation}
\begin{tikzcd}
\arrow{r}{(\ref{aaa})} \dR_{R / R'} / \Fil^2 \dar &  \gr^{[0, 2)}_\adic (R' \to R) \arrow{d} \\
\rar{\sim} \Crys_{R / (R' \to R)} / \Fil^2 & \gr^{[0, 2)}_\pdadic (R' \to R)
\end{tikzcd}\label{abg}
\end{equation}
in $\derD(k)$. Using that $I^{[2]} = 0$, we may identify 
\[
\gr^{[0, 2)}_\adic (R' \to R) \simeq \gr^{[0, 2)}_\pdadic (R' \to R) \simeq R'
\]
so that by moving around the arrows in (\ref{abg}) we obtain a commutative diagram
\[
\begin{tikzcd}
\dR_{R / R'} / \Fil^2 \arrow{dr}{(\ref{aaa})} \rar & \Crys_{R / (R' \to R)} / \Fil^2  \arrow{d}{\sim} \\
& R'
\end{tikzcd}
\]
in $\derD(k)$. Denote with $\alpha$ also the composition
\[
\dR_{R / R'} \to \widehat{\dR}_{R / R'} \xrightarrow{\alpha} R'
\]
so that by definition of $\alpha$ (Definition \ref{defalpha}) we obtain a commutative diagram
\begin{equation}
\begin{tikzcd}
\dR_{R / R'} \arrow{dr}{\alpha} \rar& \Crys_{R / (R' \to R)} \arrow{d}{\sim} \\
& R'
\end{tikzcd}\label{abm}
\end{equation}
in $\derD(k)$. Applying $\dR_{X' / R'} \underset{R'} \otimes (-)$ to the diagram (\ref{abm}), we obtain a diagram
\begin{equation}
\begin{tikzcd}
\dR_{X' / R'} \otimes \dR_{R / R'} \arrow{dr}[swap]{\id \otimes \alpha} \rar & \dR_{X' / R'} \underset{R'} \otimes \Crys_{R / (R' \to R)} \dar{\sim} & \lar{\sim} \Crys_{X / (R' \to R)} \arrow{dl}{\alpha_\Crys} \\
& \dR_{X' / R'} \underset{R'} \otimes R'
\end{tikzcd}\label{abo}
\end{equation}
in $\derD(k)$, where the right triangle comes by definition of $\alpha_\Crys$, see (\ref{abh}).

Define $\alpha_X$ as the unique map fitting in a commutative diagram
\begin{equation}
\begin{tikzcd}
 \dR_{X' / R'} \otimes \dR_{R / R'} \dar{(\ref{kunneth})}[swap]{\sim}  \rar{\id \otimes \alpha} &  \dR_{X' / R'} \otimes R' \dar{\sim} \\
 \dR_{X / R'} \rar{\alpha_X} &  \dR_{X' / R'}
\end{tikzcd}  \label{abl}
\end{equation}
in $\derD(k)$. By definition of $\alpha_\sqzero$ (see Definition \ref{acc}), we have a commutative diagram
\begin{equation}
\begin{tikzcd}
\dar \dR_{X / R'} \rar{\alpha_X} &  \dR_{X' / R'} \dar \\
 \widehat{\dR}_{X / R'} \rar{\alpha_\sqzero} &  \widehat{\dR}_{X' / R'}
\end{tikzcd}  \label{acb}
\end{equation}
in $\derD(k)$. 

Combining (\ref{abo}) and (\ref{abl}) we get a diagram
\begin{equation}
\begin{tikzcd}
\dR_{X / R'} \arrow{dr}[swap]{\alpha_X} \rar &  \Crys_{X / (R' \to R)} \arrow{d}{\alpha_\Crys} \\
  & \dR_{X' / R'}
\end{tikzcd}\label{abj}
\end{equation}
in $\derD(k)$. 

On the other hand, the commutative diagram
\[
\begin{tikzcd}
\Crys_{R / (R' \to R)} \rar{\sim} & \Crys_{R_0 / (R' \to R_0)} \\
\dR_{R' / R'} 	 `\uar{\sim} \arrow{ur}{\sim} 
\end{tikzcd}
\]
in $\derD(k)$ induces a commutative diagram 
\[\begin{tikzcd}
\Crys_{X / (R' \to R)} \dar{\alpha_\Crys} \rar{\sim} & \Crys_{X_0 / (R' \to R_0)} \arrow{dl}{\alpha_\Crys} \\
\dR_{X' / R}
\end{tikzcd}
\]
in $\derD(k)$. Combining the above diagram with the diagram (\ref{abj}), we get a commutative diagram
\[
\begin{tikzcd}
\dR_{X / R'} \arrow{dr}[swap]{\alpha_X} \rar &  \Crys_{X_0 / (R' \to R_0)} \arrow{d}{\alpha_\Crys}[swap]{\sim}  \\
  & \dR_{X' / R'} 
\end{tikzcd}
\]
in $\derD(k)$. By definition of $\varphi_{{\Crys}, X'}$ (see Definition \ref{abk}), we thus get a commutative diagram
\begin{equation} \label{aca}
\begin{tikzcd}
\dR_{X / k} \rar \dar &  \dR_{X_0 / k} \dar \arrow[bend left = 30]{ddr}{\varphi_{{\Crys}, X'}}  \\
\dR_{X / R'} \rar \arrow[bend right = 15]{drr}{}[swap]{\alpha_X} &  \Crys_{X_0 / (R' \to R_0)} \arrow{dr}{\alpha_\Crys}[swap]{\sim}  \\
 & & \dR_{X' / R'} 
\end{tikzcd}
\end{equation}
in $\derD(k)$. We thus get a commutative diagram
\[
\begin{tikzcd}[column sep = small]
\dar \coh^{2i}(\Fil^i \dR_{X / k}) \rar&  \rar \dar \coh^{2i}(\dR_{X / k}) \ar[rd, phantom, "\text{(\ref{aca})}" label] &  \coh^{2i}(\dR_{X_0 / k}) \dar{\varphi_{{\Crys}, X'}}  \\
\dar \coh^{2i}(\Fil^i \dR_{X / R'}) \rar  & \ar[rd, phantom, "\text{(\ref{acb})}"] \dar \coh^{2i}(\dR_{X / R'}) \rar{\alpha_X} & \dar \coh^{2i}(\dR_{X' / R'})  \arrow{dr} \\
\coh^{2i}(\Fil^i \widehat{\dR}_{X / R'}) \rar  &  \coh^{2i}(\widehat{\dR}_{X / R'}) \rar{\alpha_\sqzero} & \coh^{2i}(\widehat{\dR}_{X' / R'})  \rar & \coh^{2i}(\widehat{\dR}_{X' / R'} / \Fil^i)
\end{tikzcd}
\]
Tracing $v$ around the edges of the diagram above, it follows that $\ob^{\Crys}_{X' / R'}(v_0)$ is equal to the image of $v$ under the composition
\begin{align*}
\coh^{2i}(\Fil^i \dR_{X / k}) & \to \coh^{2i}(\Fil^i \widehat{\dR}_{X / R'}) \to  \coh^{2i}(\widehat{\dR}_{X / R'}) \\ & \xrightarrow{\alpha_\sqzero} \coh^{2i}(\widehat{\dR}_{X' / R'})  \to \coh^{2i}(\widehat{\dR}_{X' / R'} / \Fil^i)
\end{align*}
The result then follows from Proposition \ref{kappa_comp_sqz}.
\end{proof}

\section{Obstruction theory for complexes}
The goal of this section is to study the obstruction class to deforming a complex, and relate it to the Hodge-theoretic obstruction class of its Chern class. 

In Section \ref{sec_univ_der}, we study square zero extensions of rings, expressing them as a pullback square involving the Kodaira--Spencer class and the de Rham differential. In Section \ref{sec_mod_sq_zero} we study the derived category of modules over a split square zero extension of rings. These two sections are rather technical and only needed for the proofs in Section \ref{sec_ob_complex}.

In Section \ref{sec_atiyah}, we define the Atiyah class and show that its trace equals the Chern character from Definition \ref{def_chern_character}. Then in Section \ref{sec_ob_complex} we define the obstruction class to deforming a complex along a square zero extension, and express it as a product of the Kodaira--Spencer class with the Atiyah class. Finally in Section \ref{sec_relate} we relate with the Hodge-theoretic obstruction class of its Chern class, by means of the semiregularity map. 
\subsection{The universal derivation and square zero extensions}
\label{sec_univ_der}
The main purpose of this technical section is the proof of Proposition \ref{prop_pb_ani_ring}, relating square zero extensions of rings with the Kodaira--Spencer class and the de Rham differential. 
\begin{definition}\label{def_univ_der}
Let $k$ be a ring. We define the universal derivation 
\[
\delta \colon \Fun(\Delta^1, \CAlg_k^\an) \to \Fun(\Delta^1, \CAlg_k^\an)
\]
as the left derived functor of the functor
\begin{align*}
\Fun(\Delta^1, \Poly_k)_\gen &\to \Fun(\Delta^1, \Alg_k) \\
\(P \to Q\) &\mapsto \(Q \xrightarrow{(\id, \mathrm{d})} Q \oplus \Omega_{Q / P}\)
\end{align*}
where the multiplication on $Q \oplus \Omega_{Q / P}$ is given by
\[
(x, \omega) \cdot (y, \eta) := (xy, x \eta + y \omega)
\]
By right Kan extension, we obtain for any stack $X$ over $k$ a map 
\[
\catO_X \xrightarrow{\delta} \catO_X \oplus \LL_{X / k}
\]
in $\Shv_{\CAlg_k}(X)$. 
\end{definition}
In particular, for any map of (animated) rings $A \to B$, we obtain a map $\delta \colon B \to B \oplus \LL_{B / A}$ in $\CAlg_k^\an$, which is informally given by sending $b \mapsto (b, \mathrm{d}b)$.
\begin{lemma}\label{lem_com_square_der}
Let $k$ be a ring and let $(A \to B) \in \Fun(\Delta^1, \CAlg_k^\an)$. There exists a commutative diagram 
\[
\begin{tikzcd}
A \dar \rar& B \dar{\delta} \arrow[bend left = 30]{rdd}{\id} \\
B \arrow[bend right = 30]{drr}{\id} \rar{(\id, 0)}  & \arrow{dr} B \oplus \LL_{B / A} \\
& & B
\end{tikzcd}
\]
in $\CAlg_k^\an$, functorially in $(A \to B) \in \Fun(\Delta^1, \CAlg_k^\an)$. 

\end{lemma}
\begin{proof}
For any morphism $P \to Q$ of polynomial algebras, the diagram
\[
\begin{tikzcd}
P \dar \rar& Q \dar{\delta} \arrow[bend left=30]{rdd}{\id} \\
Q \arrow[bend right=30]{drr}{\id} \rar{(\id, 0)} & \arrow{dr} Q \oplus \Omega_{Q / P} \\
& & Q
\end{tikzcd}
\]
commutes, hence the diagram in the lemma is obtained by extending by sifted colimits.
\end{proof}

\begin{lemma}\label{lem_dr_diff}
Let $k$ be a ring, and $(A \to B) \in \CAlg_k^\an$. The composition
\[
B \xrightarrow{\delta} B \oplus \LL_{B / A} \xrightarrow{\pi} \LL_{B / A} 
\]
in $\derD(k)$ agrees with the map $B \to \LL_{B / A}$ coming from the fiber sequence
\[
\LL_{B / A}[-1] \to \dR_{B / A} / \Fil^2 \to B
\]
\end{lemma}
\begin{proof}
Since both constructions of the map commute with sifted colimits, it suffices to show this in the case that $B$ is a polynomial $A$-algebra, in which case this follows from the construction of the boundary map in the long exact sequence.
\end{proof}
\begin{prop}\label{prop_pb_ani_ring}
Let $k$ be a ring, and let $A \to B$ be a surjective ring map with kernel $I$ such that $I^2 = 0$. There exists a pullback diagram
\[
\begin{tikzcd}
A \dar \rar& B \dar{\eta} \\
B \rar{(\id, 0)} & B \oplus I[1]
\end{tikzcd}
\]
in $\CAlg_k^\an$. Moreover, the composition
\[
B \xrightarrow{\eta}  B \oplus I[1] \to I[1]
\]
agrees with the composition
\[
B \to \LL_{B / A} \xrightarrow{\kappa_{B / A / k}} I[1]
\]
in $\derD(k)$ (see Definition \ref{cons_ks}).
\end{prop}
\begin{proof}
Shifting the Kodaira--Spencer map, we obtain a map
\[
\kappa(A \to B) \colon \LL_{B / A} \to I[1]
\]
in $\derD(B)_{\geq 0}$. We thus get a map
\[
B \oplus \LL_{B / A} \to B \oplus I[1]
\]
in $\CAlg_k^\an$. Thus by Lemma \ref{lem_com_square_der} we get the desired commutative diagram
\begin{equation}\label{show_pullback}
\begin{tikzcd}
A \dar \rar& B \dar{\eta} \\
B \rar{(\id, 0)} & B \oplus I[1]
\end{tikzcd}
\end{equation}
Remains to show this is a pullback square in $\CAlg_k^\an$. Let $A' = B \times_{B \oplus I[1]} B$ be the actual pullback, so that we have a commutative diagram
\[
\begin{tikzcd}
A\arrow{dr} \arrow[bend left = 30]{rrd} \arrow[bend right = 30]{rdd} \\
& A' \dar \rar& B \dar{\eta} \\
& B \rar{(\id, 0)} & B \oplus I[1]
\end{tikzcd}
\]
We wish to show the map $A \to A'$ is an isomorphism. The commutative square
\[
\begin{tikzcd}
\dar A \rar& B \dar \\
A' \rar& B
\end{tikzcd}
\]
induces a morphism of fiber sequences
\[
\begin{tikzcd}
\dar A \rar& B\dar  \rar& I[1] \dar {\epsilon}\\
A' \rar& B \rar{\pi_I \circ \eta} & I[1]
\end{tikzcd}
\]
It suffices to verify that $\epsilon$ is an isomorphism. But since $A' = \fib(B \to I[1])$ and $\dR_{B /A} = \fib(B \to \LL_{B / A})$, we may factor as
\[
\begin{tikzcd}
\dar A \rar& B\dar  \rar& I[1] \dar \arrow[bend left=30]{dd}{\epsilon}\\
\dar \dR_{B / A} / \Fil^2 \rar& B  \rar{\mathrm{d}} \dar & \dar{\kappa} \LL_{B / A} \\
A' \rar& B \rar{\pi_I \circ \eta} & I[1]
\end{tikzcd}
\]
By Lemma \ref{c_k_compare}, the map $\epsilon$ is induced by the projection $I \to I / I^2$, hence it is an isomorphism since we assumed $I^2 = 0$. 
\end{proof}

\subsection{Modules over split square zero extensions}
\label{sec_mod_sq_zero}
In this section, we venture slightly into the world of derived algebraic geometry. The goal is to eventually construct an obstruction class to deforming complexes along a square zero extension in Section \ref{sec_ob_complex}. We shall use the notion of a \emph{spectral scheme} as in Lurie \cite[Definition 1.1.2.8]{sag}. 

Although we could theoretically use the notion of a derived scheme for all of the constructions we need, the main reason for choosing to work with spectral schemes is that the theory is substantially better developed, so that we can bootstrap the results we need from \cite{sag}.

For any spectral scheme $X$, we shall denote by $\derD(X)$ the category of quasi-coherent sheaves on $X$ (see \cite[Definition 2.2.2.1]{sag}). 

Let $X$ be a scheme, and let $M \in \derD(X)_{\geq 0}$. Let $X^M$ be the spectral scheme $(X, \catO_X \oplus M)$. The goal of this section is to give an explicit description of $ \derD(X^M)_{\geq 0}$. In fact, we will construct another $\infty$-category $\mathcal{D}_{X, M}$ in terms of $\derD(X)_{\geq 0}$ and $M$, and show that it is equivalent to $\derD(X^M)_{\geq 0}$. 

To this end, let $\pi \colon X^M \to X$ be the morphism of spectral schemes induced by the morphism of sheaves of $\EE_\infty$-rings informally given by $(\id, 0) \colon \catO_X \to \catO_X \oplus M$, and let $\iota \colon X \to X^M$ be the morphism of spectrally ringed spaces induced by the projection $\catO_X \oplus M \to \catO_X$. We then have a fiber sequence
\[
\iota_* M \to  \catO_X \oplus M \to \iota_* \catO_X
\]
in $\derD(X^M)_{\geq 0}$, inducing a map 
\[
\alpha \colon \iota_*\catO_X \to \iota_* M[1]
\]
in $\derD(X^M)_{\geq 0}$. Note that $\pi_* \alpha \simeq 0$  in $\derD(X)_{\geq 0}$, since $\catO_X \oplus M$ is split as an $\catO_X$-module (but not as an $(\catO_X \oplus M)$-module). 

We define the $\infty$-category 
\[
\mathcal{D}_{X,M} := \Fun(\Delta^1, \derD(X)_{\geq 0}) \underset {\Fun(\{0, 1\}, \derD(X)_{\geq 0})} \times \derD(X)_{\geq 0}
\]
where the functor 
\[
\Fun(\Delta^1, \derD(X)_{\geq 0}) \to \Fun(\{0, 1\}, \derD(X)_{\geq 0})
\]
is induced by the inclusion of simplicial sets $\{0, 1\} \to \Delta^1$, and the functor 
\[
\derD(X)_{\geq 0} \to \Fun(\{0, 1\}, \derD(X)_{\geq 0}) \cong \derD(X)_{\geq 0} \times \derD(X)_{\geq 0}
\]
is given by $\(\id, (-) \otimes_{\catO_X} M[1]\)$. 

By \cite[Corollary 2.3.2.5, Corollary 2.4.6.5]{htt} the restriction map 
\[
\Fun(\Delta^1, \derD(X)_{\geq 0}) \to \Fun(\{0, 1\}, \derD(X)_{\geq 0})
\]
is a categorical fibration. It follows by \cite[Remark A.2.4.5]{htt} that the homotopy fiber product defining $\mathcal{D}_{X, M}$ can be computed as the fiber product of simplicial sets. In particular, an object in $\mathcal{D}_{X, M}$ can be described by an object $\mathcal{F} \in \derD(X)_{\geq 0}$ and a morphism $\eta \colon \mathcal{F} \to \mathcal{F} \otimes_{\catO_X} M[1]$ in $\derD(X)_{\geq 0}$.

We now construct an equivalence of categories $\derD(X^M)_{\geq 0} \xrightarrow{\sim} \mathcal{D}_{X, M}$. Denote with $\varphi$ the composition
\[
\derD(X^M)_{\geq 0} \xrightarrow{(-) \otimes \alpha} \Fun(\Delta^1, \derD(X^M)_{\geq 0}) \xrightarrow{\pi_*} \Fun(\Delta^1, \derD(X)_{\geq 0})
\]
By the projection formula  (\cite[Remark 3.4.2.6]{sag}) one may identify 
\begin{align*}
\pi_*(\mathcal{F} \otimes_{} \iota_* \catO_X) &\cong \iota^* \mathcal{F} \\
\pi_*(\mathcal{F} \otimes_{} \iota_* M[1]) &\cong \iota^* \mathcal{F} \otimes_{\catO_X} M[1]
\end{align*}
functorially in $\mathcal{F} \in \derD(X^M)_{\geq 0}$. One thus has a commutative diagram
\[
\begin{tikzcd}[column sep = 100pt]
\dar{\iota^*} \derD(X^M)_{\geq 0} \rar{\varphi}& \Fun(\Delta^1, \derD(X)_{\geq 0}) \dar \\
\derD(X)_{\geq 0} \rar{(\id, (-) \otimes_{\catO_X} M[1])} & \Fun(\{0, 1\}, \derD(X)_{\geq 0})
\end{tikzcd}
\]
inducing a functor $\Phi \colon \derD(X^M)_{\geq 0} \to \mathcal{D}_{X, M}$. 
\begin{lemma}\label{lem_cat_square_zero}
The functor $\Phi \colon \derD(X^M)_{\geq 0} \to \mathcal{D}_{X, M}$ is an equivalence of categories. 
\end{lemma}
\begin{proof}
First note that $\varphi$ and $\iota^*$ commute with all colimits, hence $\Phi$ commutes with all colimits by \cite[Proposition 5.5.3.12]{htt}. 

We then show that $\Phi$ is fully faithful. Observe that for any two given objects $(\mathcal{F}, \eta_{\mathcal{F}}), (\mathcal{G}, \eta_{\mathcal{G}}) \in \mathcal{D}_{X, M}$, one has
\[
\Map_{\mathcal{D}_{X, M}}((\mathcal{F}, \eta_{\mathcal{F}}), (\mathcal{G}, \eta_{\mathcal{G}})) = \tau_{\geq 0}\fib(\derR\Hom_{X}(\mathcal{F}, \mathcal{G}) \to \derR\Hom_X(\mathcal{F}, \mathcal{G} \otimes_{\catO_X} M[1]))
\]
where the map 
\[
\derR\Hom_{X}(\mathcal{F}, \mathcal{G}) \to \derR\Hom_X(\mathcal{F}, \mathcal{G} \otimes_{\catO_X} M[1])
\]
is given by $f \mapsto (\eta_{\mathcal{G}} \circ f - (f \otimes {\id_{M[1]}}) \circ \eta_{\mathcal{F}})$. 

Now let $\mathcal{F}', \mathcal{G}' \in \derD(X^M)_{\geq 0}$, and let $(\mathcal{F}, \eta_\mathcal{F}) = \Phi(\mathcal{F}')$, $(\mathcal{G}, \eta_{\mathcal{G}}) = \Phi(\mathcal{G}')$. It suffices to show that the natural map
\[
\tau_{\geq 0} \derR\Hom_{X^M}(\mathcal{F}', \mathcal{G}') \to \tau_{\geq 0}\fib(\derR\Hom_{X}(\mathcal{F}, \mathcal{G}) \to \derR\Hom_X(\mathcal{F}, \mathcal{G} \otimes_{\catO_X} M[1]))
\]
is a weak equivalence. Since $\Phi$ commutes with colimits we may reduce to the case $\mathcal{F}' = \catO_X \oplus M$. Then $\eta_F = 0$, so it suffices to show
\[
\pi_* \mathcal{G}' \to \fib(\mathcal{G} \xrightarrow{\eta_\mathcal{G}} \mathcal{G} \otimes_{\catO_X} M[1])
\]
is an equivalence in $\derD(X)$, which is immediate by definition $\eta_G$ (since $\pi_*$ is exact).

Remains to show that $\Phi$ is essentially surjective. First note that by \cite[Proposition 5.5.3.6, Proposition 5.5.3.12]{htt} one may show $\mathcal{D}_{X, M}$ is presentable. Since $\Phi$ commutes with colimits it follows by the adjoint functor theorem (\cite[Corollary 5.5.2.9]{htt}) that $\Phi$ admits a right adjoint $\Psi$. Thus, to show that $\Phi$ is essentially surjective, it suffices to show that the canonical map $\Phi \circ \Psi(A) \to A$ is an equivalence for all $A \in \mathcal{D}_{X, M}$. Let $(P, p) = \fib(\Psi \circ \Phi(A) \to A)$ in $\mathcal{D}_{X, M}$, it suffices to show that $P \cong 0$. 

Note that $\Psi$ commutes with limits, so $\Psi(P) \cong 0$ (as $\Psi \circ \Phi \circ \Psi(A) \to \Psi(A)$ is an equivalence by general nonsense), and thus
\[
\Map_{\mathcal{D}_{X, M}}(\Phi(\catO_{X^M}), P) = \Map_{\QCoh(X^M)_{\geq 0}}(\catO_{X, M}, \Psi(P)) \cong 0
\]
But clearly 
\[
\pi_i(\Map_{\mathcal{D}_{X, M}}(\Phi(\catO_{X^M}), P)) = \Ext^{-i}_X(\catO_X, \fib(P \xrightarrow{p} P \otimes_{\catO_X} M[1]))
\]
An induction argument (using that $P$ is connective) shows that $\coh^{-i} \derR \Gamma(X, P) = 0$ for all $i \geq 0$, hence $P \cong 0$ as required.
\end{proof}

\begin{corollary}\label{lem_sheafs_triv_sqzero}
Let $X$ be a scheme, let $M \in \derD(X)_{\geq 0}$, and let $\mathcal{E} \in \derD(X)_{\geq 0}$. Let 
\[
\alpha \colon \iota_* \catO_X \to \iota_*  M[1]
\]
be the unique map in $\derD(X^M)$ whose fiber is isomorphic to $\catO_X \oplus M$. 

Then the functor $Q \mapsto \pi_*(Q \otimes \alpha)$ induces a bijection
\begin{equation} \label{bca}
\Phi \colon \(\derD(X^M)_{\geq 0} \underset {\derD(X)_{\geq 0}} \times \{\mathcal{E}\}\)^{\simeq} \simeq \Map_{\derD(X)}(\mathcal{E}, \mathcal{E} \otimes_{\catO_X} M[1])
\end{equation}
of pointed spaces.
\end{corollary}
\begin{lemma}\label{bcf}
Let $R$ be a discrete ring, and let $M$ be a discrete $R$-module. Then the composition
\begin{align}
\Aut_{\derD(R \oplus M)}(R \oplus M) \times_{\Aut_{\derD(R)}(R)}  \{\id_R\} &\xrightarrow{\Phi} \pi_1 \Map_{\derD(R)}(R, M[1])  \nonumber \\ &\xrightarrow{\sim} \Hom_R(R, M)\label{bcd}
\end{align}
is given by $\varphi \mapsto \pi_M \circ \varphi \circ \iota_R$, where 
\begin{align*}
\iota_R &\colon R \to R \oplus M \\
\pi_M &\colon R \oplus M \to M
\end{align*}
are the inclusion and projection maps.
\end{lemma}
\begin{proof}
Write 
\[
A_R := \Aut_{\derD(R \oplus M)}(R \oplus M) \underset{\Aut_{\derD(R)}(R)} \times \{\id_R\}
\]
Note that $\pi_* \colon \derD(R \oplus M) \to \derD(R)$ is just forgetful functor induced by the zero section $R \to R \oplus M$. Recall that we have a fiber functor \cite[Definition 1.1.1.6, Remark 1.1.1.7]{ha} 
\[
\fib \colon \Fun(\Delta^1, \derD(R)) \to \derD(R)
\]
Hence by definition of $\Phi$ we see that for any $\varphi \in A_R$ one has
\[
\fib(\Phi(\varphi)) = \fib(\pi_*(\varphi \otimes \alpha))
\]
as maps $R \oplus M \to R \oplus M$. 

Thus by Lemma \ref{bce} the map (\ref{bcd}) is given by 
\[
\varphi \mapsto \pi_M \circ \fib(\pi_*(\varphi \otimes \alpha)) \circ \iota_R
\]
Since $\pi_*$ and $(-) \otimes \alpha$ are exact functors, we see that the map (\ref{bcd}) can also be described as 
\[
\varphi \mapsto \pi_M \circ \pi_* (\varphi \otimes \fib(\alpha)) \circ \iota_R
\]
But by definition, one has $\fib(\alpha) = R \oplus M$, so the functor 
\[
(-) \otimes \fib(\alpha) \colon \derD(R \oplus M) \to \derD(R \oplus M)
\]
is equivalent to the identity. The result follows.

\end{proof}
\subsection{Atiyah classes and Chern characters}
\label{sec_atiyah}
Let $k$ be a ring, and let $X$ be a scheme over $k$ which has the resolution property \stacksref{0F85}. The goal of this section is to construct, for any $\mathcal{E} \in \Perf(X)$, an Atiyah class
\[
\At_{X / k}(\mathcal{E}) \in \Ext^1_X(\mathcal{E}, \mathcal{E} \otimes \LL_{X / k}) 
\]
Moreover, we will introduce a notion of a trace map
\[
\Ext^p_X(\mathcal{E}, \mathcal{E} \otimes \LL^p_{X / k}) \xrightarrow{\tr} \coh^p(\LL^p_{X / k})
\]
and show that if $p!$ is invertible $k$ then
\begin{equation} \label{eq_trace_exp_at}
\tr\(\frac{(\At_{X / k}(\mathcal{E}))^p}{p!}\) \in \coh^p(\LL^p_{X / k})
\end{equation}
agrees with the image of $\ch_p(P)$ (see Definition \ref{def_chern_character}) under the natural map $\Fil^p \widehat{\dR}_{X / k} \to \LL^p_{X / k}[-p]$. The expression (\ref{eq_trace_exp_at}) was taken as the definition for Chern classes in the affine case by Illusie in \cite{ill71}. In the classical (smooth) case this result is well known, see for example \cite[Section 10.1]{huyleh}. However we are not aware of the result for derived de Rham cohomology appearing anywhere in the literature.

Before giving the definition of the Atiyah class, observe that for any scheme $X$ over a ring $k$ there exists a spectral scheme $X^\LL := (X, \catO_X \oplus \LL_{X / k})$, together with maps of spectral schemes
\begin{align*}
\iota &\colon (X, \catO_X) \to (X, \catO_X \oplus \LL_{X / k}) \\
\pi_0, \pi_\delta &\colon (X, \catO_X \oplus \LL_{X / k}) \to (X, \catO_X)
\end{align*}
which are given by the identity maps on topological spaces, and where the morphisms of sheaves of $\EE_\infty$-rings are as follows: 
\begin{align*}
\pi_0^\# &:= (\id, 0) \colon \catO_X \to \catO_X \oplus \LL_{X / k} \\
\pi_\delta^\# &:= \delta \colon \catO_X \to \catO_X \oplus \LL_{X / k} &(\text{see Definition \ref{def_univ_der}}) \\
\iota^\# &:= \pi_{\catO_X} \colon  \catO_X \oplus \LL_{X / k} \to \catO_X &(\text{projection to $\catO_X$})
\end{align*}
Thus clearly $\pi_0 \circ \iota = \pi_\delta \circ \iota = \id_X$. Moreover this construction is functorial in $X$, that is a map $f \colon X \to Y$ induces a map $f \colon X^\LL \to Y^\LL$ for which the natural squares with $\pi_0$, $\pi_\delta$ and $\iota$ commute.

We now wish to define the Atiyah class. Inspired by the philosophy of \cite{huto}, we first define the universal Atiyah class.
\begin{definition}\label{aea}
Let $k$ be a ring, and let $X$ be a scheme over $k$. We define the \emph{universal Atiyah class} to be the unique element
\[
\alpha_X \in \Ext^1_{X^\LL}(\iota_* \catO_{X / k}, \iota_* \LL_{X /k})
\]
as the boundary map induced by the fiber sequence
\[
\iota_* \LL_{X / k} \to \catO_{X} \oplus \LL_{X / k} \xrightarrow{\eta_\iota}  \iota_* \catO_X 
\]
in $\derD(X^\LL)$ (here $\eta_\iota$ is the unit of the adjunction $\iota^* \dashv \iota_*$).
\end{definition}

\begin{definition}[Atiyah class]\label{aeb}
Let $k$ be a ring, and let $X$ be a scheme over $k$. Let $\mathcal{E} \in \Perf(X)$. We define the \emph{Atiyah class} of $\mathcal{E}$ 
\[
\At_{X / k}(\mathcal{E}) \in \Ext^1_X(\mathcal{E}, \mathcal{E} \otimes \LL_{X / k})
\]
as
\[
\At_{X / k}(\mathcal{E}) := (\pi_0)_* (\pi_\delta^*(\mathcal{E}) \otimes \alpha_X)
\]
We denote with 
\[
\At_{X / k}(\mathcal{E})^p \in \Ext^p_X(\mathcal{E}, \mathcal{E} \otimes \LL^p_{X / k})
\]
the composition
\begin{align*}
\mathcal{E} &\xrightarrow{\At_{X / k}(\mathcal{E})} \mathcal{E} \otimes \LL_{X / k}[1] \\\
&\xrightarrow{\At_{X / k}(\mathcal{E}) \otimes \LL_{X / k}[1]}  \mathcal{E} \otimes \LL_{X / k}^{\otimes 2}[2] \\
&\qquad \qquad \vdots \\
& \xrightarrow{\At_{X / k}(\mathcal{E}) \otimes \LL_{X / k}^{\otimes(p - 1)}[p - 1]}  \mathcal{E} \otimes \LL_{X / k}^{\otimes p}[p] \\
&\to \mathcal{E} \otimes \LL^p_{X / k}[p]
\end{align*}
which we call the \emph{$p$-th power of the Atiyah class}.
\end{definition}
Note that the first power $\At_{X / k}(\mathcal{E})^1$ is simply $\At_{X / k}(\mathcal{E})$. 
\begin{lemma}\label{lem_atiyah_two}
Let $k$ be a ring, and let $f \colon X \to Y$ be a morphism of schemes over $k$. Let $n \in \ZZ$. Denote with 
\[
\phi \colon f^* \LL_{Y  / k} \to \LL_{X / k}
\]
the canonical map induced by $f$. 

Then the equality 
\[
\At_{X / k}(f^* \mathcal{E}) = (\id_{\mathcal{E}} \otimes \ \phi ) \circ f^* \At_Y(\mathcal{E})
\]
holds in $\Ext^1_Y(\mathcal{E}, \mathcal{E} \otimes \LL_{Y / k})$ for all $\mathcal{E} \in \Perf(Y)$.
\end{lemma}
\begin{proof}
Denote both maps $X \to X^\LL$ and $Y \to Y^\LL$ by $\iota$, similarly for $\pi_\delta, \pi_0$. Tensoring the commutative diagram
\[
\begin{tikzcd}
\arrow[swap]{d}{\iota_* \phi \circ (\text{base change})} f^* \iota_* \LL_{Y / k} \rar{f^*\alpha_Y} & f^* \iota_* \catO_Y[1] \arrow{d}{\text{base change}} \\
\iota_* \LL_{X / k} \rar{\alpha_X} & \iota_* \catO_X[1]
\end{tikzcd}
\]
in $\derD(X^\LL)$ with $f^* \pi_\delta^* \mathcal{E}$, we obtain a commutative diagram
\[
\begin{tikzcd}[column sep = huge]
\dar f^* \iota_* (\mathcal{E} \otimes \LL_{Y / k}) \rar{f^*(\pi_\delta^* \mathcal{E} \otimes \alpha_Y)} & f^* \iota_* (\mathcal{E} \otimes \catO_Y[1]) \dar  \\
\iota_*(f^* \mathcal{E} \otimes \LL_{X /k}) \rar{ (\pi_\delta^*f^* \mathcal{E}) \otimes \alpha_X} & \iota_*(f^* \mathcal{E} \otimes \catO_X[1])
\end{tikzcd}
\]
in $\derD(X^\LL)$. Applying $(\pi_0)_*$ we obtain the lower square in the following commutative diagram
\[
\begin{tikzcd}[column sep = 3.5cm]
\dar f^*(\mathcal{E} \otimes \LL_{Y / k}) \rar{f^* \At_Y(\mathcal{E})} & f^* (\mathcal{E} \otimes \catO_Y[1]) \dar \\
\dar (\pi_0)_* f^* \iota_* (\mathcal{E} \otimes \LL_{Y / k}) \rar{(\pi_0)_*f^*(\pi_\delta^* \mathcal{E} \otimes \alpha_Y)} & (\pi_0)_*  f^* \iota_* (\mathcal{E} \otimes \catO_Y[1]) \dar  \\
f^* \mathcal{E} \otimes \LL_{X /k} \rar{ \At_{X / k}(f^* \mathcal{E})} & f^* \mathcal{E} \otimes \catO_X[1]
\end{tikzcd}
\]
where the upper square is obtained via the base change map $f^* (\pi_0)_* \to (\pi_0)_* f^*$. The composition down--down--right is equal to $\At_{X / k}(f^* \mathcal{E})$ and the composition right--down--down is equal to the composition $ (\id_{\mathcal{E}} \otimes \ \phi ) \circ f^* \At_Y(\mathcal{E})$, so the result follows. 
\end{proof}
\begin{definition}[Dualizable object]\label{haa}
Let $\mathcal{C}^\otimes$ be a symmetric monoidal $\infty$-category with unit object $\catO$, and let $\mathcal{E} \in \mathcal{C}$. We say that $\mathcal{E}$ is \emph{dualizable} if there exists an object $\mathcal{E}^\vee \in \mathcal{C}$ and maps
\begin{align*}
\ev \colon \mathcal{E} \otimes \mathcal{E}^\vee  \to \catO \\
\coev \colon \catO \to \mathcal{E}^\vee \otimes \mathcal{E}
\end{align*}
such that the compositions
\begin{align*}
\mathcal{E}^\vee \simeq \catO \otimes \mathcal{E}^\vee  \xrightarrow{\coev \otimes \id} \mathcal{E}^\vee \otimes \mathcal{E} \otimes \mathcal{E}^\vee \xrightarrow{\id \otimes \ev} \mathcal{E}^\vee \otimes \catO \simeq \mathcal{E}^\vee  \\
\mathcal{E} \simeq \mathcal{E} \otimes \catO \xrightarrow{\id \otimes \coev} \mathcal{E} \otimes \mathcal{E}^\vee \otimes \mathcal{E} \xrightarrow{\ev \otimes \id} \catO \otimes \mathcal{E}\simeq \mathcal{E} 
\end{align*}
are homotopic to the identity. 
\end{definition}
It is well known that for any scheme $X$, an object $\mathcal{E} \in \Perf(X)$ is dualizable. For any dualizable object, we may define a trace map. 

\begin{definition}[Trace map]\label{def_trace_map}
Let $\mathcal{C}^\otimes$ be a symmetric monoidal $\infty$-category, and let $\mathcal{E} \in \mathcal{C}$ be a dualizable object. For two objects $M, N \in \mathcal{C}$ and a map 
\[
\alpha \colon \mathcal{E} \otimes M \to \mathcal{E} \otimes N
\]
in $\mathcal{C}$, we define 
\[
\tr_\mathcal{E}(\alpha) \colon M \to N
\]
as the composition
\[
M \xrightarrow{\mathrm{coev} \otimes M} \mathcal{E}^\vee \otimes \mathcal{E} \otimes M \xrightarrow{{\mathcal{E}^\vee} \otimes \alpha} \mathcal{E}^\vee \otimes \mathcal{E} \otimes N \xrightarrow{\sigma_{12}} \mathcal{E} \otimes \mathcal{E}^\vee \otimes N \xrightarrow{\mathrm{ev} \otimes N} N
\]
in $\mathcal{C}$. 
\end{definition}
\begin{lemma}\label{lem_atiyah_one}
Let $k$ be a ring, and let $X$ be a scheme over $k$. Fix $n \in \ZZ$, and let
\[
\mathcal{E} \to \mathcal{F} \to \mathcal{G} \xrightarrow{+1} 
\]
be an exact triangle of objects in $\Perf(X)$. Then
\[
\tr\(\At_{X / k}(\mathcal{F})^p\) = \tr\(\At_{X / k}(\mathcal{E})^p\) + \tr\(\At_{X / k}(\mathcal{G})^p\)
\]
for all $p \geq 0$. 
\end{lemma}
\begin{proof}
We have a commutative diagram
\[
\begin{tikzcd}
 \dar{\At_{X / k}(\mathcal{E})^p} \mathcal{E} \rar& \dar{\At_{X / k}(\mathcal{F})^p}\mathcal{F} \rar& \dar{\At_{X / k}(\mathcal{G})^p}\mathcal{G} \rar{+1} &  \ \\
\mathcal{E} \otimes \LL^p_{X / k}  \rar& \mathcal{F} \otimes \LL^p_{X / k} \rar&  \mathcal{G} \otimes \LL^p_{X / k} \rar{+1}  & \
\end{tikzcd}
\]
in $\derD(X)$, in which the rows are exact triangles. Thus this follows from a well-known result on traces, see Proposition \ref{prop_add_traces}. 
\end{proof}
Next, observe that any element $f \in \Ext^1_X(\mathcal{L}, \mathcal{L} \otimes \LL_{X /k })$ can be considered as a map
\[
f \colon \mathcal{L} \to  \mathcal{L} \otimes \LL_{X /k}[1]
\]
in $\derD(X)$. In particular, we may consider $\tr_\mathcal{L}(f)$, which is a map 
\[
\catO_X \to \LL_{X / k}[1]
\]
We thus get a natural trace map $\tr_{\mathcal{L}} \colon \Ext^1_X(\mathcal{L}, \mathcal{L} \otimes \LL_{X /k }) \to \coh^1(\LL_{X / k})$. Our goal now is to show this map sends the Atiyah class to the first Chern class (see Corollary \ref{corol_atiyah_threeb}). 
\begin{construction}
Let $k$ be a ring, and let $X$ be a scheme over $k$. Consider a line bundle $\mathcal{L} \in \Pic(X)$. Then we have canonical equivalences
\[
\iota^* (\pi_\delta^*\mathcal{L} \otimes \pi_0^* \mathcal{L}^\vee) \simeq \mathcal{L} \otimes \mathcal{L}^\vee \simeq \catO_X
\]
We thus obtain a map
\begin{align*}
v \colon \Pic(X) &\to \Pic(X^\LL) \times_{\Pic(X)} \{\catO_X\} \\
\mathcal{L} &\mapsto \pi_\delta^*\mathcal{L} \otimes \pi_0^* \mathcal{L}^\vee
\end{align*}
in $\derD(\ZZ)_{\geq 0}$. 
\end{construction}

\begin{lemma}\label{bcc}
The composition 
\[
\Pic(X) \xrightarrow{v} \Pic(X^\LL) \times_{\Pic(X)} \{\catO_X\} \xrightarrow{\Phi} \coh^1(X, \LL_{X / k})
\]
of maps of abelian groups sends $\mathcal{L} \in \Pic(X)$ to $\tr_{\mathcal{L}}(\At_{X / k}(\mathcal{L}))$, where $\Phi$ is the map (\ref{bca}). 
\end{lemma}
\begin{proof}
By definition, for any $\mathcal{L} \in \Pic(X)$, we have
\begin{align*}
(\Phi \circ v)(\mathcal{L}) &= \pi_{0, *}(\pi_\delta^* \mathcal{L} \otimes \pi_0^* \mathcal{L}^\vee \otimes \alpha_X) \\
&= \pi_{0, *}(\pi_\delta^* \mathcal{L} \otimes \alpha_X) \otimes \mathcal{L}^\vee \\
&= \At_{X / k}(\mathcal{L}) \otimes \mathcal{L}^\vee \\
&= \tr_{\mathcal{L}}(\At_{X / k}(\mathcal{L}))
\end{align*}
Here the first equality follows by definition of $\Phi$, the second is the projection formula, the third equality follows by definition of the Atiyah class (see Definition \ref{aeb}), and the last equality follows by observing that the trace map
\[
\tr_{\mathcal{L}} \colon \Ext^1_X(\mathcal{L}, \mathcal{L} \otimes \LL_{X / k}) \to \coh^1(X, \LL_{X / k})
\]
is simply given by $(-) \otimes \mathcal{L}^\vee$ (since $\mathcal{L}$ is a line bundle).  
\end{proof}
\begin{lemma}\label{bcb}
Let $k$ be a ring, and let $X$ be a scheme over $k$. Then there exists a commutative diagram
\[
\begin{tikzcd}
\Pic(X) \dar{v} & \tau_{\geq 0} \derR \Gamma(X, \GG_m[1]) \arrow{l}{\sim} \arrow{dd}{\mathrm{d} \log} \\
\Pic(X^\LL) \times_{\Pic(X)} \{\catO_X\} \dar{\Phi}[swap]{\sim} & \\
  \Map_{\derD(X)}(\catO_X, \LL_{X / k}[1]) \rar[swap]{\sim} & \tau_{\geq 0} \derR \Gamma(X, \LL_{X /k}[1])
\end{tikzcd}
\]
in $\derD(\ZZ)_{\geq 0}$. 
\end{lemma}
\begin{proof}
Since all terms appearing in the diagram are fppf-sheaves, it suffices to construct a diagram
\[
\begin{tikzcd}
\Pic(R) \dar{v} &  \tau_{\geq 0} \derR \Gamma(\Spec(R), \GG_m[1]) \arrow{l}{\sim} \arrow{dd}{\mathrm{d} \log} \\
\Pic(R \oplus \LL_{R / k}) \times_{\Pic(R)} \{R\} \dar{\Phi}[swap]{\sim} &  \\
 \Map_{\derD(R)}(R, \LL_{R / k}[1]) \rar[swap]{\sim} & \tau_{\geq 0} \derR \Hom_{\derD(R)}(R, \LL_{R /k}[1])
\end{tikzcd}
\]
functorial in $R \in \Alg_k$. Recall that we denote with $\GG_m(-)$ the functor
\begin{align*}
\Alg_k &\to \derD(\ZZ)^\heart \\
R &\mapsto R^\times
\end{align*}
Thus if we denote with $\B \GG_m(R)$ the associated functor taking values in $1$-groupoids (with a $\ZZ$--action), we see that the functor
\begin{align*}
\Alg_k &\to \derD(\ZZ)_{\geq 0} \\
R &\mapsto \derR \Gamma(\Spec(R), \GG_m[1]) 
\end{align*}
can be identified with the sheafification of the functor $R \mapsto \B \GG_m(R)$. By the universal property of sheafification it thus suffices to construct a commutative diagram
\[
\begin{tikzcd}
\Pic(R) \dar{v} & \B \GG_m(R) \arrow{l} \arrow{dd}{\mathrm{d} \log} \\
\Pic(R \oplus \LL_{R / k}) \times_{\Pic(R)} \{R\} \dar{\Phi}[swap]{\sim} & \\
\Map_{\derD(R)}(R, \LL_{R / k}[1]) \rar[swap]{\sim} & \tau_{\geq 0} \derR \Hom(R, \LL_{R /k}[1])
\end{tikzcd}
\]
in $\derD(\ZZ)_{\geq 0}$, functorial in $R \in \Alg_k$. By Lemma \ref{gm_extend} and the universal property of left Kan extension \cite[Proposition 4.3.2.17]{htt} we may assume $R$ is smooth over $k$, so that $\LL_{R / k} \simeq \Omega_{R / k}$ is an abelian group.

Since the unique point in $\B \GG_m(R)$ maps to zero under both compositions, it suffices to show the diagram of abelian groups
\[
\begin{tikzcd}
\Aut_{\derD(R)}(R) \dar{v} &  \GG_m(R) \arrow{l} \arrow{dd}{\mathrm{d} \log} \\
\Aut_{\derD(R \oplus \Omega_{R / k})}(R \oplus \Omega_{R / k}) \times_{\Aut_{\derD(R)}(R)} \{\id_R\} \dar{\Phi}[swap]{\sim} \\
\pi_1 \Map_{\derD(R)}(R, \Omega_{R / k}[1])  \rar[swap]{\sim} & \Hom(R, \Omega_{R /k})
\end{tikzcd}
\]
 commutes. By definition of $v$, the composition left--down sends $r \in R^\times$ to the automorphism given by multiplication with the element
\[
(r,  \mathrm{d} r) \cdot (r^{-1}, 0) = (1, \frac{1}{r} \mathrm{d} r)
\]
It thus follows from Lemma \ref{bcf} that the composition left--down--down sends an element $r \in R^\times$ to the map $s \mapsto s \cdot \frac{1}{r} \mathrm{d}r$, which completes the proof.
\end{proof}

\begin{corollary}\label{corol_atiyah_threeb}
Let $k$ be a ring, and let $X$ be a scheme over $k$. Let $\mathcal{L}$ be a line bundle on $X$. Then the trace map
\[
\tr_{\mathcal{L}} \colon \Ext^1_X(\mathcal{L}, \mathcal{L} \otimes \LL_{X /k }) \to \coh^1(\LL_{X / k})
\]
sends $\At_{X / k}(\mathcal{L})$ to the image of $\csmall^{\widehat{\dR}}_{1}(\mathcal{L})$ under the natural map
\[
\coh^2(\Fil^1 \widehat{\dR}_{X / k}) \to \coh^1(\LL_{X  /k})
\]
\end{corollary}
\begin{proof}
By Lemma \ref{bcb} we obtain a commutative diagram
\[
\begin{tikzcd}
\Pic(X) \dar{v} \rar{\sim} & \coh^1(X, \GG_m)  \dar{\mathrm{d} \log} \\
\Pic(X^\LL) \times_{\Pic(X)} \{\catO_X\}  \rar{\Phi} & \coh^1(X, \LL_{X / k})
\end{tikzcd}
\]
of abelian groups. By definition of $\csmall^{\widehat{\dR}}_{1}$ (see Definition \ref{defc1dr}) the composition right--down sends $\mathcal{L} \mapsto \csmall^{\widehat{\dR}}_{1}(\mathcal{L})$. Hence the result follows from Lemma \ref{bcc}.
\end{proof}
\begin{prop}\label{prop_compare_chern}
Let $i \in \NN$, and let $k$ be a ring such that $i!$ is invertible in $k$. Let $X$ be a quasi-compact and quasi-separated scheme over $k$ which has the resolution property. Then for all $\mathcal{E} \in \Perf(X)$, the equality
\[
\ch_i(\mathcal{E}) = \frac{\tr_{\mathcal{E}}(\At_{X / k}(\mathcal{E})^i)}{i!}
\]
holds in $\coh^{i} \(\LL^i_{X / k}\)$.
\end{prop}
\begin{proof}
Write 
\[
\widetilde \ch_i(\mathcal{E})  := \frac{\tr_{\mathcal{E}}(\At_{X / k}(\mathcal{E})^i)}{i!}
\]
we verify axioms (1) -- (3) of Proposition \ref{prop_chern_unique}, this will imply the result. Clearly (1) is a direct consequence of Lemma \ref{lem_atiyah_one}, and (2) is a direct consequence of Lemma \ref{lem_atiyah_two}. Finally (3) is implied by the combination of Lemma \ref{lem_atiyah_threea} and Corollary \ref{corol_atiyah_threeb} which completes the proof. 
\end{proof}
\subsection{Obstruction classes for complexes}
\label{sec_ob_complex}
In this section we provide a result constructing an obstruction class to deforming complexes, and show that it can be written as the product of the Atiyah class with the Kodaira--Spencer class. 
This result is quite well-known, and already goes back to \cite{ill71}.  Our approach is to bootstrap from \cite[Theorem 16.2.0.1]{sag}. A direct construction of the obstruction class (originally due to Gabber) in a similar setting can be found in \cite{liebols}. For a proof of the same result in a different language using the truncated cotangent complex, see \cite{huto}. 

The precise result we will need is the following.
\begin{prop}\label{prop_obstruction_class}
Let $k$ be a ring, and let $R' \to R$ be a surjective ring map with kernel $I$ such that $I^2 = 0$. Let $X'$ be a smooth and proper scheme over $R'$, and let $X = X' \times_{\Spec(R')} \Spec(R)$ be the base change, and let $\mathcal{I} = I \otimes_{R} \catO_X$. Let $\mathcal{E} \in \Perf(X)$. Then there exists a complex $\mathcal{E}' \in \Perf(X')$ such that $\mathcal{E}' \rvert_X \cong \mathcal{E}$ if and only if the obstruction class
\[
\ob(\mathcal{E}, X, X') := (\mathcal{E} \otimes \kappa_{X / X' / k}) \circ \At_{X / k}(\mathcal{E}) \in \Ext^2_X(\mathcal{E}, \mathcal{E} \otimes_{\catO_X} \mathcal{I})
\]
is equal to zero.
\end{prop}
\begin{proof}
By shifting we may assume $\mathcal{E}$ is connective. Let $\mathcal{X}, \mathcal{X}'$ be the associated spectral schemes (see \cite[Remark 1.1.8.5]{sag}), and let $\mathfrak{X}, \mathfrak{X}'$ be the associated spectral Deligne-Mumford stacks (see \cite[1.6.6, Remark 1.6.6.5]{sag}). 
By Proposition \ref{prop_pb_ani_ring} and \cite[Prop 16.1.3.1]{sag}, we get a pushout diagram
\[
\begin{tikzcd}
\spet(R \oplus I[1]) \dar{\eta} \rar{\eta_0}& \spet(R) \dar \\
\spet(R) \rar& \spet(R')
\end{tikzcd}
\]
in the $\infty$-category $\spdm$ of spectral Deligne-Mumford stacks. Write
\[
\mathfrak{X}^I := \mathfrak{X}'  \times_{\spet(R')} \spet(R \oplus I[1])
\]
so that by \cite[Prop 16.3.1.1]{sag}, we get a pushout diagram
\[
\begin{tikzcd}
\mathfrak{X}^I  \rar{\eta_0} \dar{\eta}&  \mathfrak{X}'  \times_{\spet(R')} \spet(R) \dar \\
 \mathfrak{X}' \times_{\spet(R')} \spet(R) \rar&  \mathfrak{X}'
\end{tikzcd}
\]
in the $\infty$-category $\spdm$ of spectral Deligne-Mumford stacks. By \cite[Corollary 1.6.7.5]{sag}, the fiber product $\mathfrak{X'} \times_{\spet(R')} \spet(R)$ may be computed as the fiber product $\mathcal{X}' \times_{\Spec(R')} \Spec(R)$ in the $\infty$-category $\spsch$ of (connective) spectral schemes. By (the proof of) \cite[Corollary 4.3.1.11]{htt} one may show that this fiber product is the spectrally ringed space
\[
(X', \catO_{X'} \otimes_{R'} R)
\]
where the tensor product is computed in the $\infty$-category $\Shv_{\CAlg(\derD(k))}(X')$. Now since $X'$ is smooth over $R'$, the map $R'  \to \catO_{X'}$ is flat, so we conclude the fiber product $\mathfrak{X}' \times_{\spet(R')} \spet(R)$ is simply $\mathfrak{X}$. We thus get a pushout diagram
\[
\begin{tikzcd}
\mathfrak{X}^I  \rar{\eta_0} \dar{\eta}& \mathfrak{X} \dar \\
\mathfrak{X} \rar& \mathfrak{X}'
\end{tikzcd}
\]
in $\spdm$. By \cite[Theorem 16.2.0.1]{sag}, we get a pullback diagram
\[
\begin{tikzcd}
\QCoh(\mathfrak{X}')^{\mathrm{cn}} \rar \dar & \QCoh(\mathfrak{X})^{\mathrm{cn}} \dar{\eta^*} \\
\QCoh(\mathfrak{X})^{\mathrm{cn}} \rar{\eta_0^*}& \QCoh\(\mathfrak{X}_I  \)^{\mathrm{cn}} 
\end{tikzcd}
\]
of $\infty$-categories. By \cite[Corollary 2.2.6.2]{sag} we may identify $\QCoh(\mathfrak{X})$ with $\derD(X)$ (and similarly for $X'$). We conclude that $\mathcal{E}'$ exists if and only if there exists an equivalence $\eta^* \mathcal{E} \simeq \eta_0^*\mathcal{E}$ in $\QCoh\(\mathfrak{X}^I \)$. By Corollary \ref{lem_sheafs_triv_sqzero}, this is the case if and only if the class
\[
[\eta^*\mathcal{E}] \in \Ext^1_{\catO_X}(\mathcal{E}, \mathcal{E} \oplus \mathcal{I}[1])
\]
is equal to $0$ (the zero object is the class of $[\eta_0^*(\mathcal{E})]$). The result now follows directly by noting that the map $\eta^\#$ factors as
\[
\catO_X \xrightarrow{\delta} \catO_X \oplus \LL_{X / k} \xrightarrow{(\id, \kappa_{X / X'  /k})} \catO_X \oplus (I \otimes_R \catO_X)
\]
(which follows essentially from Corollary \ref{lem_kodaira_base_change}).  
\end{proof}
\subsection{Computing signs for (shifted) permutation actions}
In this section we gather some technical computations in order to streamline the proofs of Section \ref{sec_relate}. 

We start by introducing some notation. For any $n \geq 0$, we denote with $\mathfrak{S}_n$ the symmetric group on $n$ elements. For any symmetric monoidal $\infty$-category $\mathcal{C}$, any object $X \in \mathcal{C}$ and any $\sigma \in \mathfrak{S}_n$, the symmetric monoidal structure provides a map
\[
\sigma(X)\colon X^{\otimes n} \to X^{\otimes n}
\]
If $\mathcal{C}$ is also stable, then $\mathcal{C}$ comes with a shift (suspension) functor $[1] \colon \mathcal{C} \to \mathcal{C}$. These constructions are related in the following way.
\begin{lemma}\label{sign1}-
Let $\mathcal{C}$ be a presentable stable symmetric monoidal $\infty$-category  for which the tensor product preserves finite limits in each variable. For any $\sigma \in \mathfrak{S}_n$ and $X \in \mathcal{C}$, there exists a commutative diagram
\begin{equation}\label{tensor_sign}
\begin{tikzcd}[column sep = huge]
X^{\otimes n}[n] \dar{\sim} \rar{\sigma(X)[n]} & X^{\otimes n}[n] \dar{\sim} \\
(X[1])^{\otimes n} \rar{s} & (X[1])^{\otimes n}
\end{tikzcd}
\end{equation}
in $\mathcal{C}$, where $s = \mathrm{sgn}(\sigma) \cdot \sigma(X[1])$. 
\end{lemma}
\begin{proof}
By decomposing $\sigma$ into cycles of length $2$, we may reduce to the case where $n = 2$ and $\sigma$ is the only nontrivial element of $\mathfrak{S}_2$. Denote with $\mathbb{S}$ the sphere spectrum in the stable $\infty$-category of spectra $\mathrm{Sp}$. By an explicit computation, one may verify that the composition
\[
\mathbb{S}[2] \xrightarrow{\sim} \mathbb{S}[1] \underset{\mathbb{S}}{\otimes} \mathbb{S}[1] \xrightarrow{\sigma(\mathbb{S}[1])} \mathbb{S}[1] \underset{\mathbb{S}}{\otimes} \mathbb{S}[1] \xrightarrow{\sim} \mathbb{S}[2]
\]
is equivalent to $-\id_{\mathbb{S}}[2]$. Let $\catO \in \mathcal{C}$ be the unit for the symmetric monoidal structure, so that we have an essentially unique symmetric monoidal functor $\mathrm{Sp} \to \mathcal{C}$ sending $\mathbb{S} \mapsto \catO$ (see \cite[Corollary 4.8.2.19]{ha}). Then by functoriality we see that the composition
\[
\catO[2] \xrightarrow{\sim} \catO[1] \underset{\catO}{\otimes} \catO[1] \xrightarrow{\sigma(\catO[1])} \catO[1] \underset{\catO}{\otimes} \catO[1] \xrightarrow{\sim} \catO[2]
\]
is equivalent to $-\id_{\catO}[2]$. For arbitrary $X \in \mathcal{C}$, we get a commutative diagram
\[
\begin{tikzcd}[column sep = tiny]
(X \otimes X)[2] \dar \rar{-\id} & (X \otimes X)[2] \dar \rar{\sigma_{12}}  & (X \otimes X)[2] \dar \\
\dar X  \otimes (\catO[1]  \otimes \catO[1])  \otimes X \rar[swap]{\sigma_{23}} & X  \otimes (\catO[1]  \otimes \catO[1]) \otimes X \rar[swap]{\sigma_{14}}  & X  \otimes (\catO[1]  \otimes \catO[1])  \otimes X \dar \\
X[1] \otimes X[1] \arrow{rr}{\sigma_{12}} & & X[1] \otimes X[1]
\end{tikzcd}
\]
where we denote with $\sigma_{ij}$ the morphism given by the symmetric monoidal structure on $\mathcal{C}$ swapping factors $i$ and $j$ in a tensor product. The outer square now gives the desired diagram.
\end{proof}
For any $n \in \NN$, we shall denote by $\sigma^{i, n}$ the unique permutation sending $i \mapsto n$ and keeping all other elements in order. Explicitly, $\sigma^{i, n}$ is given by
\[
\sigma^{i, n}(j) := \begin{cases} 
j & j < i \\
n & j = i\\
j - 1 & j > i
\end{cases}
\]
for $j \in \{1, \dots, n\}$.

 For $\mathcal{C}$ any stable symmetric monoidal $\infty$-category and $X \in \mathcal{C}$ we define
\[
\Sigma^{+, n}_X := \sum_{i = 1}^n \sigma^{i, n}_X \qquad \Sigma^{-, n}_X := \sum_{i = 1}^n (-1)^{n - i} \sigma^{i, n}_X
\]
as maps $X^{\otimes n} \to X^{\otimes n}$. 
\begin{lemma}\label{sign2}
Let $k$ be a ring and let $X$ be a stack over $k$. Write 
\[
\Sigma^{-} := \Sigma^{-,i}_{\LL_{X / k}}
\]
Then there exists a commutative diagram
\[
\begin{tikzcd}
\LL_{X / k}^{\otimes i} \dar{\Sigma^{-}} \rar & \LL_{X / k}^i \dar{\Delta_{i - 1}} \\
\LL_{X / k}^{\otimes i} \rar & \LL_{X / k}^{i - 1} \underset{\catO_X}{\otimes} \LL_{X / k}
\end{tikzcd}
\]
in $\derD(X)$, where $\Delta_{i - 1}$ is the map from Definition \ref{def_delta}. 
\end{lemma}
\begin{proof}
Unwinding the definitions, we may reduce to the case where $X$ is the spectrum of a finitely generated polynomial algebra over $k$, and the result follows  by definition of $\Delta_{i - 1}$. 
\end{proof}
\begin{corollary}\label{corolsigma}
Let $k$ be a ring and let $X$ be a stack over $k$. Write 
\[
\Sigma^{+} := \Sigma^{+,i}_{\LL_{X / k}[1]}
\]
Then there exists a commutative diagram
\[
\begin{tikzcd}
\LL_{X / k}^{\otimes i}[i] \dar{\Sigma^{+}} \rar & \LL_{X / k}^i[i] \dar{\Delta_{i - 1}} \\
\LL_{X / k}^{\otimes i}[i] \rar & \LL_{X / k}^{i - 1}[i - 1] \underset{\catO_X}{\otimes} \LL_{X / k}[1]
\end{tikzcd}
\]
in $\derD(X)$.
\end{corollary}
\begin{proof}
Combine Lemma \ref{sign1} and Lemma \ref{sign2}.
\end{proof}

\subsection{Relating the obstruction classes}
\label{sec_relate}
In this section we relate the obstruction class to deforming a perfect complex with the Hodge-theoretic obstruction classes of its Chern character staying in the Hodge filtration along the deformation. The short slogan is that there exists a semiregularity map between the two obstruction spaces, mapping the former to the latter. This is essentially a result from Buchweitz and Flenner, see \cite[Proposition 4.2]{buchflen}, however they only defined the Hodge-theoretic obstruction in a restricted characteristic zero setting, using Bloch's technique. 

We start by defining the map that will relate the obstruction classes.
\begin{definition}\label{def_semiregularity}
Let $i, j \in \ZZ_{\geq 0}$. Let $k$ be a ring such that $i!$ is invertible in $k$. Let $X$ be a scheme over $k$, let $\mathcal{E} \in \Perf(X)$. We define the \emph{semiregularity map}
\[
\sigma_{X, i} \colon \Ext^j_X(\mathcal{E}, \mathcal{E}) \to \coh^{i + j}(\LL^i_{X / k})
\]
as the composition
\[
\Ext^j_X(\mathcal{E}, \mathcal{E}) \xrightarrow{\frac{\At_{X / k}^i(\mathcal{E})}{i!}} \Ext_X^{i + j}(\mathcal{E}, \mathcal{E} \otimes \LL^i_{X / k}) \xrightarrow{\tr_{\mathcal{E}}} \coh^{i + j}(\LL^i_{X / k})
\]
where the first map is induced by postcomposing with $\frac{\At_{X / k}^i(\mathcal{E})}{i!}$.  

If $X$ is smooth over $k$ of dimension $d$ such that $d!$ is invertible in $k$, we write
\[
\sigma_{X} \colon \Ext^j_X(\mathcal{E}, \mathcal{E}) \to \bigoplus_{i = 0}^d \coh^{i + j}(\Omega^i_{X / k})
\]
for the \emph{total semiregularity map} given componentwise by $\sigma_{X, i}$. 
\end{definition}
The following result contains the essential computation, expressing the image of the obstruction class under the semiregularity map in terms of the Atiyah class. 
\begin{lemma}\label{tr_e}
Let $k$ be a ring such that $i!$ is invertible in $k$, and let $R' \to R$ be a surjective ring map with kernel $I$ such that $I^2 = 0$. Let $X'$ be a smooth and proper scheme over $R'$, let $X = X' \times_{\Spec(R')} \Spec(R)$  and let $\mathcal{I} = I \otimes_{R} \catO_X$. Let $\mathcal{E} \in \Perf(X)$ and let $\mathcal{E}_0 := \mathcal{E} \rvert_{X_0}$. 

If $\alpha$ denotes the composition
\[
\begin{tikzcd}[column sep =3cm,row sep=0cm]
\mathcal{E} \rar{\At_{X / k}(\mathcal{E})^{i}} &\mathcal{E} \otimes \LL^{i}_{X / k}[i] \\ 
\ \rar{\mathcal{E} \otimes \Delta_{i - 1}[i - 1]}& \mathcal{E} \otimes  \LL_{X / k}^{i - 1}[i - 1] \otimes \LL_{X  / k} [1] \\ 
\ \rar{\mathcal{E} \otimes \LL^{i - 1}_{X / k}[i - 1] \otimes \kappa_{X / X' / k}[1]} & \mathcal{E} \otimes \LL^{i - 1}_{X / k}[i - 1] \otimes \mathcal{I}[2]
\end{tikzcd}
\]
then 
\[
\tr_{\mathcal{E}}(\alpha) = i! \cdot \sigma_{X, i - 1}(\ob(\mathcal{E}, X, X'))
\]
in $\coh^{i + 1}(\LL^{i - 1}_{X / k} \otimes \mathcal{I})$
\end{lemma}
\begin{proof}
By Corollary \ref{corolsigma}, we obtain a commutative diagram
\[
\begin{tikzcd}
\mathcal{E} \dar{\At_{X / k}(\mathcal{E})^i} \rar{\At_{X / k}(\mathcal{E})^i} & \mathcal{E} \otimes \LL_{X / k}^i \dar{\Delta_{i - 1}} \\
\mathcal{E} \otimes (\LL_{X / k}[1])^{\otimes i} \rar{\Sigma^+} & \mathcal{E} \otimes \LL_{X / k}^{i - 1} \otimes \LL_{X / k}
\end{tikzcd}
\]
By symmetry of the trace map, it follows that
\begin{align*}
\tr_{\mathcal{E}}((\mathcal{E} \otimes \Delta_{i - 1}) \circ \At_{X / k}(\mathcal{E})^i) &= \tr_{\mathcal{E}}\((\mathcal{E} \otimes \Sigma^+) \circ \At_{X / k}(\mathcal{E})^{i}\) \\
&= \sum_{j = 1}^i \sigma^{j, i}_{\LL_{X / k}} \circ \tr_{\mathcal{E}}(\At_{X / k}(\mathcal{E})^{i})\\
&= i \cdot \tr_{\mathcal{E}}(\At_{X / k}(\mathcal{E})^{i})
\end{align*}
in $\coh^i(\LL_{X / k}^{i - 1} \otimes \LL_{X / k})$. Thus
\[
\tr_{\mathcal{E}}(\alpha) = i \cdot \tr_{\mathcal{E}}((\id \otimes \kappa_{X / X'/ k}[1]) \circ \At_{X / k}(\mathcal{E})^{i})
\]
in $\coh^i(\LL_{X / k}^{i - 1}) \otimes_{R} I$. By Proposition \ref{prop_obstruction_class}, we obtain
\[
\tr_{\mathcal{E}}(\alpha) = i \cdot \tr_{\mathcal{E}}(\At_{X / k}(\mathcal{E})^{i - 1} \circ \ob(\mathcal{E}, X, X'))
\]
which proves the result.
\end{proof}
The following result compares the obstruction classes of a complex and it's Chern character in characteristic zero.
\begin{theorem}\label{thm_compare_obs_char0}
Let $k$ be a ring with $\QQ \subseteq k$, and suppose $R' \to R$ is a square zero morphism of nilpotent thickenings with $I = \ker(R' \to R)$. Let $X'  \xrightarrow{f} \Spec(R')$ be a smooth and proper morphism, and set
\begin{align*}
X &:= X'  \times_{\Spec(R' )} \Spec(R) \\
X_0 &:= X'  \times_{\Spec(R' )} \Spec(k)
\end{align*}
Write $\mathcal{I} := I \otimes_{R} \catO_{X}$. Let $\mathcal{E} \in \Perf(X)$ and let $\mathcal{E}_0 := \mathcal{E} \rvert_{X_0}$. 

Then for all $i \geq 1$, the semiregularity map
\[
\sigma_{X, p} \colon \Ext_X^2(\mathcal{E}, \mathcal{E} \otimes \mathcal{I}) \to \coh^{i + 1}(\Omega^{i  -1}_{X / R}\otimes \mathcal{I})
\]
sends $\ob(\mathcal{E}, X, X')$ to $\ob^{\widehat{\dR}}_{X' / R'}(\ch_{i}(\mathcal{E}_0))$. 
\end{theorem}
\begin{proof}
By Proposition \ref{prop_compare_chern} we have $\tr_{\mathcal{E}}(\At_{X / k}(\mathcal{E})^{i}) = i!  \cdot \ch_{i}(\mathcal{E})$. Thus by Lemma \ref{tr_e}
\begin{align*}
\sigma_{X, i - 1}(\ob(\mathcal{E}, X, X')) = \frac{1}{i!} \tr_{\mathcal{E}}(\alpha) = (\kappa_{X / X' / k} \circ \Delta_{i - 1})(\ch_{i}(\mathcal{E}))
\end{align*}
in $\coh^{i + 1}(\LL^{i - 1}_{X / k})$. Since the diagram
\[
\begin{tikzcd}
\coh^{i}( \LL^{{i}}_{X / k}) \dar \rar{\Delta_{{i - 1}}} & \dar \coh^i( \LL^{i  - 1}_{X / k} \otimes \LL_{X / k}) \rar{\kappa_{X / X' / k}} \rar& \coh^{i + 1}(\LL^{i - 1}_{X / k}) \rar \dar & 
\coh^{i + 1}( \LL^{i - 1}_{X / R}) \arrow[equal]{d} \\
\coh^i( \LL^{i}_{X / R' }) \rar{\Delta_{i - 1}} & \coh^i( \LL^{i - 1}_{X / R' } \otimes \LL_{X / R' }) \rar{\kappa_{X / X' / R' }}& 
\coh^{i + 1}( \LL^{i - 1}_{X / R' }) \rar & \coh^{i + 1}(\LL^{i - 1}_{X / R}) 
\end{tikzcd}
\]
commutes, the result follows from Lemma \ref{lem_chern_horizontal} and Proposition \ref{label_chern_result}. 
\end{proof}
And we have the following result in mixed characteristic.
\begin{theorem}\label{thm_compare_obs_p}
Let $k$ be a ring over $\ZZ / p^n \ZZ$ for some $n \geq 1$, and let $R_0$ be a $k$-algebra. Let 
\[
(R' \to R_0, \gamma' ) \to (R \to R_0, \gamma)
\]
 be a morphism in $\PDPair_k$ such that $R' \to R$ is a surjection with kernel $I$ and $I^{[2]} = 0$. 

Let $X'  \xrightarrow{f} \Spec(R')$ be a smooth and proper morphism, and set
\begin{align*}
X &:= X'  \times_{\Spec(R' )} \Spec(R) \\
X_0 &:= X'  \times_{\Spec(R' )} \Spec(k)
\end{align*}
Write $\mathcal{I} := I \otimes_{R} \catO_{X}$.  Suppose that $i!$ is invertible in $k$.  Let $\mathcal{E} \in \Perf(X)$   and let $\mathcal{E}_0 := \mathcal{E} \rvert_{X_0}$. 

Then the semiregularity map
\[
\sigma_{X, i - 1} \colon \Ext_X^2(\mathcal{E}, \mathcal{E} \otimes \mathcal{I}) \to \coh^{i + 1}(\Omega^{i - 1}_{X / R}\otimes \mathcal{I})
\]
sends $\ob(\mathcal{E}, X, X')$ to $\ob^{{\Crys}}_{X' / R'}(\ch_{i}(\mathcal{E}_0))$. 
\end{theorem}
\begin{proof}
Again, using that $i!$ is invertible, by Proposition \ref{prop_compare_chern} and Lemma \ref{tr_e} we have
\[
\sigma_{X, i - 1}(\ob(\mathcal{E}, X, X'))  = (\kappa_{X / X' / k} \circ \Delta_{i - 1})(\ch_{i}(\mathcal{E}))
\]
in $\coh^{i + 1}(\LL^{i - 1}_{X / k})$. Thus the result follows from Lemma \ref{lem_chern_horizontal} and Proposition \ref{label_chern_result_pd}. 
\end{proof}

\section{Hochschild (co)homology and the semiregularity map}
In this section, we relate the semiregularity map from Definition \ref{def_semiregularity} with Hochschild--theoretic constructions. There are two main results we need. The first is in Section \ref{hochschild_semi}, where we  relate the semiregularity map with a Hochschild--theoretic semiregularity map (Definition \ref{def_hocshchild_semiregular} and Proposition \ref{prop_semireg_compare}). The second is Corollary \ref{corol_actions}, relating the Hochschild--theoretic semiregularity map with the action from Hochschild cohomology on Hochschild homology. Later, in Section \ref{sec_injective} these results will combine to show the semiregularity map is injective in specific cases.
\label{sec_hochschild_semiregular}
\subsection{Fourier--Mukai transforms, duality and Hochschild (co)homology}
If $X$ and $Y$ are smooth and proper schemes over a ring $k$, then any perfect complex $\mathcal{E} \in \Perf(X \times Y)$ induces a functor
\begin{align*}
\Phi_{\mathcal{E}} \colon \Perf(X) &\to \Perf(Y) \\
\mathcal{F} &\mapsto \pi_{Y, *}(\pi_X^*(\mathcal{F}) \otimes \mathcal{E})
\end{align*}
We shall refer to $\Phi_{\mathcal{E}}$ as the \emph{Fourier-Mukai transform} associated to $\mathcal{E}$, and to $\mathcal{E}$ as the \emph{kernel} associated to $\Phi_{\mathcal{E}}$. 

For any three smooth and proper schemes $X, Y$ and $Z$ over a field $k$ we have a projection map
\[
\pi_{XY} \colon X \times Y \times Z \to X \times Y
\]
Similarly we have projections $\pi_{XZ}$ and $\pi_{YZ}$. For any two objects $\mathcal{E} \in \Perf(X \times Y)$ and $\mathcal{F} \in \Perf(Y \times Z)$, we shall write
\begin{equation} \label{faa}
\mathcal{F} \star \mathcal{E} := \pi_{XZ, *}(\pi_{XY}^* \mathcal{E} \otimes \pi_{YZ}^* \mathcal{F}) \in \Perf(X \times Z)
\end{equation}
One may show that $\Phi_{\mathcal{F}} \circ \Phi_{\mathcal{E}} \simeq \Phi_{\mathcal{F} \star \mathcal{E}}$ as functors $\Perf(X) \to \Perf(Z)$, see e.g. \cite[Proposition 5.10]{huybrechts}.

From now on, we restrict our attention to the case where $k$ is a field. Since we will use many techniques from \cite{caldararu1, caldararu2}, in this case we adapt to match Căldăraru's notation. In particular, for any smooth and proper scheme $X$ over a field $k$ we shall write $\derD^b(X) := \Perf(X)$ to match the notation of Căldăraru.
For $X$ a smooth and proper scheme over a field $k$ of dimension $d$ we denote with $\Delta \colon X \to X \times X$ the diagonal embedding. We will write $\catO_{\Delta_X} := \Delta_* \catO_X \in \derD^b(X \times X)$.  We will write $S_X = \Omega_X^d[d] \in \derD^b(X)$, and we will sometimes denote with $S_X(-)$ the functor 
\[
S_X  \otimes(-) \colon \derD^b(X) \to \derD^b(X)
\]
The starting point for most of the constructions is the following classical theorem.
\begin{theorem}[Grothendieck--Verdier--Serre duality]\label{gvs}
Let $X$ be a smooth and proper scheme over a field $k$. There exists a map
\[
\tr_X \colon \Hom_X(\catO_X, S_X) \to k
\]
such that for any $\mathcal{E}, \mathcal{F} \in \Perf(X)$ the pairing
\begin{align*}
\Ext^{-*}(\mathcal{E}, S_X \mathcal{F}) \otimes_k \Ext^*(\mathcal{F}, \mathcal{E}) &\to k\\
f \otimes g \mapsto \langle f, g \rangle := \tr_X(\tr_{\mathcal{F}}(f \circ g))
\end{align*}
is a perfect pairing (see Definition \ref{def_trace_map} for the definition of $\tr_{\mathcal{F}}$).
\end{theorem}
\begin{proof}
The existence of a perfect pairing is well known, see for example \cite[Theorem 3.12]{huybrechts}. This explicit description of the pairing is given by Căldăraru, see \cite[\S 2.2]{caldararu1}. The proof identifying Căldăraru's construction with more classical constructions can be found in  \cite[Theorem 17]{caldararu1b}.
\end{proof}

We will often write $\Tr_X(f) := \tr_X \tr_{\mathcal{F}}(f)$ for $f \in \Hom_X(\mathcal{F}, \mathcal{F} \otimes S_X)$. 
\begin{definition}\label{beb}
Let $k$ be a field and let $X, Y$ be smooth and proper schemes over $k$. Let $\mathcal{E}, \mathcal{F} \in \Perf(X)$ and let $\mathcal{E}', \mathcal{F'} \in \Perf(Y)$. Then any map
\[
\Phi \colon \Ext^*_X(\mathcal{F}, \mathcal{E}) \to \Ext^*_Y(\mathcal{F}', \mathcal{E}')
\]
has a unique left adjoint for the pairing from Theorem \ref{gvs}. That is, there exists a unique map
\[
\Psi \colon \Ext^*_Y(\mathcal{F}', \mathcal{E}')^\vee \to \Ext^*_X(\mathcal{F}, \mathcal{E})^\vee
\]
such that for any $f \in \Ext^*_Y(\mathcal{F}', \mathcal{E}')^\vee \cong \Ext^{-*}_Y(\mathcal{E}', S_Y \mathcal{F}')$ and $g \in \Ext_X^*(\mathcal{F}, \mathcal{E})$ one has $\langle f, \Phi(g) \rangle = \langle \Psi(f), g \rangle$. We will refer to $\Psi$ as the \emph{Serre left adjoint} of $\Phi$, and to $\Phi$ as the \emph{Serre right adjoint} of $\Psi$. 
\end{definition}
For $X$ a smooth and proper scheme over a ring $k$, we write $\Delta_! \colon \derD^b(X) \to \derD^b(X \times X)$ for the left adjoint of $\Delta^*$. Explicitly, $\Delta_!$ is given by
\begin{equation} \label{bga}
\mathcal{F} \mapsto S_{X \times X}^{-1} \otimes \Delta_* (S_X \otimes \mathcal{F})
\end{equation}
for $\mathcal{F} \in \derD^b(X)$. 
\begin{definition}
Let $k$ be a field, and let $X$ be a smooth and projective scheme over $k$. We define the \emph{Hochschild homology} of $X$ as
\[
\HH_*(X) := \Ext^{- *}_{X \times X}(\Delta_! \catO_X, \Delta_* \catO_X)
\]
where $\Delta_!$ is the left adjoint of 	$\Delta^*$.
\end{definition}

\begin{definition}
Let $k$ be a field, and let $X$ be a smooth and projective scheme over $k$. We define the \emph{Hochschild cohomology} of $X$ as
\[
\HH^*(X) := \Ext^*_{X \times X}(\Delta_* \catO_X, \Delta_* \catO_X)
\]

\end{definition}
Note that by Serre duality one has
\begin{align*}
\HH_*(X)^\vee &= \Ext_{X \times X}^*(\Delta_* \catO_X, \Delta_* S_X) \\
\HH^*(X)^\vee &= \Ext_{X \times X}^*(\Delta_* \catO_X, \Delta^* \catO_X \otimes S_{X \times X})
\end{align*}
since $\Delta_! \catO_X = S_{X \times X}^{-1} \Delta_* S_X$. 
\subsection{The semiregularity map for Hochschild homology}
\label{hochschild_semi}
Let $X$ be a smooth and proper scheme over a field $k$. In this section, we define the Hochschild--theoretic semiregularity map 
\[
\sigma^{\HH_*}_{\mathcal{E}} \colon \Ext^*(\mathcal{E}, \mathcal{E}) \to \HH_{-*}(X)
\]
for any $\mathcal{E} \in \derD^b(X)$, which is essentially due to \cite{buchflen2}. Moreover, we show the this map corresponds to the semiregularity map(s) defined in Definition \ref{def_semiregularity} under the Hochschild-Kostant-Rosenberg isomorphism. 
\begin{definition}
We define the \emph{universal Hochschild--Atiyah character}
\[
\At_X^{\HH} \in \Hom_{\derD(X \times X)}(\catO_{\Delta_X}, \Delta_* \Delta^* \catO_{\Delta_X})
\]
to be the unit of the adjunction $\Delta^* \dashv \Delta_*$. 
\end{definition}
\begin{construction} \label{agb} 
Let $k$ be a ring, and let $X, Y$ be smooth and proper schemes over $k$. Let
\[
\alpha \colon \mathcal{F} \to \mathcal{G}
\]
be a morphism in $\derD(X \times Y)$. Then $\alpha$ induces a natural transformation 
\begin{align*}
\Phi_{\alpha}(-) \colon \Phi_{\mathcal{F}}(-) \to \Phi_{\mathcal{G}}(-)
\end{align*}
of functors $\derD(X) \to \derD(Y)$, sending $\mathcal{E} \in \derD(X)$ to the morphism
\[
\Phi_{\alpha}(\mathcal{E}) := \pi_{2*}(\pi_1^* \mathcal{E} \otimes \alpha)
\]
in $\derD(Y)$. 
\end{construction}
In particular for any $\mathcal{E} \in \derD(X)$, we get a map
\begin{equation} \label{bed}
\At_X^{\HH}(\mathcal{E}) := \Phi_{\At_X^{\HH}}(\mathcal{E}) \colon \mathcal{E} \to \mathcal{E} \otimes \Delta^* \catO_{\Delta_X}
\end{equation}
in $\derD(X)$.

We now wish to relate $\At_X^{\HH}$ with the universal Atiyah class (see Definition \ref{aea}), by means of the Hochschild--Kostant--Rosenberg isomorphism. For this, we need some setup first. 

Let $k$ be a ring and let $X$ be a smooth and proper scheme over $k$. If one writes $J_\Delta \subseteq \catO_{X \times X}$ for the ideal sheaf of the diagonal, we let $\Delta^{(2)}_{X} \subseteq X \times X$ be the nilpotent thickening of $\Delta_X$ corresponding to the ideal $J_\Delta^2$. 
\begin{lemma}
There exists an an isomorphism
\begin{equation} \label{baa}
\varphi_\Delta \colon \catO_X \oplus \Omega_{X / k} \xrightarrow{\sim} \catO_{\Delta^{(2)}_{X}}
\end{equation}
 of sheaves of rings on $X$, given locally by $(f, g \mathrm{d}x) \mapsto 1 \otimes f+ x \otimes g - 1 \otimes gx$.
\end{lemma}
\begin{proof}
We first verify that $\varphi_\Delta$ is a well--defined map of sheaves of abelian groups. Since it is clearly linear, we only need to check compatibility with the Leibniz rule. Indeed, one sees that
\begin{align*}
\varphi((0,  \mathrm{d} xy) &- (0, x \mathrm{d}y) - (0, y \mathrm{d}x))  \\
&= xy \otimes 1 - 1 \otimes xy - (y \otimes x - 1 \otimes xy) - (x \otimes y - 1 \otimes xy) \\
&= xy \otimes 1 + 1  \otimes xy - y \otimes x - x \otimes y \\
&= (x \otimes 1 - 1 \otimes x) \cdot (y \otimes 1 - 1 \otimes y)
\end{align*}
lies in $J_\Delta^2$. To check that it is a ring map, we compute
\begin{align*}
\varphi((f_1, &g_1 dx_1)) \cdot  \varphi((f_2, g_2 dx_2)) \\
& = (1 \otimes f_1 + (1 \otimes g_1)  (x_1 \otimes 1 - 1 \otimes x_1)) \\
& \qquad \cdot (1 \otimes f_2 + (1 \otimes g_2)  (x_2 \otimes 1 - 1 \otimes x_2))  \\
& \equiv1 \otimes  f_1 f_2 + (1 \otimes f_1 g_2)  (x_2 \otimes 1 - 1 \otimes x_2) \\
&\qquad + (1 \otimes f_2 g_1)  (x_1 \otimes 1 - 1 \otimes x_1) & (\mathrm{mod} \ J_\Delta^2) \\ 
&= \varphi((f_1 f_2, f_1 g_2 dx_2 + f_2 g_1 dx_1)) \\
&= \varphi((f_1, g_1 dx_1) \cdot (f_2, g_2 dx_2))
\end{align*}
as desired. We leave it to the reader to verify that the map is an isomorphism.
\end{proof}

Write $X^\Omega = \Spec_X(\catO_X \oplus \Omega_{X / k})$, so that $\varphi_\Delta$ induces an isomorphism
\begin{equation} \label{bab}
\Delta^{(2)}_X \xrightarrow{\sim} X^\Omega
\end{equation}
of schemes over $k$.
\begin{lemma}\label{bac}
The map (\ref{bab}) fits in a commutative diagram
\[
\begin{tikzcd}
X \\
X^\Omega \uar{\pi_\delta} \dar[swap]{\pi_0}  & \lar{\sim} \Delta^{(2)}_X \arrow{ul}[swap]{\pi_1} \arrow{dl}{\pi_2} \\
X 
\end{tikzcd}
\]
\begin{proof}
By the definition of $\pi_\delta$ and $\pi_0$ (Section \ref{sec_atiyah}) it suffices to show the diagram
\[
\begin{tikzcd}
\catO_X \dar[swap]{(\id, \mathrm{d})} \arrow{dr}{\pi_1^\#} \\
\catO_X \oplus \Omega_{X /k}   \rar{\varphi_\Delta} & \Delta^{(2)}_X  \\
X  \uar{(\id, 0)} \arrow{ur}[swap]{\pi_2^\#}
\end{tikzcd}
\]
commutes. For $f$ a local section of $\catO_X$, we compute
\begin{align*}
\varphi_\Delta \circ (\id, d) (f) &= \varphi_\Delta ((f, \mathrm{d}f)) = 1 \otimes f + f \otimes 1 - 1 \otimes f = f \otimes 1 = \pi_1^\#(f) \\
\varphi_\Delta \circ (\id, 0) (f) &= \varphi_\Delta((f, 0)) = 1 \otimes f = \pi_2^\#(f)
\end{align*}
as required.
\end{proof}
\end{lemma}
\begin{definition}
The universal Atiyah class $\alpha_X$ (see Definition \ref{aea}) can be represented by a map
\[
\catO_X \to \Omega_{X / k}[1]
\]
in $\derD(X^\Omega)$. Under the isomorphism (\ref{bab}), this induces a map
\[
\catO_{\Delta_X} \to \Delta_* \Omega_{X / k}[1]
\]
in $\derD(\Delta^{(2)}_X)$. Pushing forward along the closed immersion $\Delta^{(2)}_X \subseteq X \times X$, we get a map
\[
\catO_{\Delta_X} \to \Delta_* \Omega_{X / k}[1]
\]
in $\derD(X \times X)$. We define $\tilde{\alpha}_X \in \Ext^1_{X \times X}(\catO_{\Delta_X}, \Delta_* \Omega_{X / k})$ to be the element corresponding to this map.
\end{definition}
Explicitly, $\tilde{\alpha}_X$ is given by the extension
\[
0 \to \Delta_* \Omega^1_{X / k} \to \catO_{\Delta^{(2)}_X} \to \catO_{\Delta_X} \to 0
\]
where the first map is locally given by sending $g \mathrm{d} x \mapsto x \otimes g - 1 \otimes gx$. We denote with
\[
\exp(\tilde{\alpha}_X) \colon \catO_{\Delta_X} \to \bigoplus_{i = 0}^d \Delta_*   \Omega^i_{X}[i]
\]
the map whose $i$-th component is the composition of the maps
\begin{align*}
\catO_{\Delta_X} &\xrightarrow{\tilde{\alpha}_X} \Delta_*\Omega_{X / k}[1] \\
&\xrightarrow{\tilde{\alpha}_X  \otimes \Delta_* \Omega_{X / k}[1]}  \Delta_* \Omega_{X / k}^{\otimes 2}[2]  \\
&\qquad \qquad \vdots \\
&\xrightarrow{\tilde{\alpha}_X \otimes \Delta_* \Omega_{X / k}^{\otimes(i - 1)}[i - 1]}   \Delta_* \Omega_{X / k}^{\otimes i}[i] \\ &\xrightarrow{\Delta_* \epsilon}  \Delta_* \Omega^i_{X / k}[i]
\end{align*}
where $\epsilon$ is locally given by $v_1 \otimes \dots \otimes v_i \mapsto \frac{1}{i!} v_{1} \wedge \dots \wedge v_{i}$. 
\begin{lemma}\label{afd}
Let $k$ be a field and let $X$ be a smooth and proper scheme over $k$ of dimension $d$, such that $d!$ is invertible in $k$.  Let $\mathcal{E} \in \derD^b(X)$. Then the $i$-th component of the map (see Construction \ref{agb})
\[
\Phi_{\exp(\tilde{\alpha}_X)}(\mathcal{E}) \colon \mathcal{E} \to \bigoplus_{i = 0}^d \mathcal{E} \otimes \Omega^i_X[i]
\]
is given by $\frac{1}{i!} \At_X^i(\mathcal{E})$. 
\end{lemma}
\begin{proof}
Denote with $\tilde{\alpha}_X^i \in \Ext^i_{X \times X}(\catO_{\Delta_X}, \Delta_* \Omega^p_{X / k})$ the composition 
\begin{multline*}
\catO_{\Delta_X} \xrightarrow{\tilde{\alpha}_X} \Delta_*\Omega_{X / k}[1] \xrightarrow{\tilde{\alpha}_X  \otimes \Delta_* \Omega_{X / k}[1]}  \dots \\ 
\xrightarrow{\tilde{\alpha}_X \otimes \Delta_* \Omega_{X / k}^{\otimes(i - 1)}[i - 1]}   \Delta_* \Omega_{X / k}^{\otimes i}[i] \xrightarrow{\Delta_* \epsilon}  \Delta_* \Omega^i_{X / k}[i]
\end{multline*}
and with $\alpha_X^i \in \Ext^i_{X^\Omega}(\catO_X, \Omega_{X / k}^i)$ the composition
\[
\catO_{X} \xrightarrow{{\alpha}_X} \Omega_{X / k}[1] \xrightarrow{{\alpha}_X  \otimes \Omega_{X / k}[1]}  \dots \xrightarrow{{\alpha}_X \otimes \Omega_{X / k}^{\otimes(i - 1)}[i - 1]}   \Omega_{X / k}^{\otimes i}[i] \xrightarrow{ \epsilon} \Omega^i_{X / k}[i]
\]
Lemma \ref{bac} then implies that
\begin{equation}\label{eaa}
\Phi_{\exp(\tilde{\alpha}^i_X)}(\mathcal{E}) = \pi_{0, *}(\pi_{\delta}^*(\mathcal{E}) \otimes \alpha_X^i)
\end{equation}
in $\Ext^i_X(\mathcal{E}, \mathcal{E} \otimes \Omega_{X / k}^i)$. The result follows as by definition of the Atiyah class (Definition \ref{aeb}), the right hand side of (\ref{eaa}) is equal to $\frac{1}{i!} \At_{X / k}^i(\mathcal{E})$.

\end{proof}
\begin{theorem}[Hochshild--Kostant--Rosenberg, \cite{hkr}, \cite{caldararu2}]\label{thm_hkr}
Let $k$ be a field and let $X$ be a smooth and proper scheme over $k$ of dimension $d$, such that $d!$ is invertible in $k$. Then there exists an isomorphism
\begin{equation} \label{afc}
I \colon \Delta^* \catO_{\Delta_X} \xrightarrow{\sim} \bigoplus_{i = 0}^d \Omega^i_X[i]
\end{equation}
in $\derD(X)$, such that there exists a commutative diagram
\begin{equation} \label{bea}
\begin{tikzcd}
& \catO_{\Delta_X} \arrow[swap]{dl}{\At_X^{\HH}} \arrow{d}{\exp(\tilde{\alpha}_X)} \\
\Delta_* \Delta^* \catO_{\Delta_X} \rar{\Delta_*(I)} & \bigoplus_{i = 0}^d \Delta_*   \Omega^i_{X}[i]
\end{tikzcd}
\end{equation}
in $\derD(X \times X)$. 
\end{theorem}
\begin{proof}
Yekutili \cite{hkr} originally showed the existence of an isomorphism $I$. It was shown by Căldăraru \cite[\S 4]{caldararu2} that $I$ can be chosen such that a diagram (\ref{bea}) exists (technically speaking Căldăraru assumes $k = \CC$, but if one reads \S 4 of \cite{caldararu2} carefully one sees that he only uses that $d!$ is invertible in $k$).
\end{proof}
\begin{construction}[Of the map $I_{\HKR}^{\HH_j}$]
Let $j \in \ZZ_{\geq 0}$. Let $k$ be a field and let $X$ be a smooth and proper scheme over $k$ of dimension $d$, such that $d!$ is invertible in $k$. Then the composition
\begin{align*}
\HH_j(X)  &= \Ext^{ - j}_{X \times X}(\Delta_! \catO_X, \Delta_* \catO_X) \\
&\cong  \Ext^{ - j}_X(\catO_X, \Delta^*  \Delta_* \catO_X) & (\Delta_! \dashv \Delta^*) \\
&\cong  \bigoplus_{i = 0}^d \Ext^{- j}_X(\catO_X, \Omega_X^{i}[i]) & (\ref{afc})
\end{align*}
defines an isomorphism 
\[
I_{\HKR}^{\HH_j} \colon \HH_j(X) \to \bigoplus _{i = 0}^d \coh^{i - j}(X, \Omega^{i}_{X / k}) 
\]
\end{construction} 

\begin{definition}\label{def_hocshchild_semiregular}
Let $k$ be a field, and let $X$ be a smooth and proper scheme over $k$. For $P \in \derD^b(X)$, define the \emph{Hochschild--theoretic semiregularity map} 
\[
\sigma^{\HH_*}_P \colon \Ext_X^*(P, P) \to \HH_{-*}(X)
\]
as the composition
\begin{align*}
\Ext_X^*(P, P) &\xrightarrow{\At_X^{\HH}(P) \circ (-)}  \Ext_X^*(P, P \otimes \Delta^* \Delta_* \catO_X) \\
&\xrightarrow{\tr_P} \Ext_X^*(\catO_X, \Delta^* \Delta_* \catO_X) \xrightarrow{\sim} \HH_{-*}(X)
\end{align*}
where the last isomorphism is given by the adjunction $\Delta_! \dashv \Delta^*$.
\end{definition}
\begin{remark}
Although useful for relating it to the classical Chern character, in practice the above definition is rather hard to work with. Instead we will often work with its Serre left adjoint, which can be shown to be ``evaluation at $P$'', see Lemma \ref{lem_semiregular_adjoint} for a precise statement.
\end{remark}
\begin{prop}\label{prop_semireg_compare}
Let $k$ be a field and let $X$ be a smooth and proper scheme over $k$ of dimension $d$, such that $d!$ is invertible in $k$. Let $i \in \ZZ$. The diagram
\[
\begin{tikzcd}
& \Ext_X^i(P, P) \arrow{dl}[swap]{\sigma_P^{\HH_{i}}} \arrow{d}{\sigma_X}\\
\HH_{-i}(X) \arrow{r}{\sim} [swap]{I^{\HH_{-i}}_\HKR}& \bigoplus _{j = 0}^d  \coh^{i + j}(X, \Omega^{j}_{X / k})
\end{tikzcd}
\]
commutes, where $\sigma_X$ is the map given componentwise by the semiregularity maps defined in Definition \ref{def_semiregularity}.
\end{prop}
\begin{proof}
Let $f \in \Ext^i(P, P)$. By definition one has
\[
\sigma^j_X(f) = \tr_{P}(\At^j_X(P) \circ f)
\]
so that using Lemma \ref{afd} we obtain
\[
\sigma_X(f) = \tr_{P}(\pi_{2, *}(\pi_1^*P \otimes \exp(\tilde{\alpha}_X)) \circ f)
\]
Using the diagram (\ref{bea}) we may rewrite this as
\begin{align*}
\sigma_X(f) &= \tr_{P}(\pi_{2, *}(\pi_1^*P \otimes (\Delta_* I \circ \At_X^{\HH})) \circ f) \\
&= I \circ \tr_{P}(\pi_{2, *}(\pi_1^*P \otimes \At_X^{\HH}) \circ f) \\
&= I \circ \tr_{P}(\At_X^{\HH}(P) \circ f) \\
&= I \circ \sigma^{\HH_{i}}_P(f)
\end{align*}
where the second equality follows from the projection formula, the third by definition of $\At_X^{\HH}$ and the last by definition of $\sigma^{\HH_{i}}$. The result follows.
\end{proof}

\subsection{Functoriality for Fourier--Mukai transforms}
Let $X$ and $Y$ be smooth and proper schemes over a field $k$. In this section, we define for any Fourier--Mukai transform $\Phi_P \colon \derD^b(X) \to \derD^b(Y)$ a map 
\[
\Phi^{\HH_*}_P \colon \HH_*(X) \to \HH_*(Y)
\]
Moreover, we show that this map is compatible with the Hochschild--theoretic semiregularity maps defined in Definition \ref{def_hocshchild_semiregular} (see Proposition \ref{prop_functorial_semiregular} below for a precise statement).

Although it is possible to define the map $\Phi^{\HH_*}_P$ without referring to Serre duality explicitly (see \cite[\S 6.2]{addingtonthomas14}), for some reason it appears to be quite difficult to prove Proposition \ref{prop_functorial_semiregular} in this way directly. Instead, we follow the proof of \cite[Theorem 7.1]{caldararu1}, which is basically a slightly less general statement then Proposition \ref{prop_functorial_semiregular}.
\begin{definition}
Let $X$ and $Y$ be smooth and proper schemes over a field $k$. For $P \in \derD^b(X \times Y)$ we define
\begin{align*}
P_L := P^\vee \otimes \pi_Y^* S_Y \\
P_R := P^\vee \otimes \pi_X^* S_X
\end{align*}
in $\derD^b(X \times Y)$. 
\end{definition}
It is well known that $\Phi_{P_L} = \Phi_{P^\vee} \circ S_Y$ and $\Phi_{P_R} = S_X \circ \Phi_{P^\vee}$ are left, resp. right adjoint to $\Phi_P$. By abstract nonsense (see \cite[Proposition 5.1]{caldararu1}) there exist maps
\begin{align*}
\eta_P \colon \catO_{\Delta_X} \to P_R \star P \\
\epsilon_P \colon P_L \star P\to \catO_{\Delta_Y}
\end{align*}
in $\derD(X \times X)$ corresponding to the unit, resp. counit of the adjunctions $\Phi_P \dashv \Phi_{P_R}$ and $\Phi_{P_L} \dashv \Phi_P$. 

\begin{definition}[Functoriality for Hochschild homology]\label{def_hochschild_homology_functor}
Let $X$ and $Y$ be smooth and proper schemes over a field $k$. For $P \in \derD^b(X \times Y)$ we define
\[
\tilde{\Phi}_P \colon \Ext^*_{Y \times Y}(\Delta_* \catO_Y, \Delta_* S_Y) \to \Ext^*_{X \times X}(\Delta_* \catO_X, \Delta_* S_X)
\]
by sending a map
$\nu \colon \Delta_* \catO_Y \to \Delta_* S_Y[i]$ in $\derD(Y \times Y)$ to the composition
\begin{multline*}
\Delta_* \catO_X \xrightarrow{\eta_P} P_R \star P \simeq P_R \star \catO_{\Delta_Y} \star P \\
\xrightarrow{P_R \star \nu \star P} P_R \star S_Y \star P \simeq S_X \star P_L \star P \xrightarrow{\eta_P} \catO_{\Delta_X}
\end{multline*}
in $\derD^b(X \times X)$. 
We define
\[
\Phi^{\HH_*}_P \colon \HH_*(X) \to \HH_*(Y)
\]
as the Serre left adjoint of the map $\tilde{\Phi}_P$ (see Definition \ref{beb}). 
\end{definition}

\begin{definition}\label{bec}
Let $k$ be a field, and let $X$ be a smooth and proper scheme over $k$. For $\mathcal{E} \in \derD^b(X)$, define the \emph{evaluation map}
\begin{align*}
\ev_{\mathcal{E}} \colon \Ext^i_{X \times X}(\Delta_* \catO_X, \Delta_* S_X) &\to \Ext^i_X(\mathcal{E}, \mathcal{E} \otimes S_X) \\
\eta &\mapsto \Phi_\eta(\mathcal{E})
\end{align*}
see Construction \ref{agb}.
\end{definition}
\begin{lemma}\label{lem_semiregular_adjoint}
The map $\ev_{\mathcal{E}}$ is the Serre left adjoint of $\sigma^{\HH_*}_{\mathcal{E}}$, that is for any $\mu \in \Ext^i_{X \times X}(\Delta_* \catO_X, \Delta_* S_X)$ and $\alpha \in \Ext^{-i}_X(\mathcal{E}, \mathcal{E})$ one has
\[
\langle \ev_{\mathcal{E}}(\mu), \alpha \rangle = \langle \mu, \sigma^{\HH_*}_{\mathcal{E}}(\alpha) \rangle
\]
in $k$ (see Theorem \ref{gvs} for the definition of the pairing). 
\end{lemma}
\begin{proof}
Essentially the same argument as in \cite[Theorem 4.5]{caldararu2}. Suppose that  $\mu' \in \Ext^i_X (\Delta^* \Delta_* \catO_X, S_X)$ is the image of $\mu$ under the adjunction $\Delta^* \dashv \Delta_*$, so that $\mu = \Delta_* \mu' \circ \At_X^{\HH}$. 

We have 
\begin{align*}
\langle \ev_{\mathcal{E}}(\mu), \alpha \rangle &= \Tr_X\(\ev_{\mathcal{E}}(\mu) \circ \alpha\) \\
&= \Tr_X\(\Phi_{\mu}(\mathcal{E}) \circ \alpha\) & (\text{Definition \ref{bec}} )\\
&= \Tr_X\(\pi_{2, *}(\pi_1^*\mathcal{E} \otimes \mu) \circ \alpha\) & (\text{Construction \ref{agb}} ) \\
&= \Tr_X\(\pi_{2, *}(\pi_1^*\mathcal{E} \otimes (\Delta_* \mu' \circ \At_X^{\HH})) \circ \alpha\) & (\mu = \Delta_* \mu' \circ \At_X^{\HH}) \\
&= \Tr_X\(\pi_{2, *}(\Delta_*(\mathcal{E} \otimes \mu') \circ \pi_1^*\mathcal{E} \otimes  \At_X^{\HH} ) \circ \alpha \) & (\text{projection formula}) \\
&= \Tr_X\((\mathcal{E} \otimes  \mu') \circ \pi_{2, *}(\pi_1^* \mathcal{E} \otimes  \At_X^{\HH} ) \circ \alpha \) \\
&= \tr_X\(\mu' \circ \tr_{\mathcal{E}}(\At_X^{\HH}(\mathcal{E}) \circ \alpha)\) & (\ref{bed})
\end{align*}
The adjunctions $\Delta_! \dashv \Delta^*$ and $\Delta^* \dashv \Delta_*$ give equivalences
\begin{align*}
\varphi_1 \colon \Ext_X^*(\catO_X, \Delta^* \Delta_* \catO_X) &\xrightarrow{\sim} \Ext_{X \times X}^*(\Delta_! \catO_X, \Delta_* \catO_X) \\
\varphi_2 \colon \Ext_X^*(\Delta^* \Delta_* \catO_X, S_X) &\xrightarrow{\sim} \Ext_{X \times X}^*(\Delta_* \catO_X, \Delta_! \catO_X \otimes S_{X \times X})
\end{align*}
Here for the second equivalence, we used that $\Delta_! \catO_X \otimes S_{X \times X} \simeq \Delta_* S_X$, see (\ref{bga}). One may show these are compatible with the Serre trace, that is
\begin{align*}
\Tr_X(g \circ f) = \Tr_{X \times X}(\varphi_2(g) \circ \varphi_1(f))
\end{align*}
for $f \in \Ext_X^*(\catO_X, \Delta^* \Delta_* \catO_X)$ and $g \in\Ext_X^*(\Delta^* \Delta_* \catO_X, S_X)$. It follows that  
\begin{align*}
\langle \ev_{\mathcal{E}}(\mu), \alpha \rangle &= \tr_X\(\mu' \circ \tr_{\mathcal{E}}(\At_X^{\HH}(\mathcal{E}) \circ \alpha)\) \\
&= \tr_{X \times X}\(\tr_{\Delta_! \catO_X }(\mu \circ \sigma^{\HH_*}_{\mathcal{E}}(\alpha))\) \\
&= \langle \mu, \sigma^{\HH_*}_{\mathcal{E}}(\alpha) \rangle
\end{align*}
as required.
\end{proof}

\begin{prop}\label{prop_functorial_semiregular}
Let $X$ and $Y$ be smooth and proper schemes over a field $k$. For $P \in \derD^b(X \times Y)$ and $\mathcal{E} \in \derD^b(X)$, the diagram
\[
\begin{tikzcd}
\Ext^i_X(\mathcal{E}, \mathcal{E}) \dar {\sigma_{\mathcal{E}}^{\HH_{i}}} \rar{\Phi_P}& \Ext^i_Y(\Phi_P(\mathcal{E}), \Phi_P(\mathcal{E})) \dar{\sigma_{\Phi_P(\mathcal{E})}^{\HH_{i}}} \\
\HH_{i}(X) \rar{\Phi^{\HH_{-i}}_P} & \HH_{-i}(Y)
\end{tikzcd}
\]
commutes.
\end{prop}
\begin{proof}
Let 
\[
\Phi^\dagger_P \colon \Ext^i_Y(\Phi_P(\mathcal{E}), \Phi_P(\mathcal{E}) \otimes S_Y) \to \Ext^i_X(\mathcal{E}, \mathcal{E} \otimes S_X)
\]
be the map sending $\nu \colon \Phi_P(\mathcal{E}) \to \Phi_P(\mathcal{E}) \otimes S_Y$ to the composition
\begin{multline*}
\mathcal{E} \xrightarrow{\eta_P} \Phi_{P_R} \circ \Phi_P(\mathcal{E}) \xrightarrow{\Phi_{P_R}(\nu)} \Phi_{P_R} \circ S_Y \circ \Phi_P(\mathcal{E})  \\ \cong S_X \circ \Phi_{P_L} \circ \Phi_P(\mathcal{E}) \xrightarrow{S_X(\epsilon_P)} S_X \mathcal{E}
\end{multline*}
By \cite[Proposition 3.1]{caldararu1}, the map $\Phi^\dagger_P$ is a Serre left adjoint for $\Phi_P$. Using Lemma \ref{lem_semiregular_adjoint} and Definition \ref{def_hochschild_homology_functor} to identify the other Serre left adjoints of the diagram, we see it suffices to show the diagram
\[
\begin{tikzcd}
 \Ext^*_{Y \times Y}(\Delta_* \catO_Y, \Delta_* S_Y) \dar{\ev_{\Phi_P(\mathcal{E})}} \rar{\tilde{\Phi}_P} & \Ext^*_{X \times X}(\Delta_* \catO_X, \Delta_* S_X) \dar{\ev_{\mathcal{E}}} \\
 \Ext^i_Y(\Phi_P(\mathcal{E}), \Phi_P(\mathcal{E}) \otimes S_Y) \rar{\Phi^\dagger_P} &  \Ext^i_X(\mathcal{E}, \mathcal{E} \otimes S_X)
\end{tikzcd}
\]
commutes (by uniqueness of adjoints for a perfect pairing), which is obvious from the definitions. 
\end{proof}
\subsection{The action of Hochschild cohomology and the semiregularity map}
Let $X$ be a smooth and proper scheme over a field $k$. The Hochschild-theoretic semiregularity map (Definition \ref{def_hocshchild_semiregular}) defines a map
\[
\sigma^{\HH_*}_{\catO_\Delta}\colon \HH^*(X) \to \HH_{-*}(X \times X)
\]
(see Definition \ref{def_hocshchild_semiregular}). In this section, we will construct a K\"unneth isomorphism
\[
K \colon \HH_{-*}(X \times X) \to \bigoplus_i \HH_{-i}(X) \otimes \HH_{i - *}(X)
\]
and an isomorphism $\psi \colon \HH_{-i}(X) \simeq \HH_{i}(X)^\vee$. Moreover, we will show that the composition
\begin{align*}
\HH^*(X) &\xrightarrow{\sigma^{\HH_*}_{\catO_\Delta}} \HH_{-*}(X \times X) \\ 
&\xrightarrow{(\psi \otimes \id) \circ K} \bigoplus_i \HH_{i}(X)^\vee \otimes \HH_{i - *}(X) \\
&\simeq \Hom_k(\HH_{i}(X), \HH_{i - *}(X))
\end{align*}
can be identified with the natural action of Hochschild cohomology on Hochschild homology (see Corollary \ref{corol_actions} below for a precise statement).  Again, for some reason this seems to be difficult to prove directly, but by passing to the Serre duals it is possible to establish a comparison result. 

We start off by introducing the natural action.
\begin{definition}
Let $k$ be a field, and let $X$ be a smooth and proper scheme over $k$. We define the \emph{action map}
\begin{align*}
\HH^*(X) &\xrightarrow{a} \bigoplus_i \Hom(\HH_i(X), \HH_{i - *}(X))  \\
f &\mapsto ((x_i)_i \mapsto (f \circ x_i)_i)
\end{align*}
where $x_i \in \HH_i(X) = \Ext^{-i}(\Delta_! \catO_X, \Delta_* \catO_X)$. 
\end{definition}
\begin{remark}
Since for any two finite dimensional vector spaces $V, W$ we have a canonical isomorphism $\Hom(V, W) \cong V^\vee \otimes W$, we may also think of the action as a map
\[
\HH^*(X) \to \bigoplus_i \HH_i(X)^\vee \otimes \HH_{i - *}(X)
\]
which we will also denote with $a$. 
\end{remark}
\begin{lemma}\label{lem_dual_action}
Let $k$ be a field, and let $X$ be a smooth and proper scheme over $k$. The Serre left adjoint of the action map is given by the map
\begin{align*}
a^\dagger \colon \bigoplus_i \HH_i(X) \otimes \HH_{i - *}(X)^\vee  &\to \Ext^{-*}_{X \times X}(\Delta_* \catO_X, \Delta_* \catO_X \otimes S_{X \times X}) \\
\alpha_i \otimes \beta_i &\mapsto (S_{X \times X}(\alpha_i) \circ \beta_i)
\end{align*}
for $\alpha_i \in \Ext^{-i}_{X \times X}(\Delta_! \catO_X, \Delta_* \catO_X)$ and $\beta_i \in \Ext^{i - *}_{X \times X}(\Delta_* \catO_X, \Delta_* S_X)$. 
\end{lemma}
\begin{proof}
Fix 
\begin{align*}
\alpha_i \in \HH_i(X) = \Ext^{-i}_{X \times X}(\Delta_! \catO_X, \Delta_* \catO_X) \\
\beta_i \in \HH_{i - *}(X)^\vee = \Ext^{i - *}_{X \times X}(\Delta_* \catO_X, \Delta_* S_X)
\end{align*}
Let $\{v_j\}_{j \in J}$ be a basis of $\HH_i(X)$, and let $\{v_j^\vee\}_{j \in J}$ be the dual basis of $\Ext^i_{X \times X}(\Delta_* \catO_X, \Delta_* S_X)$. Then by definition the action map is given by 
\[
a(f) = \sum_{j \in J} v_j^\vee \otimes (f \circ v_j)
\]
It follows that
\begin{align*}
\langle \alpha_i \otimes \beta_i, a(f) \rangle &= \sum_{j \in J} \langle \alpha_i \otimes \beta_i, v_j^\vee \otimes (f \circ v_j) \rangle \\
&=  \sum_{j \in J}  \langle \alpha_i, v_j^\vee \rangle \cdot \langle \beta_i, f \circ v_j \rangle \\
&= \langle \beta_i, f \circ \(\sum_{j \in J} \langle v_j^\vee, \alpha_i \rangle \cdot v_j\) \rangle \\
&= \langle \beta_i, f \circ \alpha_i \rangle 
\end{align*}
By \cite[Lemma 2.2]{caldararu1} it follows that
\begin{align*}
\langle \alpha_i \otimes \beta_i, a(f) \rangle &= \Tr_{X \times X}(\beta_i \circ f \circ \alpha_i) \\
&= \Tr_{X \times X}(S_{X \times X}(\alpha_i) \circ \beta_i \circ f) \\
&= \langle S_{X \times X}(\alpha_i) \circ \beta_i, f \rangle
\end{align*}
as required. 
\end{proof}
Note that if $X$ and $Y$ are two smooth and  proper schemes over a field $k$, one has a K\"unneth isomorphism 
\[
K_{X, Y} \colon \bigoplus_{i} \HH_i(X) \otimes_k \HH_{* - i}(Y) \xrightarrow{\sim} \HH_{*}(X \times Y)
\]
given by sending $\alpha \in \Ext^{-i}_{X \times X}(\Delta_! \catO_X, \Delta_* \catO_X)$ and $\beta \in \Ext^{i - *}(\Delta_! \catO_Y, \Delta_* \catO_Y)$ to the element
\[
\pi_{13}^*(\alpha) \otimes \pi_{24}^*(\beta) \in \Ext^{-*}_{X \times Y \times X \times Y}(\Delta_! \catO_{X \times Y}, \Delta_* \catO_{X \times Y})
\]
(note $\Delta_* \catO_{X \times Y} = \pi_{13}^* \Delta_* \catO_X \otimes \pi_{24}^* \Delta_* \catO_Y$, and similarly for $\Delta_!$). 

\begin{definition}
Let $X$ be a smooth and proper scheme over a field $k$. Define the map
\begin{align*} 
\varphi \colon \Ext_{X \times X}^i(\Delta_* \catO_X, \Delta_* S_X) &\to \Ext_{X \times X}^{-i}(\Delta_! \catO_X, \Delta_* O_X) \\
\eta &\mapsto (\sigma_X)_*(\eta \otimes \pi_2^* S_X^{-1})
\end{align*}
where $\sigma_X$ is the map $X \times X \to X \times X$ swapping the two factors
\end{definition}

\begin{prop}\label{prop_compare_action}
Let $X$ be a smooth and proper scheme over a field $k$. Then the diagram
\[
\begin{tikzcd}
\Ext^{-*}_{X \times X}(\Delta_* \catO_X, \Delta_* \catO_X \otimes S_{X \times X}) & \lar{\ev_{\catO_{\Delta}}} \Ext^{-*}_{X \times X \times X \times X}(\Delta_* \catO_{X \times X}, \Delta_* S_{X \times X}) \\
\uar{a^\dagger} 
\bigoplus_i \HH_i(X) \otimes \HH_{i - *}(X)^\vee & \bigoplus_i \lar{\varphi \otimes \id }  \uar{(K_{X, X}^{-1})^\vee}\HH_{-i}(X)^\vee \otimes \HH_{i - *}(X)^\vee
\end{tikzcd}
\]
commutes.
\end{prop}
\begin{proof}
Let 
\begin{align*}
\alpha \in \HH_i(X)^\vee = \Ext^i_{X \times X}(\Delta_* \catO_X, \Delta_* S_X) \\
\beta \in \HH_{* - i}(X)^\vee = \Ext^{* - i}_{X \times X}(\Delta_* \catO_X, \Delta_* S_X) 
\end{align*}
and consider $\ev_{\catO_\Delta} \circ (K_{X, X}^{-1})^\vee (\alpha \otimes \beta)$. Then 
\begin{align*}
\ev_{\catO_\Delta} \circ (K_{X, X}^{-1})^\vee (\alpha \otimes \beta) &= \pi_{34*}(\pi_{12}^*(\catO_\Delta) \otimes \pi_{13}^* \alpha \otimes \pi_{24}^* \beta ) \\
&= \pi_{34*}((\Delta_X \times \id_X \times \id_X)_*( \catO_{X \times X \times X}) \otimes \pi_{13}^* \alpha \otimes \pi_{24}^* \beta) \\
& =\pi_{34*}(\Delta_X \times \id_X \times \id_X)_*(\pi_{12}^* \alpha \otimes \pi_{13}^* \beta) \\
&= \pi_{23*}(\pi_{12}^* \alpha \otimes \pi_{13}^* \beta)
\end{align*}
where in the second equality we used 
\[
\pi_{12}^*(\catO_\Delta) = (\Delta_X \times \id_X \times \id_X)_*( \catO_{X \times X \times X})
\]
(which follows by base change), and the rest follows from the projection formula. Now let
\[
\tilde{\Delta} \colon X \times X \to X \times X \times X
\]
be the map informally given by $(x, y) \to (x, y, x)$, so that we have the equality $\pi_{13}^* \Delta_* = \tilde{\Delta}_* \pi_1^*$ induced by the pullback square
\[
\begin{tikzcd}
X \times X \dar{\tilde{\Delta}}\rar{\pi_1}&  X \dar{\Delta} \\
X \times X \times X \rar{\pi_{13}} & X \times X
\end{tikzcd}
\]
Similarly, we have $\pi_{12}^* \Delta_* = (\Delta_X \times \id_X)_* \pi_1^*$. Thus 
\begin{align*}
 \pi_{23*}(\pi_{12}^* \alpha \otimes \pi_{13}^* \beta) &= \pi_{23*}\((\pi_{12}^* \alpha \circ \id_{(\Delta_X \times \id_X)_* \catO_{X \times X}}) \otimes (\id_{\tilde{\Delta}_* \pi_1^* S_X} \circ \pi_{13}^* \beta)\) \\
 &= \pi_{23*}(\pi_{12}^* \alpha \otimes \tilde{\Delta}_* \pi_1^* S_X) \circ \pi_{23*}((\Delta_X \times \id_X)_* \catO_{X \times X} \otimes \pi_{13}^* \beta) \\
&= (\pi_{23} \circ \tilde{\Delta})_*(\alpha \otimes \pi_1^* S_X) \circ (\pi_{23} \circ \Delta_X \times \id_X)_*(\beta) \\
&= \sigma_{X*}(\alpha \otimes \pi_1^* S_X) \circ \beta
\end{align*} 
On the other hand one has 
\[
(\alpha^\dagger \circ (\varphi \otimes \id))(\alpha \otimes \beta) = S_{X \times X}(\sigma_{X*}(\alpha \otimes \pi_2^* S_X^{-1})) \circ \beta = \sigma_{X*}(\alpha \otimes \pi_1^* S_X) \circ \beta
\]
which completes the proof.
\end{proof}
\begin{corollary}\label{corol_actions}
Let $k$ be a field, and let $X$ be a smooth and proper scheme over $k$. Then the diagram
\[
\begin{tikzcd}[column sep = huge]
\HH^*(X) \dar{a} \rar{\sigma^{\HH_*}_{\catO_\Delta}} & \HH_{-*}(X \times X) \dar{K_{X, X}^{-1}} \\
\bigoplus_i \HH_i(X)^\vee \otimes \HH_{i - *}(X) & \lar{(\varphi^\vee)^{-1} \otimes \id} \bigoplus_i \HH_{-i}(X) \otimes \HH_{i - *}(X)
\end{tikzcd}
\]
commutes.
\end{corollary}
\begin{proof}
This follows directly from Proposition \ref{prop_compare_action} after identifying $a$ with the Serre right adjoint of $a^\dagger$ and $\sigma^{\HH_*}_{\catO_\Delta}$ with the Serre right adjoint of $\ev_{\catO_\Delta}$ using Lemma \ref{lem_dual_action} and Lemma \ref{lem_semiregular_adjoint}.
\end{proof}

\section{Deforming Fourier--Mukai transforms between Calabi--Yau varieties}\label{sec_deforming_transforms}
In this section we restrict our attention to so-called Calabi--Yau varieties.

\begin{definition}
Let $X$ be a smooth and proper scheme over a field $k$. We say that $X$ is \emph{Calabi--Yau} if $X$ is equidimensional and $\Omega_{X / k}^{\dim(X)} \cong \catO_X$. 
\end{definition}
We combine everything to prove Theorem \ref{thm_zero} and Theorem \ref{thm_mixed}. In Section \ref{sec_injective} we show the semiregularity map is injective for Calabi--Yau varieties. Then in the following sections we prove our main results. 

\subsection{Injectivity of the semiregularity map}
\label{sec_injective}
Recall (Definition \ref{def_semiregularity}) that for any scheme $X$ and $\mathcal{E} \in \Perf(X)$, we have a semiregularity map 
\[
\sigma_{X} \colon \Ext^i_X(\mathcal{E}, \mathcal{E}) \to \bigoplus_p \coh^{p + i}(\LL^p_{X / k})
\]
By the results in Section \ref{sec_relate}, the $p$-th component $\sigma_{X, p}$ maps obstructions to deformations of $\mathcal{E}$ to obstruction classes to $\ch_p(\mathcal{E})$ staying within the $p$-th part of the Hodge filtration. 

In this section we show that if $X = Y \times Z$ where $Y$ is Calabi--Yau and $\mathcal{E}$ is the kernel of a fully faithful Fourier--Mukai transform $\derD(Y) \to \derD(Z)$, then the total map $\sigma_X$ map is always injective. Informally speaking, this says that one can read of whether or not $\mathcal{E}$ will deform by checking whether or not $\ch_p(\mathcal{E})$ remains within the $p$-th part of the Hodge filtration for all $p \geq 0$. 

Our strategy is essentially due to \cite{toda}: The fact that $\mathcal{E}$ is fully faithful implies that the transform $\mathcal{E} \star (-) \colon \derD(X \times X) \to \derD(X \times Y)$ has a left inverse, which will allow us to reduce to the case where $X = Y$ and $\mathcal{E} = \Delta_* \catO_X$. Thus the following lemma is all we will need.
\begin{lemma}\label{lem_diagonal_inj}
Let $k$ be a field and let $X$ be a smooth and proper scheme over $k$. If $X$ is Calabi--Yau, then the semiregularity map
\[
\sigma^{\HH_i}_{\catO_{\Delta_X}} \colon \HH^i(X) \to \HH_{-i}(X \times X)
\]
is injective for all $i$. 
\end{lemma}
\begin{proof}
By Corollary \ref{corol_actions}, it suffices to show the map 
\begin{equation}\label{eq_need_inj}
a \colon \bigoplus_j \HH^j(X) \to \bigoplus_{i, j} \Hom(\HH_i(X), \HH_{i - j}(X))
\end{equation}
is injective. Since $X$ is Calabi--Yau, we have $S_X = \catO_X[d]$ where $d = \dim(X)$. Thus
\[
\HH_i(X) = \Ext^{-i}_{X \times X}(\catO_{\Delta_X}[d], \catO_{\Delta_X}) = \Ext^{d - i}_{X \times X}(\catO_{\Delta_X}, \catO_{\Delta_X}) = \HH^{d - i}(X)
\]
Under this identification, the map
\[
a \colon \HH^j(X) \to \Hom(\HH_d(X), \HH_{d - j}(X))
\]
corresponds to the map
\[
\HH^j(X) \to \Hom(\HH^0(X), \HH^j(X))
\]
given by composition (i.e. the natural ring multiplication on Hochschild cohomology). But this last map is clearly injective (since we can evaluate at $\id_{\catO_{\Delta_X}}$). We conclude that (\ref{eq_need_inj}) is injective as desired.
\end{proof}
The following observation is probably well known and used in the footnote in \cite[page 19, footnote 6]{addingtonthomas14} to construct a left inverse to $\mathcal{E} \star (-)$, however we could not find it anywhere in the literature so we give a direct proof.
\begin{lemma}\label{lem_right_adjoint_fourier}
Let $k$ be a field, and let $X$ and $Y$ be smooth and proper schemes over $k$. Let $\mathcal{E} \in \derD(X \times Y)$ be the kernel of a Fourier--Mukai transform $$\Phi_{\mathcal{E}} \colon \derD(X) \to \derD(Y)$$ Let 
\[
\mathcal{E}_R = \mathcal{E}^\vee \otimes \pi_X^* S_X
\]
be the kernel of the right adjoint. Then the convolution functor  $$\mathcal{E}_R \star (-) \colon \derD(X \times Y) \to \derD(X \times X)$$ (see (\ref{faa})) is right adjoint to $$\mathcal{E} \star (-) \colon \derD(X \times X) \to \derD(X \times Y)$$
\end{lemma}
\begin{proof}
Denote with $\tilde{\Delta} \colon X \times X \to X \times X \times X$ the map sending $(x, y) \mapsto (x, y, x)$. Let $A \in \derD(X \times X)$. Then 
\begin{align*}
\mathcal{E} \star A :&= \pi_{13*}(\pi_{12}^* A \otimes \pi_{23}^* \mathcal{E}) \\
&= \pi_{34*}((\tilde{\Delta} \times \id_Y)_*(\pi_{12}^* A \otimes \pi_{23}^* \mathcal{E}) ) \\
&=  \pi_{34*}((\tilde{\Delta} \times \id_Y)_*(\catO_{X \times X \times Y}) \otimes \pi_{12}^* A \otimes \pi_{24}^* \mathcal{E}) \\
&=  \pi_{34*}(\pi_{12}^* A \otimes \(\pi_{24}^* \mathcal{E} \otimes \pi_{13}^* \Delta_* \catO_X \))
\end{align*}
Thus $\pi_{24}^* \mathcal{E} \otimes \pi_{13}^* \Delta_* \catO_X \in \derD(X \times X \times X \times Y)$ is the kernel corresponding to $\mathcal{E} \star (-)$. It follows that the kernel corresponding to the right adjoint of $\mathcal{E} \star (-)$ is given by
\begin{align*}
\(\pi_{24}^* \mathcal{E} \otimes \pi_{13}^* \Delta_* \catO_X \)^\vee \otimes \pi_{12}^* S_{X \times X} &= (\pi_{24}^* \mathcal{E}^\vee \pi_2^* \otimes S_X) \otimes \pi_{13}^* ((\Delta_* \catO_X)^\vee  \otimes \pi_1^* S_X) \\
&= (\pi_{24}^*(\mathcal{E}^\vee \otimes\pi_1^* S_X)) \otimes \pi_{13}^* (\Delta_* \catO_X) \\
&= (\pi_{24}^*(\mathcal{E}_R) \otimes \pi_{13}^* (\Delta_* \catO_X) 
\end{align*}
where in the second equality we have used $(\Delta_* \catO_X)^\vee  \otimes \pi_1^* S_X = \Delta_* \catO_X$ (this follows from the fact that $\catO_{\Delta_X}$ is the kernel of the identity functor, hence the kernel of its right adjoint is equal to itself). 

By the exact same argument as above, this last expression is the kernel corresponding to $\mathcal{E}_R \star (-)$, which completes the proof.
\end{proof}
We now have everything we need to conclude the injectivity we need.
\begin{theorem}\label{thm_semireg_inj}
Let $k$ be a field, and let $X$ and $Y$ be smooth and proper schemes over $k$. Let $\mathcal{E} \in \derD(X \times Y)$ be the kernel of a Fourier--Mukai transform 
\[
\Phi_{\mathcal{E}} \colon \derD(X) \to \derD(Y)
\]
If $\Phi_{\mathcal{E}}$ is fully faithful and $X$ is Calabi--Yau, then the semiregularity map
\[
\sigma^{\HH_i}_{\mathcal{E}} \colon \Ext^i_{X \times Y}(\mathcal{E}, \mathcal{E}) \to \HH_{-i}(X \times Y)
\]
is injective for all $i$.
\end{theorem}
\begin{proof}
Let $\mathcal{E}_R := \mathcal{E}^\vee \otimes \pi_X^* S_X \in \derD(Y \times X)$ be the kernel of the right adjoint
\[
\Phi_{\mathcal{E}_R} \colon \derD(Y) \to \derD(X) 
\]
to $\Phi_{\mathcal{E}}$.  Since $\Phi_{\mathcal{E}}$ is fully faithful, its right adjoint $\Phi_{\mathcal{E}_R}$ is a left inverse by abstract nonsense. We thus have $\mathcal{E}_R \star \mathcal{E} \simeq \catO_{\Delta_X}$. Write $\widehat{\mathcal{E}} \in \derD(X \times X \times X \times Y)$ and $\widehat{\mathcal{E}}_R  \in \derD(X \times Y \times X \times X)$ for the kernels of the functors $\mathcal{E} \star (-)$ and $\mathcal{E}_R \star (-)$ respectively.

By Proposition \ref{prop_functorial_semiregular} we have a commutative diagram
\[
\begin{tikzcd}
\Ext^*_{X \times X}(\catO_{\Delta_X}, \catO_{\Delta_X}) \dar {\sigma_{\catO_{\Delta_X}}^{\HH_{*}}} \rar{\Phi_{\widehat{\mathcal{E}}}}& \Ext^*_{X \times Y}(\mathcal{E}, \mathcal{E}) \dar{\sigma_{\mathcal{E}}^{\HH_{*}}} \\
\HH_{-*}(X \times X) \rar{\Phi^{\HH_{-*}}_{\widehat{\mathcal{E}}}} & \HH_{-*}(X \times Y)
\end{tikzcd}
\]
Since $\mathcal{E}_R \star \mathcal{E}  \simeq \catO_{\Delta_X}$, it follows from Lemma \ref{lem_right_adjoint_fourier} that the map
\[
\Ext^*_{X \times X}(\catO_{\Delta_X}, \catO_{\Delta_X}) \xrightarrow{\Phi_{\widehat{\mathcal{E}}}} \Ext^*_{X \times Y}(\mathcal{E}, \mathcal{E})
\]
is an isomorphism. Since $\Phi_{\widehat{\mathcal{E}}_R}$ is a right inverse to $\Phi_{\widehat{\mathcal{E}}}$, by Proposition \ref{prop_functorial_semiregular} the map 
\[
\Phi^{\HH_*}_{\widehat{\mathcal{E}}} \colon \HH_*(X \times X) \to \HH_*(X \times Y)
\]
admits a left inverse (given by $\Phi^{\HH_*}_{\widehat{\mathcal{E}}_R}$), thus in particular is injective. Finally by Lemma \ref{lem_diagonal_inj} the map $\sigma^{\HH_*}_{\catO_{\Delta_X}}$ is injective, the result follows.
\end{proof}
Finally, we obtain the following result for Hodge cohomology by transferring the previous result along the Hochschild--Kostant--Rosenberg isomorphism.
\begin{corollary}\label{corol_semireg_inj}
Let $k$ be a field and let $X$ and $Y$ be smooth and proper equidimensional schemes over $k$. Suppose that $d = \dim(X \times Y)$ is such that $d!$ is invertible in $k$.  Let $\mathcal{E} \in \derD(X \times Y)$ be the kernel of a Fourier--Mukai transform 
\[
\Phi_{\mathcal{E}} \colon \derD(X) \to \derD(Y)
\]
If $\Phi_{\mathcal{E}}$ is fully faithful and $X$ is Calabi--Yau, then for all $j$ the (total) semiregularity map
\[
\sigma_{X} \colon \Ext^j_{X \times Y}(\mathcal{E}, \mathcal{E}) \to \bigoplus_{i = 0}^{\dim(X) + \dim(Y)} \coh^{i + j}(X \times Y, \Omega^i_{X \times Y / k})
\]
is injective.
\end{corollary}
\begin{proof}
Combine Proposition \ref{prop_semireg_compare} and Theorem \ref{thm_semireg_inj}.
\end{proof}
\subsection{Deformations in the characteristic zero case}
We combine all the previous results in the characteristic zero case to prove Theorem \ref{thm_zero} at the end of this section. We isolate the statement of the inductive step in the following proposition.
\begin{prop}\label{prop_induction_char0}
Let $k$ be a field with $\QQ \subseteq k$, and let $R' \to R$ be a surjective ring map of local Artinian $k$-algebras with kernel $I$ such that $I^2 = 0$ and $\mathfrak{m}_{R'} \cdot I = 0$. Let $X', Y'$ be smooth and projective schemes over $R'$, and let 
\begin{align*}
X &:= X' \times_{\Spec(R')} \Spec(R) \\
X_0 &:= X'  \times_{\Spec(R')} \Spec(k)
\end{align*}
and similarly for $Y, Y_0$. Let $\mathcal{E} \in \Perf(X \times Y)$ be the kernel of a fully faithful Fourier--Mukai transform $\Phi_{\mathcal{E}} \colon \Perf(X) \to \Perf(Y)$, and assume $\mathcal{E}$ can be represented by a bounded complex of vector bundles. Let $\mathcal{E}_0 := \mathcal{E} \rvert_{X_0 \times Y_0}$. If $X_0$ is Calabi--Yau, then the following are equivalent. 
\begin{enumerate}
\item There exists a kernel $\mathcal{E}' \in \Perf(X' \times Y')$ of a fully faithful Fourier--Mukai transform $\Phi_{\mathcal{E}'} \colon \Perf(X') \to \Perf(Y')$ such that $\mathcal{E}' \rvert_{X \times Y} \simeq \mathcal{E}$.
 \item The image of $\ch_i(\mathcal{E}_0) \otimes 1$ under the stratifying map (Definition \ref{aal})
\[
\varphi_{\widehat{\dR},X' \times Y'} \colon \coh^{2i}(\widehat{\dR}_{X_0  \times Y_0 / k}) \otimes_k R'  \to \coh^{2i}(\widehat{\dR}_{X' \times Y'  / R'})
\]
lands in $\Fil^i \coh^{2i} ( \widehat{\dR}_{X' \times Y'  / R'})$ for all $i \geq 0$. 
\end{enumerate} 
If these hold then $\Phi_{\mathcal{E}'}$ is an equivalence if and only if $\Phi_{\mathcal{E}}$ is. 
\end{prop}
\begin{proof}
Note that by definition of the obstruction class, we have 
\[
\varphi_{\widehat{\dR},X' \times Y'}(\ch_i(\mathcal{E}_0) \otimes 1) \in \Fil^i \coh^{2i} (\widehat{\dR}_{X' \times Y'  / R'})
\]
if and only if $\ob^{\widehat{\dR}}_{X' \times Y' / R'}(\ch_i(\mathcal{E}_0)) = 0$. Since $\mathfrak{m}_{R'} \cdot I = 0$, the semiregularity map
\[
\sigma_{X \times Y, i} \colon \Ext_{X \times Y}^2(\mathcal{E}, \mathcal{E} \otimes \mathcal{I}) \to \bigoplus_i  \coh^{i + 2}(\LL^i_{X \times Y / R}\otimes \mathcal{I})
\]
is just the map
\[
\sigma_{X_0 \times Y_0} \colon \Ext_{X_0 \times Y_0}^2(X_0 \times Y_0, \mathcal{E}_0, \mathcal{E}_0) \otimes_k I \to \bigoplus_i  \coh^{i + 2}(X_0 \times Y_0, \Omega^i_{X_0 \times Y_0 / k}) \otimes_k I
\]
hence injective by Corollary \ref{corol_semireg_inj}.  Thus $\ob^{\widehat{\dR}}_{X' \times Y' / R'}(\ch_i(\mathcal{E}_0)) = 0$ for all $p$ if and only if $\ob(\mathcal{E}, X \times Y, X' \times Y') = 0$ by Theorem \ref{thm_compare_obs_char0}. 

By Proposition \ref{prop_obstruction_class} it follows that there exists $\mathcal{E'} \in \Perf(X' \times Y')$ deforming $\mathcal{E}$ if and only if we have $\ob^{\widehat{\dR}}_{X' \times Y' / R'}(\ch_i(\mathcal{E}_0)) = 0$ for all $i$. By \cite[Proposition 2.15]{rumasa} we see that $\Phi_{\mathcal{E}'}$ is always fully faithful, and an equivalence if and only if $\Phi_{\mathcal{E}}$ is an equivalence, which proves the result. 
\end{proof}

\begin{proof}[Proof of Theorem \ref{thm_zero}]
By induction on the size of $A$, using Proposition \ref{prop_induction_char0}.
\end{proof}
\subsection{Deformations in mixed characteristic}
In this section we combine all the previous results in the $p$-adic case, to prove Theorem \ref{thm_mixed}. The following proposition gives the essential ingredient, allowing us to lift the transform to a slightly smaller extension. The proposition is very general, we urge the reader to keep in mind the following example: $k$ is a field of characteristic $p > 2$, $W = W(k)$ is the ring of Witt vectors of $k$, $R' = W / p^{m + 1}$, $R = W / p^{m}$ and $I = (p^m) \subseteq W / p^{m + 1}$.
\begin{prop}\label{prop_mixed_def}
Let $W$ be a ring. Let $k$ be a field with a map $W \to k$. Let
\[
(R' \to k, \gamma') \to (R \to k, \gamma)
\]
be a morphism of divided power $W$-algebras such that $R' \to R$ is a surjection with kernel $I$, such that $\gamma'_n(x) = 0$ for all $x \in I$ and all $n \geq 2$, and such that $\mathfrak{m}_{R'} \cdot I = 0$. Let $X', Y'$ be smooth and projective schemes over $R'$, and let 
\begin{align*}
X &:= X' \times_{\Spec(R')} \Spec(R) \\
X_0 &:= X'  \times_{\Spec(R')} \Spec(k)
\end{align*}
and similarly for $Y, Y_0$. Let $\mathcal{E} \in \Perf(X \times Y)$ be the kernel of a fully-faithful Fourier--Mukai transform $\Phi_{\mathcal{E}} \colon \Perf(X) \to \Perf(Y)$. Let $\mathcal{E}_0 := \mathcal{E} \rvert_{X_0 \times Y_0}$. If $X_0$ is Calabi--Yau, $Y_0$ is equidimensional and $d := \dim(X_0) + \dim(Y_0)$ is such that $d!$ is invertible in $W$, then the following are equivalent. 

\begin{enumerate}
\item There exists a kernel $\mathcal{E}' \in \Perf(X' \times Y')$ of a fully faithful Fourier--Mukai transform $\Phi_{\mathcal{E}'} \colon \Perf(X') \to \Perf(Y')$ such that $\mathcal{E}' \rvert_{X \times Y} \simeq \mathcal{E}$.
\item The image of $\ch_i(\mathcal{E}_0) \otimes 1$ under the stratifying map (Definition \ref{abk})
\[
\varphi_{\Crys,X' \times Y'} \colon \coh^{2i} ({\dR}_{X_0  \times Y_0 / k}) \otimes_k R'  \to \coh^{2i} ({\dR}_{X' \times Y'  / R'})
\]
lands in $\Fil^i \coh^{2i} ({\dR}_{X' \times Y'  / R'})$ for all $i \geq 0$. 
\end{enumerate}
If either of the equivalent conditions holds then $\Phi_{\mathcal{E}'}$ is an equivalence if and only if $\Phi_{\mathcal{E}}$ is. 
\end{prop}
\begin{proof}
Similar to the proof of Proposition \ref{prop_induction_char0}, we only give details where the proof differs. Note that by definition of the obstruction class, we have 
\[
\varphi_{\Crys,X' \times Y'}(\ch_i(\mathcal{E}_0) \otimes 1) \in \Fil^i \coh^{2i} ({\dR}_{X' \times Y'  / R'})
\]
if and only if $\ob^{\Crys}_{X' \times Y' / R'}(\ch_i(\mathcal{E}_0)) = 0$. Again, by Corollary \ref{corol_semireg_inj} the semiregularity map is injective, thus $\ob^{\Crys}_{X' \times Y' / R'}(\ch_i(\mathcal{E}_0)) = 0$ for all $i$ if and only if one has $\ob(\mathcal{E}, X \times Y, X' \times Y') = 0$ by Theorem \ref{thm_compare_obs_p}. We may again conclude by Proposition \ref{prop_obstruction_class} and \cite[Proposition 2.15]{rumasa}.
\end{proof}
\begin{definition}\label{gae}
Let $(A, I, \gamma)$ be a divided power ring, and let $n \in \NN$. Define $\gamma_n^0(I) := I$ and inductively define the ideals
\[
\gamma_n^k(I) := \langle \gamma_n(x) \mid x \in \gamma_n^{k - 1}(I) \rangle
\]
for $k \geq 1$. We say that $\gamma_n$ \emph{acts nilpotently on I} if $\gamma_n^k(I) = 0$ for some $k \in \NN$. 
\end{definition}
We now wish to show that this condition on $\gamma_p$ implies that we can find a suitable sequence of ideals to apply Proposition \ref{prop_mixed_def} to.
\begin{lemma}\label{gaa}
Let $p$ be a prime, let $A$ be a $\ZZ_{(p)}$-algebra and let $I \subseteq A$ be an ideal with a divided power structure $\gamma$. Then any $x \in I^{[2]} \setminus I^2$ can be written as 
\[
x = a + c_1 \cdot \gamma_p(b_1) + \dots + c_m \gamma_p(b_m)
\]
for some $a \in I^2$, $m \in \NN$, $c_1, \dots, c_m \in A$ and $b_1, \dots, b_m \in I \setminus I^{2}$.
\end{lemma}
\begin{proof}
Since $x \in I^{[2]}$, there exists $m \in \NN$, $a \in I^2$ and elements $c_1, \dots, c_m \in A$, elements $b_1, \dots, b_m \in I$ and $n_1, \dots, n_m \in \NN_{\geq 2}$ such that
\[
x = a + \sum_{i = 1}^m c_i \gamma_{n_i}(b_i)
\]
Choose such a representation such that 
\[
N = \sum_{i = 1}^{m} n_i
\]
is minimal. Note $N \geq 1$ since $x \not \in I^2$. 

Note that if some $n_i$ were not divisible by $p$, we could write $n_i = pk + \ell$ with $\ell \in \{1, \dots, p - 1\}$. Hence 
\[
C_{n_i} := \frac{(pk + \ell)!}{(pk)! \ell!}
\]
is not divisible by $p$, and thus invertible in $A$. It follows that 
\[
\gamma_{n_i}(x) = C_{n_i}^{-1} \cdot \gamma_{pk}(x)\cdot \gamma_{\ell}(x)
\]
lies in $I^2$, contradicting minimality of $N$ (note $k = 0$ implies $\ell \geq 2$). Thus all $n_i$ are divisible by $p$. 

Next, suppose there exists $i$ such that $n_i = pb$ for some $b \geq 2$. A calculation with valuations shows that the integer 
\[
C_{p ,b} := \frac{(pb)!}{(p!)^b b!}
\]
is not divisible by $p$, and hence a unit in $A$. It follows that $\gamma_{n_i}(x) = C_{p, b}^{-1} \cdot \gamma_b(\gamma_p(x))$. Since $b \geq 2$, this again contradicts minimality of $N$. We conclude that $n_i = p$ for all $i$. 

Finally if $b_i \in I^2$ for some $i$, then $\gamma_p(b_i) \in I^2$, which again contradicts minimality of $N$. The result follows.
\end{proof}

\begin{corollary}\label{gad}
Let $(A, I, \gamma)$ be a divided power ring, and let $p$ be a prime number. If $A$ is a local Artinian $\ZZ_{(p)}$-algebra, $0 \subsetneq I \subseteq \mathfrak{m}_A$ and $\gamma_p$ acts nilpotently on $I$, then $I^{[2]} \subsetneq I$. 
\end{corollary}
\begin{proof}
Since $A$ is Artinian we see that $\mathfrak{m}_A$ is nilpotent, therefore $I$ is nilpotent. In particular $I^2 \neq I$ since $I \neq 0$. Thus after replacing $A$ by $A / I^2$ one still has $I \neq 0$. We may thus assume without loss of generality that $I^2 = 0$. 

Suppose that $I^{[2]} = I$, we show this leads to a contradiction. Since $I \neq 0$ there exists $x \in I$ such that $x \neq 0$. Since we assumed $I^{[2]} = I$, we have $x \in I^{[2]} \setminus I$, so that by Lemma \ref{gaa} we may write
\begin{equation} \label{gab}
x = c_1 \cdot \gamma_p(b_1) + \dots + c_m \gamma_p(b_m)
\end{equation}
for some $m \in \NN$, $c_i \in A$ and $b_i \in I$ nonzero. Hence $x \in \gamma_p^1(I)$. Again using that $I^{[2]} = I$, we may represent each $b_i$ as 
\begin{equation}\label{gac}
b_i = c_{i1} \gamma_p(b_{i1}) + \dots + c_{im} \gamma_p(b_{im_i})
\end{equation}
for some $m_i \in \NN$, $c_{ij} \in A$ and $b_{ij} \in I$ nonzero. Combining (\ref{gab}) and (\ref{gac}) it follows that $x \in \gamma_p^2(I)$. Continuing like this, we may show that $x \in \gamma_p^k(I)$ for all $k \geq 0$. Thus $x = 0$ since $\gamma_p$ acts nilpotently on $I$, which is a contradiction. 
\end{proof}
\begin{prop}\label{gag}
Let $(A, I, \gamma)$ be a divided power ring, and let $p$ be a prime number. If $A$ is a local Artinian $\ZZ_{(p)}$-algebra and $\gamma_p$ acts nilpotently on $\mathfrak{m}_A$, then there exists a finite chain of ideals
\[
\mathfrak{m}_A = I_1 \supsetneq I_2 \supsetneq \dots \supsetneq I_k = 0
\]
such that $I_i^{[2]} \subseteq I_{i + 1}$. 
\end{prop}
\begin{proof}
One simply defines $I_{1} := \mathfrak{m}_A$ and $I_{i + 1} := I_i^{[2]}$. Then $I_{i + 1} \subsetneq I_i$ as long as $I_i \neq 0$ by Corollary \ref{gad}, and this sequence terminates since $A$ is Artinian.
\end{proof}
\begin{proof}[Proof of Theorem \ref{thm_mixed}]
Combine Proposition \ref{gag} and Proposition \ref{prop_mixed_def}.
\end{proof}
\begin{proof}[Proof of Corollary \ref{corol_mixed}]
For $n > 0$, let $\mathcal{X}_n := \mathcal{X} \times_{\Spec(W)} \Spec(W_n)$ and similarly for $\mathcal{Y}_n$. By induction on $n$ and Theorem \ref{thm_mixed}, we may find a compatible system of lifts $\mathcal{E}_n \in \derD(\mathcal{X}_n \times \mathcal{Y}_n)$. Thus there exists a lift $\tilde{\mathcal{E}} \in \derD^b(\mathcal{X} \times \mathcal{Y})$ by \cite[Proposition 3.6.1]{lieblich05}. The induced transform is fully faithful (or an equivalence) by  \cite[Proposition 2.15]{rumasa}.
\end{proof}

\appendix
\section{Appendix}
\subsection{Compact projective generators}
Let $R$ be a ring and $k \in \NN$. The main goal of this section is to construct the compact projectively generated categories $\Fun(\Delta^k, \CAlg_R^\an)_\surj$. The case $k = 0$ and $k = 1$ were done by Mao, see \cite[Theorem 3.23]{mao}. We try to clean up the argument slightly in the process.

Throughout this section, $n$ can be any integer greater than or equal to $1$, or the symbol $\infty$. The key ingredient to finding sets of compact projective generators is the following result from Lurie.
\begin{prop}\label{prop_generators_from_adjunction}
Suppose given a pair of adjoint functors $
\begin{tikzcd} 
\mathcal{C}  \arrow[shift left=1ex]{r}{F} & \mathcal{D} \lar{G}
\end{tikzcd}$ between $n$-categories. Assume that:
\begin{enumerate}
\item The $n$-category $\mathcal{D}$ admits filtered colimits and geometric realizations, and $G$ preserves filtered colimits and geometric realizations. 
\item The $n$-category $\mathcal{C}$ is compact $n$-projectively generated.
\item The functor $G$ is conservative. 
\end{enumerate}
Then:
\begin{enumerate}
\item The $n$-category $\mathcal{D}$ is compact $n$-projectively generated.
\item An object $D \in \mathcal{D}$ is compact and $n$-projective if and only if there exists a compact $n$-projective object $C \in \mathcal{C}$ such that $D$ is a retract of $F(C)$. 
\item The functor $G$ preserves all sifted colimits. 
\item If $S$ is a set of compact $n$-projective generators for $\mathcal{C}$, then $F(S)$ is a set of compact $n$-projective generators for $\mathcal{D}$. 
\end{enumerate}
\end{prop}
\begin{proof}
See \cite[Corollary 4.7.3.18]{ha}. Note that 4. isn't stated but follows from the proof as well.
\end{proof}
The following lemma will be very useful along the way.
\begin{lemma}\label{fun_adjoint}
Suppose given a pair of adjoint functors $
\begin{tikzcd} 
\mathcal{C}  \arrow[shift left=1ex]{r}{F} & \mathcal{D} \lar{G}
\end{tikzcd}$ 
 between $n$-categories. Let $K$ be a simplicial set. Then there exists an induced pair
\[
\begin{tikzcd} 
\Fun(K, \mathcal{C})  \arrow[shift left=1ex]{r}{F \circ -} & \Fun(K, \mathcal{D}) \lar{G \circ -}
\end{tikzcd}
\]
of adjoint functors. 
\end{lemma}
\begin{proof}
Apply \cite[Proposition 5.2.2.8]{htt} twice.
\end{proof}
We now wish to study $\Fun(\Delta^k, \Ani(\mathcal{C}))$.
\begin{lemma}\label{lem_fun_proj_gen}
Let $\mathcal{C}$ be a compact $n$-projectively generated $n$-category. Then $\Fun(\Delta^k, \mathcal{C})$ is compact $n$-projectively generated. Moreover, if $S$ is a set of compact $n$-projective generators for $\mathcal{C}$, then the set
\begin{align*}
\{\ins_i(X) \mid i \in \{0, \dots, k\}, X \in S \} 
\end{align*}
where 
\[
\ins_i(X) := \underbrace{0 \to \dots \to 0}_{i \text{ times}} \to  \underbrace{X \to \dots \to X}_{(k- i + 1) \text{ times}}
\]
is a set of compact $n$-projective generators for $\Fun(\Delta^k, \mathcal{C})$. 
\end{lemma}
\begin{proof}
Since $\mathcal{C}$ is cocomplete, it follows that $\Fun(\Delta^k, \mathcal{C})$ is cocomplete by \cite[Corollary 5.1.2.3]{htt}. The forgetful functor
\begin{equation} \label{eq_forgetfull_adjoint}
\Fun(\Delta^k, \mathcal{C}) \to \Fun(\sk_0(\Delta^k), \mathcal{C})
\end{equation}
is conservative, commutes with colimits by \cite[Corollary 5.1.2.3]{htt} and admits a left adjoint explicitly given by
\[
\(X_0, \dots, X_k\) \mapsto \(X_0 \to X_0 \amalg X_1 \to \dots \to X_0 \amalg \dots \amalg X_k\)
\]
The result now follows by applying Proposition \ref{prop_generators_from_adjunction} to the set of compact $n$-projective generators for $\Fun(\sk_0(\Delta^k), \mathcal{C})$ given in \cite[Lemma 2.7]{mao}.
\end{proof}
\begin{definition}\label{def_gen_fun}
Let $\mathcal{C}$ be a compact $n$-projectively generated $n$-category and let $S$ be a set of compact $n$-projective generators.
We write $\Fun(\Delta^k, \mathcal{C})_\gen$ for the full subcategory of $\Fun(\Delta^k, \mathcal{C})$ spanned by coproducts of objects in the set $\{\ins_i(X) \mid X \in S, i \in \{0, \dots, k\}\}$. 
\end{definition}
\begin{corollary}\label{corol_ani_commutes_fun}
Let $\mathcal{C}$ be a compact $n$-projectively generated $n$-category. Then the map
\[
\Ani(\Fun(\Delta^k, \mathcal{C})) \xrightarrow{\sim} \Fun(\Delta^k, \Ani(\mathcal{C})) 
\]
is an equivalence of categories.
\end{corollary}
\begin{proof}
Let $S$ be a set of compact $n$-projective generators for $\mathcal{C}$, and denote with $j \colon \mathcal{C} \to \Ani(\mathcal{C})$ the Yoneda embedding. Then $j(S)$ is a set of compact projective generators for $\Ani(\mathcal{C})$, so by Lemma \ref{lem_fun_proj_gen} we see that 
\[
\{ \ins_i(j(X)) \mid i \in \{0, \dots, k \}, X \in S\}
\]
gives a set of compact projective generators for $\Fun(\Delta^k, \Ani(\mathcal{C}))$. Denote with $\mathcal{C}_0 \subseteq \Fun(\Delta^k, \mathcal{C})$ the full subcategory spanned by finite coproducts of objects in the set
\[
\{ \ins_i(X) \mid i \in \{0, \dots, k \}, X \in S\}
\]
Since $j$ is fully faithful one then has $\Ani(\mathcal{C}_0) = \Fun(\Delta^k, \Ani(\mathcal{C}))$. However, applying Lemma \ref{lem_fun_proj_gen} again we see that $\mathcal{C}_0$ is a set of compact $n$-projective generators for the $n$-category $\Fun(\Delta^k, \mathcal{C})$. The result follows.
\end{proof}
By Lemma \ref{fun_adjoint}, we see that the (pointwise) Yoneda embedding $$\Fun(\Delta^k, \mathcal{C}) \to \Fun(\Delta^k, \Ani(\mathcal{C}))$$ admits a left adjoint $$\pi_0 \colon \Fun(\Delta^k, \Ani(\mathcal{C})) \to \Fun(\Delta^k, \mathcal{C})$$ given by applying the left adjoint $\Ani(\mathcal{C}) \to \mathcal{C}$ pointwise. 

For the rest of this section, we restrict our attention to the $\infty$-categories $\derD(R)_{\geq 0}$ and $\CAlg_R^\an$, where $R$ is a discrete commutative ring.
\begin{lemma}\label{gens_mod}Let $R$ be a ring and $k \in \NN$. The $\infty$-category $\Fun(\Delta^k, \derD(R)^\heart)$ is compact $1$-projectively generated. A set of generators is given by the set $$\{\ins_i(R) \mid i \in \{0, \dots, k\}\}$$
\end{lemma}
\begin{proof}
By Lemma \ref{lem_fun_proj_gen} it suffices to give a proof for $k = 0$. Then this is a classical result about the category of discrete $R$-modules, we give a short sketch. Since any $R$-module can be written as a colimit of free modules, it suffices to show the functor $\Hom_{\derD(R)^\heart}(R, -) \colon \derD(R)^\heart \to \Set$ commutes with filtered colimits and geometric realizations.

Note that $\Hom_{\derD(R)^\heart}(R, -)$ is just the forgetful functor $\derD(R)^\heart \to \Set$, hence it commutes with filtered colimits (for example by \cite[Proposition 2.13.5]{alma}). By \cite[Remark A.21]{mao}, to show it commutes with geometric realizations it suffices to show it commutes with colimits over $\Delta^\op_{\leq 1}$, which we leave for the reader to verify.
\end{proof}
\begin{corollary}\label{gens_animod}
Let $R$ be a ring and $k \in \NN$. The $\infty$-category $\Fun(\Delta^k, \derD(R)_{\geq 0})$ is compact projectively generated. Moreover, a set of generators is given by $\{\ins_i(R) \mid i \in \{0, \dots, k\}\}$.
\end{corollary}
\begin{proof}
By \cite[Corollary 7.1.4.15]{ha} and \cite[Theorem 7.1.2.13]{ha} we have a canonical equivalence $\Ani(\derD(R)^\heart) \cong \derD(R)_{\geq 0}$. By Corollary \ref{corol_ani_commutes_fun}, it thus suffices to show $\Ani(\Fun(\Delta^k, \derD(R)^\heart))$ is compact projectively generated by the mentioned set of generators. This follows by combining Lemma \ref{gens_mod} and Lemma \ref{ani_gen}.
\end{proof}

\begin{lemma}\label{ani_adjoint_symmod}
Let $R$ be a ring. There exists an adjunction 
\[
\begin{tikzcd}
\Fun(\Delta^k, \derD(R)_{\geq 0})
\arrow[r, "\Sym_R", shift left=1.5] &
\Fun(\Delta^k, \CAlg_R^\an)
\arrow[l, "\mathrm{forget}"]
\end{tikzcd}
\]
Moreover, $\mathrm{forget}$ is conservative, $\mathrm{forget}$ preserves sifted colimits, and the canonical map $\pi_0 \circ \mathrm{forget} \to \mathrm{forget} \circ \pi_0$ is an equivalence. 
\end{lemma}
\begin{proof}
We have an adjunction
\[
\begin{tikzcd}
\derD(R)^\heart
\arrow[r, "\Sym_R", shift left=1.5] &
\Alg_R
\arrow[l, "\mathrm{forget}"]
\end{tikzcd}
\]
Since $\Alg_R$ is cocomplete, the forgetful functor commutes with filtered colimits and geometric realizations, $\derD(R)^\heart$ is $1$-projectively generated (by Lemma \ref{gens_mod}) and $\mathrm{forget}$ is conservative, by \cite[Corollary 2.3]{mao} there exists an adjunction
\[
\begin{tikzcd}
\derD(R)_{\geq 0}
\arrow[r, "\Sym_R", shift left=1.5] &
\CAlg_R^\an
\arrow[l, "\mathrm{forget}"]
\end{tikzcd}
\]
for which $\mathrm{forget}$ is conservative, $\mathrm{forget}$ preserves filtered colimits and geometric realizations, and the canonical map $\pi_0 \circ \mathrm{forget} \to \mathrm{forget} \circ \pi_0$ is an equivalence. By Lemma \ref{fun_adjoint} we obtain the desired adjunction (use \cite[Corollary 5.1.2.3]{htt} to show that the induced map $\mathrm{forget}$ again preserves sifted colimits). 
\end{proof}
\begin{corollary}\label{corol_comp_proj_gen_ani_ring} \label{corol_comp_proj_gen_ani_algr}
Let $R$ be a ring and $k \in \NN$. The $\infty$-category $\Fun(\Delta^k, \CAlg_R^\an)$ is compact projectively generated. A set of compact projective generators is given by $\{\ins_i(R[x]) \mid i \in \{0, \dots, k\}\}$. 
\end{corollary}
\begin{proof}
We verify the conditions of Proposition \ref{prop_generators_from_adjunction} for the adjunction given by Lemma \ref{ani_adjoint_symmod} Note that $\CAlg_R^\an$ is cocomplete by definition, hence $\Fun(\Delta^k, \CAlg_R^\an)$ is cocomplete by \cite[Corollary 5.1.2.3]{htt}. Moreover, $\mathrm{forget}$ preserves filtered colimits and geometric realizations by Lemma \ref{ani_adjoint_symmod}, so condition 1 holds. By Corollary \ref{gens_animod} condition 2 holds, and by Lemma \ref{ani_adjoint_symmod} condition 3 holds. Thus all conditions are satisfied and we may conclude by applying Proposition \ref{prop_generators_from_adjunction}.
\end{proof}
\begin{definition}\label{def_anipair}
Let $R$ be a ring and $k \in \NN$.  Let $\mathcal{C} \in \{\derD(R)^\heart, \Alg_R\}$. Define $\Fun(\Delta^k, \mathcal{C})_\surj \subseteq \Fun(\Delta^k, \mathcal{C})$ to be the full subcategory of all objects 
\[
(X_0 \to \dots \to X_k) \in \Fun(\Delta^k, \mathcal{C})
\]
for which the composition $X_0 \to X_i$ is a surjective map in $\mathcal{C}$ for all $i \in \{0, \dots, k\}$.  We set 
\[
\Fun(\Delta^k, \Ani(\mathcal{C}))_\surj := \Fun(\Delta^k, \Ani(\mathcal{C})) \times_{\Fun(\Delta^k, \mathcal{C})} \Fun(\Delta^k, \mathcal{C})_\surj
\]
Following \cite{mao}, we write $\AniPair_R := \Fun(\Delta^1, \Ani(\Alg_R))_\surj$. 
\end{definition}
Thus, an object of $\Fun(\Delta^k, \Ani(\mathcal{C}))_\surj$ is specified by a diagram
\[
X_0 \to X_1 \to \dots \to X_k
\]
of objects in $\Ani(\mathcal{C})$ such that $\pi_0(X_0) \to \pi_0(X_i)$ is surjective for all $i$. In particular, an element of $\AniPair_R$ is a morphism of animated rings $A \to B$ such that $\pi_0(A) \to \pi_0(B)$ is a surjective ring map.

We now want to find a set of compact projective generators for the $\infty$-category $\Fun(\Delta^k, \CAlg_R^\an)_\surj$. 
\begin{definition}
Let $\mathcal{C}$ be an $n$-category, and let $k \in \NN$ and $i \in \{0, \dots, k\}$. Define the functor 
\begin{align*}
\coins_i \colon \Fun(\Delta^1, \mathcal{C}) &\to \Fun(\Delta^k, \mathcal{C}) \\
(X \to Y) &\mapsto \(\underbrace{X \to \dots \to X}_{i + 1 \text{ times}} \to  \underbrace{Y \to \dots \to Y}_{(k - i - 1) \text{ times}}\)
\end{align*}
\end{definition}
\begin{lemma}\label{lem_comp_proj_gen_ab_surj}
Let $R$ be a ring and $k \in \NN$. The $\infty$-category $\Fun(\Delta^k, \derD(R)_{\geq 0})_\surj$ is compact projectively generated. Moreover, a set of compact projective generators is given by 
\begin{align*}
\{\coins_i(R \to 0) \mid i \in \{0, \dots, k\} \} 
\end{align*}
\end{lemma}
\begin{proof}
One may construct an equivalence of categories
\[
\begin{tikzcd}[row sep=tiny]
\Fun(\Delta^k, \derD(R)_{\geq 0})
\arrow[r, "\fib_k", shift left=1.5] &
\Fun(\Delta^k, \derD(R)_{\geq 0})_\surj
\arrow[l, "\mathrm{cofib}_k"]
\end{tikzcd}
\]
sending
\begin{align*}
\(X_0 \to X_1 \to \dots \to X_k\) \mapsto \(X_k \to (X_k / X_0) \to \dots (X_k / X_{k - 1})\) \\
\(\fib(Y_0 \to Y_1) \to \dots \to \fib(Y_0 \to Y_k)  \to Y_0\) \mapsfrom \(Y_0 \to \dots \to Y_k\)
\end{align*}
It follows that $\{\fib_k(\ins_i(0 \to R)) \mid i \in \{0, \dots, k\}\}$ is a set of compact projective generators for the category $\Fun(\Delta^k, \derD(R)_{\geq 0})_\surj$. The result follows as $\fib_k(\ins_i(0 \to R)) = \coins_i(R \to 0)$.  
\end{proof}
\begin{corollary}\label{corol_gen_surj}
Let $R$ be a ring, $k \in \NN$. The $\infty$-category $\Fun(\Delta^k, \CAlg_R^\an)_\surj$ is compact projectively generated. Moreover, a set of compact projective generators is given by 
\begin{align*}
\{\coins_i(R[x] \to R) \mid i \in \{0, \dots, k\} \} 
\end{align*}
\end{corollary}
\begin{proof}
Restricting the adjunction in Lemma \ref{ani_adjoint_symmod}, we obtain an adjunction
\[
\begin{tikzcd}
\Fun(\Delta^k, \derD(R)_{\geq 0})_\surj
\arrow[r, "\Sym_R", shift left=1.5] &
\Fun(\Delta^k, \CAlg_R^\an)_\surj
\arrow[l, "\mathrm{forget}"]
\end{tikzcd}
\]
where $\mathrm{forget}$ is conservative and preserves filtered colimits and geometric realizations. Moreover, the subcategory $\Fun(\Delta^k, \CAlg_R^\an)_\surj \subseteq \Fun(\Delta^k, \CAlg_R^\an)$ is closed under colimits (since $\pi_0$ preserves colimits), hence $\Fun(\Delta^k, \CAlg_R^\an)_\surj$ is cocomplete. Finally  $\Fun(\Delta^k, \derD(R)_{\geq 0})_\surj$ is compact projectively generated by Lemma \ref{lem_comp_proj_gen_ab_surj}. The result follows by applying Proposition \ref{prop_generators_from_adjunction}. 
\end{proof}
\begin{definition}\label{def_gen_surj}
For $R$ be a ring, we let $\Fun(\Delta^k, \Poly_R)_{\surj, \gen}$ be spanned by coproducts of objects in the set 
\[
\{\coins_i(R[x] \to R) \mid i \in \{0, \dots, k\}\}
\]
as a full subcategory of the $1$-category $\Fun(\Delta^k, \Alg_R)_\surj$.
\end{definition}
\begin{lemma}\label{forget_colim}
Let $R$ be a ring. The natural map
\[
\Fun(\Delta^k, \CAlg_R^\an)_\surj \to \Fun(\Delta^k, \CAlg_R^\an)
\]
commutes with colimits. 
\end{lemma}
\begin{proof}
Since the left hand side is a full subcategory of the right hand side, it suffices to show it is closed under colimits. Since the functor $$\pi_0 \colon \Fun(\Delta^k, \CAlg_R^\an) \to \Fun(\Delta^k, \Alg_R)$$ preserves colimits, it suffices to show $$\Fun(\Delta^k, \Alg_R)_{\surj} \subseteq \Fun(\Delta^k, \Alg_R)$$ is closed under colimits. 

To see that it is closed under sifted colimits, note that the natural map $\Alg_R \to \derD(R)^\heart$ commutes with sifted colimits, hence it suffices to show that $$\Fun(\Delta^k, \derD(R)^\heart)_{\surj} \subseteq \Fun(\Delta^k, \derD(R)^\heart)$$ is closed under sifted colimits. This follows since it is closed under all colimits, as the cofiber functor is a colimit, and hence commutes with colimits. 

Thus remains to show $\Fun(\Delta^k, \Alg_R)_{\surj} \subseteq \Fun(\Delta^k, \Alg_R)$ is closed under coproducts. To this end, we need to show that if $A' \to A$ and $B' \to B$ are surjective maps, then the map $A' \otimes_R B' \to A \otimes_R B$ is surjective. This follows since the tensor product is right exact by \stacksref{00DF}.
\end{proof}

\subsection{Homological algebra in stable $\infty$-categories}

\begin{construction}
Let $\mathcal{C}$ be a stable $\infty$-category. Let $A, B \in \mathcal{C}$. Then the fiber functor \cite[Definition 1.1.1.6, Remark 1.1.1.7]{ha} 
\[
\fib \colon \Fun(\Delta^1, \mathcal{C}) \to \mathcal{C} 
\]
sends 
\[
(A \xrightarrow{0} B[1]) \mapsto A \oplus B
\]
We thus get a map
\[
\fib \colon \pi_1(\Map_\mathcal{C}(A, B[1]), 0) \to \pi_0 \Map_{\mathcal{C}}(A, B)
\]
We will denote with $\theta_{A, B}$ the composition
\[
\pi_1(\Map_\mathcal{C}(A, B[1]), 0) \xrightarrow{\fib} \pi_0 \Map_{\mathcal{C}}(A \oplus B, A \oplus B) \xrightarrow{\pi_B \circ - \circ \iota_A} \pi_0 \Map_{\mathcal{C}}(A \oplus B)
\]
where $\iota_A \colon A \to A \oplus B$ and $\pi_B \colon A \oplus B \to B$ are the canonical inclusion and projection maps.
\begin{lemma}\label{bce}
Let $\mathcal{C}$ be a stable $\infty$-category, and let $A, B \in \mathcal{C}$. Then the map
\[
\pi_1(\Map_\mathcal{C}(A, B[1]), 0) \to \pi_0 \Map_{\mathcal{C}}(A \oplus B)
\]
induced by the equivalences
\[
\Omega \Map_{\mathcal{C}}(A, B[1]), 0)  \simeq  \Map_{\mathcal{C}}(A, \Omega B[1])) \simeq \Map_{\mathcal{C}}(A \oplus B)
\]
agrees with the map $\theta_{A, B}$. 
\end{lemma}
\begin{proof}
Denote with $\mathcal{D}$ the $\infty$-category $\Fun(\Delta^1, \mathcal{C})$. Let $x \in \mathcal{D}$ be the element corresponding to the map $0 \colon A \to B[1]$ in $\mathcal{C}$. The natural morphism of simplicial sets $\Delta^1 \to S^1$ induces a canonical map 
\[
s \colon \Omega \Map_{\mathcal{C}}(A, B[1]) \to \Map_{\mathcal{D}}(x, x)
\]
 Let $a \in \mathcal{D}$ be the element $(A \to 0)$, and denote with $b \in \mathcal{D}$ the element $(0 \to B[1])$. Denote with $f \in \Map_{\mathcal{D}}(a, x)$ the element corresponding to the commutative square
\[
\begin{tikzcd}
A \dar{\id} \rar &  0 \dar \\
A \rar{0} & B[1]
\end{tikzcd}
\]
and with $g \in \Map_{\mathcal{D}}(x, b)$ the element corresponding to the commutative square
\[
\begin{tikzcd}
A \dar \rar{0} & B[1] \dar{\id} \\
0 \rar& B[1]
\end{tikzcd}
\]
We then get a natural map $\Map_{\mathcal{D}}(x, x) \to \Map_{\mathcal{D}}(a, b)$, given by precomposing with $f$ and postcomposing with $g$. The universal property of the pullback square
\[
\begin{tikzcd}
B \dar \rar & 0 \dar \\
0 \rar& B[1]
\end{tikzcd}
\]
in $\mathcal{C}$ induces a canonical equivalence $\Map_{\mathcal{D}}(a, b) \simeq \Map_{\mathcal{C}}(A, B)$. It is not too hard to see that the composition
\[
\pi_1 \Map_{\mathcal{C}}(A, B[1]) \xrightarrow{s} \pi_0 \Map_{\mathcal{D}}(x, x) \xrightarrow{g \circ - \circ f} \pi_0  \Map_{\mathcal{D}}(a, b) \simeq \Map_{\mathcal{C}}(A, B)
\]
is equal to the composition 
\[
\pi_1 \Map_{\mathcal{C}}(A, B[1]), 0)  \simeq \pi_0  \Map_{\mathcal{C}}(A, \Omega B[1])) \simeq \pi_0  \Map_{\mathcal{C}}(A, B)
\]
Since the fiber functor $\fib \colon \mathcal{D} \to \mathcal{C}$ sends $f \mapsto \iota_A$ and $g \mapsto \pi_B$, the result follows.
\end{proof}
\end{construction}
\subsection{Homological algebra in symmetric monoidal stable $\infty$-categories}
In this section we record some results on tensor products of fiber sequences in stable $\infty$-categories. These results are well-known in the triangulated setting, see for example \cite{traces}.
\begin{lemma}\label{lem_fib_stable_infty}
Let $\mathcal{C}^\otimes$ be a symmetric monoidal stable $\infty$-category for which the tensor product preserves finite limits in each variable. Let  
\begin{align*}
A_1 \to A_2 \to A_3 \\
B_1 \to B_2 \to B_3
\end{align*}
be fiber sequences in $\mathcal{C}$. Write $E_{ij} := A_i \otimes B_j$. Then the canonical map
\[
E_{11} \to \fib(E_{22} \to (E_{23} \times_{E_{33}} E_{32}))
\]
is an equivalence. 
\end{lemma}
\begin{proof}
Consider the diagram
\[
\begin{tikzcd}
\dar E_{22} \rar& E_{23} \dar \\
E_{32} \rar& E_{33} \\
\uar E_{32} \rar& E_{32} \uar
\end{tikzcd}
\]
Since limits commute with limits, we may identify the fiber of the vertical pullbacks with the vertical pullback of the fibers of the horzontal arrows. As the first is equal to $E_{11}$ and the second to $\fib(E_22 \to E_{23} \times_{E_{33}} E_{32})$, the result follows.
\end{proof}
Similarly, we have the dual statement.
\begin{lemma}\label{lem_fib_stable_infty_dual}
Let $\mathcal{C}^\otimes$ be a symmetric monoidal stable $\infty$-category, and let 
\begin{align*}
A_1 \to A_2 \to A_3 \\
B_1 \to B_2 \to B_3
\end{align*}
be fiber sequences in $\mathcal{C}$. Write $E_{ij} := A_i \otimes B_j$. Then the canonical map
\[
\cofib(E_{12} \cup_{E_{11}} E_{21} \to E_{22}) \to E_{11}
\]
is an equivalence. 
\end{lemma}
\begin{proof}
Dual to Lemma \ref{lem_fib_stable_infty}. 
\end{proof}
The following result essentially summarizes all of the homological algebra constructions in \cite{traces}. 
\begin{lemma}\label{lem_complicated_square}
Let $\mathcal{C}^\otimes$ be a symmetric monoidal stable $\infty$-category for which the tensor product preserves finite limits in each variable. Let 
\begin{align*}
A_1 \to A_2 \to A_3 \\
B_1 \to B_2 \to B_3
\end{align*}
be fiber sequences in $\mathcal{C}$. Write $E_{ij} := A_i \otimes B_j$. Then there exists a commutative diagram
\[
\begin{tikzcd}
\dar E_{11} \rar& E_{21} \dar \\
\dar E_{12} \rar&  \dar E_{12} \cup_{E_{11}} E_{21} \rar& E_{22} \dar  \\
\dar E_{13} \rar& \dar E_{13} \oplus E_{31} \rar& E_{23} \times_{E_{33}} E_{32} \dar \rar& E_{23} \dar  \\
0 \rar& E_{31} \rar& E_{32} \rar& E_{33}
\end{tikzcd}
\]
in which all squares are pullback (thus pushout) squares, and all maps are the canonical ones. 
\end{lemma}
\begin{proof}
Taking the coproduct of the squares
\[
\begin{tikzcd} 
\dar E_{11} \rar&\dar E_{21} & 0 \rar\dar& 0\dar \\
0 \rar& E_{31} & E_{13} \rar& E_{31}
\end{tikzcd}
\]
we obtain the pullback (thus pushout)  square
\[
\begin{tikzcd} 
E_{11} \rar \dar & \dar E_{21} \\
E_{13} \rar& E_{13} \oplus E_{31} 
\end{tikzcd}
\]
and thus by functoriality of pushouts a commutative diagram
\[
\begin{tikzcd}
\dar E_{11} \rar& E_{21} \dar \\
\dar E_{12} \rar&  \dar E_{12} \cup_{E_{11}} E_{21}  \\
E_{13} \rar& E_{13} \oplus E_{31} 
\end{tikzcd}
\]
By Lemma \ref{lem_fib_stable_infty} there exists a pullback (thus pushout)  square
\[
\begin{tikzcd}
\dar E_{12} \rar& E_{22} \dar \\
E_{13} \rar& E_{23} \times_{E_{33}} E_{32}
\end{tikzcd}
\]
hence again by functoriality of pushouts we get an induced diagram
\[
\begin{tikzcd}
\dar E_{11} \rar& E_{21} \dar \\
\dar E_{12} \rar&  \dar E_{12} \cup_{E_{11}} E_{21} \rar& E_{22} \dar  \\
E_{13} \rar& E_{13} \oplus E_{31} \rar& E_{23} \times_{E_{33}} E_{32}
\end{tikzcd}
\]
The two bottom-right pullback (thus pushout)  squares of the final diagram are constructed dually to the construction of the two upper-right pullback (thus pushout)  squares, the final square is then obvious. 
\end{proof}
\begin{lemma}\label{lem_homotopy_pushout_123}
Let $\mathcal{C}^\otimes$ be a symmetric monoidal stable $\infty$-category for which the tensor product preserves finite limits in each variable. Let  
\begin{align*}
A_1 \to A_2 \to A_3 \\
B_1 \to B_2 \to B_3
\end{align*}
be fiber sequences in $\mathcal{C}$. Write $E_{ij} := A_i \otimes B_j$. Suppose $T \in \mathcal{C}$ is any object, and we are given a diagram
\[
\begin{tikzcd}
\dar E_{11} \rar& E_{12} \dar \\
\dar E_{21} \rar& E_{22} \dar \\
E_{31} \rar& T
\end{tikzcd}
\]
in $\mathcal{C}$. Then there exists an extension to a commutative diagram
\[
\begin{tikzcd}
E_{21} \dar \rar & E_{22}\dar \arrow[bend left = 30]{ddr}\\
E_{31} \rar \arrow[bend right = 30]{drr} & E_{31} \times_{E_{33}} E_{32} \arrow{dr} \\
& & T
\end{tikzcd}
\]
in $\mathcal{C}$. 
\end{lemma}
\begin{proof}
Consider the diagram
\[
\begin{tikzcd}
\dar E_{11} \rar& E_{21} \dar \rar & E_{22}\dar \arrow[bend left = 90]{dd}\\
0 \rar& E_{31} \arrow{dr} \rar & E_{31} \times_{E_{33}} E_{32} \\
& & T
\end{tikzcd}
\]
Clearly the left square is a pullback (thus pushout)  square, and the rectangle is a pullback (thus pushout)  square by Lemma \ref{lem_fib_stable_infty}. Thus the right square is a pushout square, which yields the result.
\end{proof}
\subsection{Notes on trace maps}
\begin{definition}[Dualizable object]
Let $\mathcal{C}^\otimes$ be a symmetric monoidal $\infty$-category with unit object $\catO$, and let $\mathcal{E} \in \mathcal{C}$. We say that $\mathcal{E}$ is \emph{dualizable} if there exists an object $\mathcal{E}^\vee \in \mathcal{C}$ and maps
\begin{align*}
\ev \colon \mathcal{E} \otimes \mathcal{E}^\vee  \to \catO \\
\coev \colon \catO \to \mathcal{E}^\vee \otimes \mathcal{E}
\end{align*}
such that the compositions
\begin{align*}
\mathcal{E}^\vee \simeq \catO \otimes \mathcal{E}^\vee  \xrightarrow{\coev \otimes \id} \mathcal{E}^\vee \otimes \mathcal{E} \otimes \mathcal{E}^\vee \xrightarrow{\id \otimes \ev} \mathcal{E}^\vee \otimes \catO \simeq \mathcal{E}^\vee  \\
\mathcal{E} \simeq \mathcal{E} \otimes \catO \xrightarrow{\id \otimes \coev} \mathcal{E} \otimes \mathcal{E}^\vee \otimes \mathcal{E} \xrightarrow{\ev \otimes \id} \catO \otimes \mathcal{E}\simeq \mathcal{E} 
\end{align*}
are homotopic to the identity. A dualizable object is said to be \emph{invertible} if the evaluation and coevaulation maps are isomorphisms.
\end{definition}
For any dualizable object $\mathcal{E}$ and any map $\mathcal{E} \otimes M \to \mathcal{E} \otimes N$ we may consider its trace $M \to N$, see Definition \ref{def_trace_map}. We give some general properties of this trace map.
\begin{lemma}\label{lem_atiyah_threea}
Let $\mathcal{C}^\otimes$ be a symmetric monoidal $\infty$-category with unit object $\catO$, and let $\mathcal{L}$ be an invertible object of $\mathcal{C}$. For any three objects $M, N, K \in \mathcal{C}$ and any two maps
\begin{align*}
\alpha \colon \mathcal{L} \otimes_{\catO} M \to \mathcal{L} \otimes_{\catO} N \\
\beta \colon \mathcal{L} \otimes_{\catO} N \to \mathcal{L} \otimes_{\catO} K 
\end{align*}
we have $\tr(\beta \circ \alpha) = \tr(\beta) \circ \tr(\alpha)$. 
\end{lemma}
\begin{proof}
Since $\mathcal{L}$ is invertible, the composition
\[
\mathcal{L}^\vee \otimes \mathcal{L} \xrightarrow{\mathrm{ev}} \catO \xrightarrow{\mathrm{coev}} \mathcal{L}^\vee \otimes \mathcal{L}
\]
is the identity map. Thus
\begin{align*}
\tr&(\beta) \circ \tr(\alpha) \\
&= (\ev \otimes K) \circ (\mathcal{L}^\vee \otimes \beta) \circ (\coev \otimes N) \circ (\ev \otimes N) \circ (\mathcal{L}^\vee \otimes \alpha) \circ (\coev \otimes M) \\
 &= (\ev \otimes K) \circ (\mathcal{L}^\vee \otimes \beta) \circ (\mathcal{L}^\vee \otimes \alpha) \circ (\coev \otimes M) \\
  &= \tr(\beta \circ \alpha)
\end{align*}
as required.
\end{proof}
\begin{prop}[Additivity of traces]\label{prop_add_traces}
Let $\mathcal{C}^\otimes$ be a symmetric monoidal stable $\infty$-category. Suppose 
\[
\begin{tikzcd}
X \rar& Y \rar& Z \rar{+1} &  \
\end{tikzcd}
\]
is a fiber sequence in $\mathcal{C}$ of dualizable objects. Given a commutative diagram
\[
\begin{tikzcd}
M \otimes X \rar \dar{f} & \dar{g} M \otimes Y \rar& M \otimes Z \rar{+1}\dar{h} & \ \\
N \otimes X \rar& N \otimes Y \rar& N \otimes Z \rar{+1} &\ 
\end{tikzcd}
\]
in which the lower and upper fiber sequence are obtained by tensoring the original fiber sequence, one has
\[
\tr_X(f) + \tr_Z(h) = \tr_Y(g)
\]
in $\pi_0 \Hom_\mathcal{C}(M, N)$. 
\end{prop}
\begin{proof}
A modern formulation of \cite{traces}. Write
\begin{align*}
V &:= (X \otimes Y^\vee) \cup_{(X \otimes Z^\vee)} Y \otimes Z^\vee \\
W &:= (Z \otimes Y^\vee) \times_{(Z \otimes X^\vee)} X \otimes Z^\vee
\end{align*}
Some calculations with adjoints yields a commutative diagram
\[
\begin{tikzcd}
\dar X \otimes Z^\vee \rar& \dar X \otimes Y^\vee \rar& X \otimes X^\vee \arrow{dd}{\mathrm{ev}} \\
\dar \rar Y \otimes Z^\vee & Y \otimes Y^\vee \arrow{dr}{\mathrm{ev}} \\
Z \otimes Z^\vee \arrow{rr}{\mathrm{ev}} & & \mathbbm{1}
\end{tikzcd}
\]
so by applying Lemma \ref{lem_homotopy_pushout_123} and its symmetric twin we obtain a commutative diagram
\[
\begin{tikzcd}
Y \otimes Y^\vee \arrow{dr} & X \otimes X^\vee \arrow[bend left = 30]{ddr}{\mathrm{ev}} \dar \\
Z \otimes Z^\vee \arrow[bend right = 30]{drr}{\mathrm{ev}} \rar&  W \arrow{dr} \\
& & \mathbbm{1}
\end{tikzcd}
\]
such that the composition $Y \otimes Y^\vee \to W \to \mathbbm{1}$ is homotopic to the evaluation map. Tensoring this diagram with $N$ and the dual of this diagram (involving coevaluations) with $M$,  using the construction of the middle square from Lemma \ref{lem_complicated_square} and functoriality of pushouts (twice), we obtain a commutative diagram
\[
\begin{tikzcd}
& M \arrow{dl}[swap]{(M \otimes \mathrm{coev}, M \otimes \mathrm{coev})} \dar \arrow{dr}{M \otimes \mathrm{coev}} & \\
(M \otimes X \otimes X^\vee) \oplus (M \otimes  Z \otimes Z^\vee) \arrow{dd}{(f \otimes X^\vee, h \otimes Z^\vee)} \arrow{dr} & \lar M \otimes V \rar &  M \otimes (Y \otimes Y^\vee) \arrow{dd}{g \otimes Y^\vee} \arrow{dl} \\
 & M \otimes W \dar &   \\
\arrow{dr}[swap]{(N \otimes \mathrm{ev}, N \otimes \mathrm{ev})} (N \otimes X \otimes X^\vee) \oplus (M \otimes  Z \otimes Z^\vee) \rar & \dar N \otimes W & N \otimes (Y \otimes Y^\vee) \arrow{dl}{N \otimes \mathrm{ev}} \lar \\
& N & 
\end{tikzcd}
\]
The result follows by comparing the outer compositions. 
\end{proof}

\emergencystretch 1em
\printbibliography
\end{document}